\documentclass[11pt]{amsart}
\usepackage{kloosterman-revised10}

\title{Hodge theory of Kloosterman connections}

\author[J. Fres\'an]{Javier Fres\'an}
\address[J. Fres\'an]{CMLS, \'Ecole polytechnique, F--91128 Palaiseau cedex, France}
\email{javier.fresan@polytechnique.edu}
\urladdr{http://javier.fresan.perso.math.cnrs.fr}

\author[C.~Sabbah]{Claude Sabbah}
\address[C.~Sabbah]{CMLS, CNRS, École polytechnique, Institut Polytechnique de Paris\\
F--91128 Palaiseau cedex\\
France}
\email{Claude.Sabbah@polytechnique.edu}
\urladdr{http://www.math.polytechnique.fr/perso/sabbah}

\author[J.-D. Yu]{Jeng-Daw Yu}
\address[J.-D. Yu]{Department of Mathematics, National Taiwan University,
	Taipei 10617, Taiwan}
\email{jdyu@ntu.edu.tw}
\urladdr{http://homepage.ntu.edu.tw/~jdyu/}

\thanks{This research was partly supported by the MoST/CNRS grants 106-2911-I-002-539/TWN PRC 1632. The research of C.\,S.\ was also partly supported by the grant ANR-16-CE40-0011 of Agence nationale de la recherche. In the final stage of the preparation of this paper, J.\,F. was also partially supported by the grant 346300 for IMPAN from the Simons Foundation and the matching 2015-2019 Polish MNiSW fund}

\begin{document}
\begin{abstract}
We construct motives over the rational numbers associated with symmetric power moments of Kloosterman sums, and prove that their $L$-functions extend meromorphically to the complex plane and satisfy a functional equation conjectured by Broadhurst and Roberts. Although the motives in question turn out to be ``classical'', we compute their Hodge numbers by means of the irregular Hodge filtration on their realizations as exponential mixed Hodge structures. We show that all Hodge numbers are either zero or one, which implies potential automorphy thanks to recent results of Patrikis and Taylor.
\end{abstract}
\maketitle
\vspace*{-\baselineskip}
{\let\\\relax\tableofcontents}

\mainmatter

\section{Introduction}

This paper is devoted to the study of a family of global $L$-functions built up by assembling symmetric power moments of Kloosterman sums over finite fields. We prove that they arise from potentially automorphic motives over the rational numbers, and hence admit a meromorphic extension to the complex plane that satisfies the expected functional equation. The exact shape of the latter was conjectured by Broadhurst and Roberts.

\subsection{\texorpdfstring{$L$}{L}-functions of symmetric power moments of Kloosterman sums}\label{sec:intro1}

Let $p$ be a prime number, $\FF_p$ the finite field with $p$ elements, and $\overline{\FF}_p$ an algebraic closure of $\FF_p$. If $q$ is a power of $p$, we denote by $\FF_q$ the subfield of $\overline{\FF}_p$ with~$q$ elements
and by $\tr_{\FF_q/\FF_p}\colon \FF_q \to \FF_p$ its trace map.
Let $\psi \colon \FF_p \to \CC^\times$ be a non-trivial additive character. For each $a \in \FF_q^\times$, the \emph{Kloosterman sum} is the real number
\begin{equation}\label{eq:Kl_sum}
\Kl_2(a; q)=\sum_{x \in \FF_q^\times}
\psi(\mathrm{tr}_{\FF_q \slash \FF_p}(x+\sfrac{a}{x})).
\end{equation}
As an application of the Riemann hypothesis for curves over finite fields \cite{Weil}, Weil proved that there exist algebraic integers~$\alpha_a, \beta_a$ of absolute value $\sqrt{q}$ satisfying $\Kl_2(a; q)=-(\alpha_a+\beta_a)$ and $\alpha_a \beta_a=q$. For each integer $k \geq 1$, we define $k$-th symmetric powers of Kloosterman sums
\[
\Kl_2^{\Sym^k}(a; q)=
\sum_{i=0}^k \alpha_a^i \beta_a^{k-i}\vspace*{-3pt}
\]
and, summing over all $a$, we form the \textit{moments}
\[
m_2^k(q)=\sum_{a \in \FF_q^\times} \Kl^{\Sym^k}_2(a; q).
\]
Note that this convention makes $\Kl_2^{\Sym^1}(a; q)$ opposite to $\Kl_2(a; q)$. Contrary to Kloosterman sums and their symmetric powers, the moments $m_2^k(q)$ are rational integers that do not depend on the choice of the additive character. We pack them into the generating series
\begin{equation*}\label{eq:firstcases}
Z_k(p; T)=\exp \biggl(\sum_{n=1}^\infty m_2^k(p^n) \frac{T^n}{n} \biggr),
\end{equation*} which in fact turns out to be a polynomial with integer coefficients. The first few cases are easy to compute: both $Z_1(p; T)$ and $Z_2(p; T)$ are equal to $1-T$, and the equalities
\begin{equation*}
Z_3(p; T)=(1-T)(1-\left(\tfrac{p}{3}\right)p^2 T), \qquad Z_4(p; T)=\begin{cases} 1-T & \text{ if } p=2, \\
(1-T)(1-p^2 T) & \text{ if }p>2, \end{cases}
\end{equation*} hold, where $(p/3)$ stands for the Legendre symbol. From this one may already infer that the polynomial $Z_k(p; T)$ is always divisible by $1-T$. Other so-called ``trivial factors'' appear when~$k$ is a multiple of $4$ or when $k$ is even and $p$ is small compared with $k$, see Section \ref{sec:properties-characteristic-polynomial} \textit{infra}. Better behaved than $Z_k(p; T)$ is the polynomial $M_k(p; T)$ obtained by removing these trivial factors, since all its roots then have the same absolute value $p^{-\psfrac{k+1}{2}}$.

We shall now build a global $L$-function over $\QQ$ with the above polynomials as local Euler factors. We first discuss the case of odd symmetric powers, say of the form $k=2m+1$.
Let~$S$ be the set of odd prime numbers smaller than or equal to $k$. For all $p \notin S$, define the local factor at $p$ as $L_{k}(p; s)=M_k(p; p^{-s})^{-1}$ and consider the Euler product
\begin{equation*}
L_k(s)=\prod_{p \notin S} L_{k}(p; s),
\end{equation*} which by the previous remark about the roots of $M_k(p; T)$ converges absolutely on the half\nobreakdash-plane \hbox{$\mathrm{Re}(s)>\sfrac{(k+3)}{2}$.} This function is expected to have meromorphic continuation to the entire complex plane and satisfy a functional equation relating its values at $s$ and~\hbox{$k+2-s$.} As usual, the functional equation has a neat formulation after completing the~$L$\nobreakdash-function by adding local factors at $p \in S$, as we shall do in \eqref{eq:bad-factors-odd-case} \textit{infra}, and gamma factors at infinity. We set $\cond_k=1_{\mathrm{s}} 3_{\mathrm{s}} 5_{\mathrm{s} } \cdots k_\mathrm{s}$,
where $n_\mathrm{s}$ denotes the square-free part of an integer $n$
(\ie the product of all prime numbers $p$ such that the $p$-adic valuation $v_p(n)$ is odd),
and we consider the complete $L$-function
\begin{equation*}
\Lambda_k(s)=\Bigl(\frac{\cond_k}{\pi^m} \Bigr)^{\sfrac{s}{2}} \prod_{j=1}^m \Gamma\Bigl(\frac{s-j}{2} \Bigr) \prod_{p \text{\ prime}} L_{k}(p; s).
\end{equation*}

\begin{thm}\label{thm:intro-k-odd}
Assume $k$ is odd. The function $L_k(s)$ admits a meromorphic continuation to the complex plane and satisfies the functional equation
\begin{equation*}\label{eq:functionalequationodd}
\Lambda_k(s)=\Lambda_k(k+2-s).
\end{equation*}
\end{thm}

A similar result holds for even symmetric powers, except that we were unable to make the local invariants explicit at $p=2$. To formulate the statement, we write either $k=2m+4$ or $k=2m+2$ with $m$ an even integer, and we define $S$ as the set of all prime numbers smaller than or equal to $k/2$. The local factors at odd primes in $S$ are described in \eqref{eq:local-factors-even} \textit{infra}. Besides, we set $\cond'_k=2_\mathrm{u} 4_\mathrm{u} 6_\mathrm{u} \cdots k_\mathrm{u}$, where $n_\mathrm{u}$ denotes the odd part of the radical of an integer~$n$ (\ie the product of all odd primes dividing~$n$). We then complete the $L$-function outside the prime $2$ as follows:
\[
\Lambda'_k(s)=\Bigl(\frac{\cond_k'}{\pi^m} \Bigr)^{\sfrac{s}{2}} \prod_{j=1}^m \Gamma\Bigl(\frac{s-j}{2} \Bigr) \prod_{p \neq 2} L_{k}(p; s).
\]

\begin{thm}\label{thm:intro-k-even} Assume $k$ is even. The function $L_k(s)$ meromorphically extends to the complex plane. Moreover, there exists a sign $\varepsilon_k\in\{\pm1\}$, an integer $r_k \geq 0$, and a reciprocal of a polynomial with rational coefficients $L_k(2; T)$ such that, setting $\Lambda_k(s)=2^{\sfrac{r_k s}{2}}L_k(2; 2^{-s})\Lambda_k'(s)$, the following functional equation holds:
\begin{equation*}\label{eq:functionalequationeven}
\Lambda_k(s)=\varepsilon_k \Lambda_k(k+2-s).
\end{equation*}
\end{thm}

The above formulas for $\Lambda_k(s)$ match the numerical observations made by Broadhurst and Roberts in
\cite[Eq.\,(128)]{Broadhurst16} and~\cite{Broadhurst17, Roberts17}, up to replacing their variable $s$ with $s-2$ due to a Tate twist. For even $k$, they also conjecture that the elusive invariants take the values \hbox{$r_k=\flr{\sfrac{k}{6}}$} and $\varepsilon_k=(-1)^{t_k}$, with $t_k$ given by the formula
\begin{displaymath}
t_k=\flr{\sfrac{k}{8}}+\sum_{p \equiv 1 \,(\mathrm{mod}\, 4)} \flr{\sfrac{k}{2p}}+\sum_{p \equiv 3 \,(\mathrm{mod}\, 4)} \flr{\sfrac{k}{4p}}+\delta_{8\ZZ}(k),
\end{displaymath}
where $\delta_{8\ZZ}$ is the characteristic function of multiples of $8$. We explain in \ref{subsec:epsilon} \textit{infra} how the last three terms above fit with the local computations at odd primes in $S$ and at infinity.

Theorems \ref{thm:intro-k-odd} and \ref{thm:intro-k-even} were previously only known for $k \leq 8$. The first four cases are straightforward, as the $L$-function is trivial for $k=1, 2, 4$ and agrees, for $k=3$, with the shifted Dirichlet $L$-function $L(\chi_{3}, s-2)$ of the non-trivial quadratic character modulo $3$. In the next four cases, the $L$-function can be expressed in terms of a Hecke cusp form for some congruence subgroup~$\Gamma_0(N)$ of~$\mathrm{SL}_2(\ZZ)$,
as indicated in Table \ref{table:modularity} below.

\begin{table}[h]
\renewcommand{\arraystretch}{1.3}
\begin{center}
\begin{tabular}{|c|c|c|c|c|}
\hline
$k$ & $L_k(s)$ & modular form & references \\
\hline
$5$ & $L(f_3, s-2) $ & $f_3 \in S_3(\Gamma_0(15), (\sfrac{\cdot}{15}))$ & Livn\'e \cite{Livne}, Peters \textit{et al.} \cite{PTV} \\
\hline
$6$ & $L(f_4, s-2) $ & $f_4 \in S_4(\Gamma_0(6))$ & Hulek \textit{et al.} \cite{HSVV} \\
\hline
\multirow{3}{*}{$7$} & \multirow{3}{*}{$L(\mathrm{Ad}(g), s-2) $}
& $g \in S_3(\Gamma_0(125), (\sfrac{\cdot}{21})\chi_5)$,
& \multirow{2}{*}{conjectured by Evans \cite{Evans};} \\
& & $\chi_5$\text{ Dirichlet character } & \multirow{2}{*}{proved by Yun \cite{Yun15}}
\\
& & modulo $5$ with $\chi_5(2)=-i$ &\\
\hline
\multirow{2}{*}{$8$} & \multirow{2}{*}{$L(f_6, s-2)$} & \multirow{2}{*}{$f_6 \in S_6(\Gamma_0(6))$} & conjectured by Evans \cite{Evans2}; \\
& & & proved by Vincent and Yun \cite{Yun15}
\\
\hline
\end{tabular}
\bigskip
\caption{Modularity for $k=5,6,7, 8$}
\label{table:modularity}
\vspace*{-\baselineskip}
\end{center}
\end{table}

\subsection{Cohomological interpretation}

After Deligne \cite{Deligne77SGA},
Kloosterman sums arise as traces of Frobenius acting on an \'etale local system on the torus $\Gmf$. Let $\ell$ be a prime number distinct from~$p$ and let $\ol \QQ_\ell$ be an algebraic closure of the field of $\ell$-adic numbers. Once we view the character~$\psi$ as taking values in $\ol \QQ_\ell$ by choosing a primitive $p$-th root of unity in this field,
there is a rank one $\ell$-adic local system $\AS_\psi$ on the affine line $\Afu_{\FF_p}$ with trace function $z \mapsto \psi(\tr_{\FF_q/ \FF_p}(z))$, the so-called \hbox{Artin-Schreier} sheaf. The \textit{Kloosterman sheaf} $\Kl_2$ is then defined by pulling back and pushing out $\AS_\psi$ through the~diagram
\begin{equation}\label{eq:diagram-Kloosterman}
\begin{array}{c}
\xymatrix@=.5cm{
&\Gmf^{2}\ar[dr]^{f}\ar[dl]_{\pi}&\\
\Gmf &&\Afu_{\FF_p}
}
\end{array}
\end{equation}
where, if $x$ and $z$ are coordinates on $\Gmf^2$, the function $f$ is given by $x+z/x$ and $\pi$ stands for the projection to the $z$-coordinate. That is, we set
\begin{equation*}
\Kl_2=\Rder\pi_!f^\ast \AS_\psi[1].
\end{equation*}

Deligne showed that the object $\Kl_2$ is concentrated in degree zero, and that it is a rank-two lisse sheaf on $\Gmf$ tamely ramified at~zero, wildly ramified at infinity, and pure of weight one. Indeed, the ``forget supports'' map $\Rder\pi_!f^\ast \AS_\psi \to \Rder\pi_\ast f^\ast \AS_\psi$ is an isomorphism. Grothendieck's trace formula and base change
yield the equalities
\[
\Kl_2(a, q)=-\tr(\mathrm{Frob}_a | \Kl_2)=-(\alpha_a+\beta_a),
\]
where $\alpha_a$ and $\beta_a$ are the eigenvalues of Frobenius acting on a geometric fiber of $\Kl_2$ above~$a$. In the same vein, symmetric powers of Kloosterman sums are local traces of Frobenius on the symmetric powers $\Sym^k \Kl_2$. To obtain the moments, we consider the action of geometric Frobenius $F_p$ on the \'etale cohomology with compact support of $\Sym^k \Kl_2$. Since it is concentrated in degree one, invoking the trace formula again we get the equality
\begin{equation}\label{eq:characteristicpolynomial}
Z_k(p; T)=\det \Bigl( 1-F_pT \mid \coH^1_{\et, \cp}(\Gmolf, \Sym^k \Kl_2)\Bigr).
\end{equation}

It follows that $Z_k(p; T)$ is a polynomial with integer coefficients of degree minus the Euler characteristic of the sheaf $\Sym^k \Kl_2$ which, by the Grothendieck-Ogg-Shafarevich formula, is equal to its Swan conductor at infinity. Fu and Wan computed it for odd primes $p$ in \cite{FW05}, completing partial results by Robba \cite{Robba}:
\begin{equation*}
\deg Z_k(p; T)=\mathrm{Sw}_\infty(\Sym^k \Kl_2)=\begin{cases}
\dfrac{k+1}{2}-\flr{\dfrac{k}{2p}+\dfrac{1}{2}} & \text{ if } k \text{ odd,}\\[10pt] \dfrac{k}{2}-\flr{\dfrac{k}{2p}} & \text{ if } k \text{ even.}
\end{cases}
\end{equation*}
The remaining case $p=2$ was treated by Yun,
who proved that the Swan conductor is equal to $\sfrac{(k+1)}{2}$ if $k$ is odd
and to $\flr{\psfrac{k+2}{4}}$ if $k$ is even \cite{Yun15}. Observe that, when $p$ is large compared with $k$, the degree takes the uniform value $\flr{\sfrac{(k+1)}{2}}$. The sets $S$ from Section~\ref{sec:intro1} consist exactly of those prime numbers $p$ at which the degree drops.

From this perspective, the trivial factors of the polynomial $Z_k(p; T)$ are accounted for the invariants and the coinvariants of the inertia action at zero and infinity. By ``removing them'', we mean replacing \'etale cohomology with compact support in the right-hand side of \eqref{eq:characteristicpolynomial} with \textit{middle extension} cohomology, which is defined as
\begin{displaymath}
\coH^1_{\et, \rmid}(\Gmolf, \Sym^k \Kl_2)=\image\bigl[\coH^1_{\et, \rc}(\Gmolf,\Sym^k\Kl_{n+1})\to \coH^1_{\et}(\Gmolf,\Sym^k\Kl_{n+1})\bigr].
\end{displaymath}
The terminology is coherent with the fact that, letting $j \colon \Gm \hookrightarrow \PP^1$ denote the inclusion, the above image agrees with the cohomology on $\PP^1$ of the intermediate (aka middle) extension sheaf~$j_{!\ast} \Sym^k \Kl_2$. By definition, $M_k(p; T)$ is the polynomial
\begin{displaymath}
M_k(p; T)=\det \Bigl( 1-F_pT \mid \coH^1_{\et, \rmid}(\Gmolf, \Sym^k \Kl_2)\Bigr).
\end{displaymath}
Since the \'etale cohomology and the \'etale cohomology with compact support of $\Sym^k \Kl_2$ have weights~\hbox{$\geq k+1$} and $\leq k+1$ respectively by the main theorem of Weil II \cite{WeilII}, the middle extension cohomology is pure of weight $k+1$. What was called $m$ in Section \ref{sec:intro1} \textit{supra} is the degree of $M_k(p; T)$, that is, the dimension of the middle extension cohomology for all $p \notin S$.

\subsection{Exponential Hodge structures and irregular Hodge filtration} Theorems \ref{thm:intro-k-odd} and~\ref{thm:intro-k-even} are proved by constructing a compatible system of potentially automorphic Galois representations
\begin{equation}\label{eqn:introGaloisreps}
r_{k, \ell} \colon \Gal(\ol\QQ / \QQ) \longrightarrow \GL_m(\QQ_\ell),
\end{equation} with $\ell$ running over all prime numbers, such that $r_{k, \ell}$ is unramified at primes $p \neq \ell$ outside $S$ and has traces of Frobenius
\begin{equation}\label{eq:tracesFrobenius}
\mathrm{tr}(r_{k, \ell}(\mathrm{Frob}_p))=\begin{cases}
-m_2^k(p)-1 & \text{ if } 4 \nmid k, \\
-m_2^k(p)-1-p^{\sfrac{k}{2}}& \text{ if } 4 \mid k
\end{cases}
\qquad
(p \notin S \cup \{\ell\}).
\end{equation}
The search for such Galois representations was initiated by Fu and Wan,
who showed in~\cite{FW08} that $L$-functions of symmetric power moments of Kloosterman sums can be realized as Hasse\nobreakdash-Weil zeta functions of \textit{virtual} schemes over $\Spec \ZZ$. An actual Galois representation with traces~\eqref{eq:tracesFrobenius} was first constructed by Yun \cite{Yun15} as a subquotient of the \'etale cohomology of a smooth projective variety over $\QQ$ cut off the affine Grassmannian of $\GL_2$.

Our construction is instead inspired by the theory of exponential motives, as developed by the first author and Jossen \cite{F-J18}. In a nutshell, this is a theory of motives for pairs $(X, f)$ consisting of a smooth variety $X$ over $\QQ$ and a regular function $f$ on~$X$ that enriches the de~Rham cohomology of the vector bundle with connection
$\fqq{f}=(\mathcal{O}_X, \de +\de f)$, that is,
\begin{displaymath}
\coH^n_{\dR}(X, \fqq{f})
=\mathbf{H}^n\Bigl(X, \mathcal{O}_X \To{\de+\de f}
\Omega^1_X \To{\de+\de f} \Omega^2_X \longrightarrow \cdots\Bigr).
\end{displaymath}
We shall also consider the de~Rham cohomology with compact support $\coH^n_{\dR,\rc}(X, \fqq{f})$ and the middle de~Rham cohomology~$\coH^n_{\dR,\rmid}(X, \fqq{f})$, defined as the image of the latter in the former under the ``forget supports'' map.
When $f$ is the zero function, $\coH^n_{\dR}(X, \fqq{f})$ is the usual de~Rham cohomology of~$X$ and one can indeed prove that Nori motives, one of the candidates for the abelian category of mixed motives, form a full subcategory of exponential motives. However, the function does not need to be identically zero for an \emph{a priori} exponential motive to be classical. For instance, this is always the case for an exponential motive of the form~$\coH^n(X \times \Afu, t f)$. If the zero locus $Z=\{f=0\}$ is smooth, it is isomorphic to $\coH^{n-2}(Z)(-1)$, which should be thought of as a cohomological shadow of the identity
\begin{displaymath}
\int_0^\infty \int_{T(\gamma)} e^{-tf} \omega\, \rd t=2\pi\sfi \int_\gamma \mathrm{Res}_Z(\sfrac{\omega}{f}),
\end{displaymath}
where $\omega$ is a differential form on the complement of $Z$ and $T(\gamma)$ is the tube of a chain $\gamma$ in~$Z$. In general, the existence of square roots of the Tate motive $\QQ(-1)$
prevents exponential motives from having realizations in mixed Hodge structures, but they do realize into certain mixed Hodge modules over the affine line that Kontsevich and Soibelman call \textit{exponential mixed Hodge structures}~\cite{K-S10}. It will be enough for our purposes to work in this category, whose main properties are summarized in the appendix.

In analogy with the $\ell$-adic setting, the \textit{Kloosterman connection} $\Kl_2$ on $\Gm$ over a field of characteristic zero is defined by keeping the same diagram \eqref{eq:diagram-Kloosterman} but replacing the \hbox{Artin}\nobreakdash-Schreier sheaf with the differential equation of the exponential. The pullback $f^\ast \AS_\psi$ then becomes $\fqq f=(\cO_{\Gm^2}, \rd+\rd f)$ and one sets
$$
\Kl_2=\pi_+ \fqq f,
$$ which can be thought of as the family of exponential motives $\coH^1(\Gm, x+\sfrac{z}{x})$ parametrized by \hbox{$z\in\Gm$}. Over the complex numbers, $\Kl_2$ is the rank-two vector bundle with connection associated with the modified Bessel differential equation \hbox{$\sfrac{\rd^2 y}{\rd z^2}+(\sfrac{1}{z})\sfrac{\rd y}{\rd z}-y=0$}, which is indeed the one the exponential periods of $\coH^1(\Gm, x+\sfrac{z}{x})$ satisfy. Note that this equation has a regular singularity at zero and an irregular singularity at infinity. We can then form the symmetric powers $\Sym^k \Kl_2$ and consider the various flavors of de~Rham cohomology
\begin{displaymath}
\coH^1_{\dR}(\Gm,\Sym^k\Kl_{2}), \qquad \coH^1_{\dR,\rc}(\Gm,\Sym^k\Kl_2), \qquad \coH^1_{\dR,\rmid}(\Gm,\Sym^k\Kl_2),
\end{displaymath}
where the last space agrees again with the cohomology of the intermediate extension computed on $\PP^1$ (Corollary~\ref{cor:vanishingSymKln}). These vector spaces admit an exponential-Hodge-theoretic interpretation as de~Rham fibers of exponential mixed Hodge structures, respectively denoted~by
\[
\coH^1(\Gm,\Sym^k\Kl_{2}), \qquad \coH^1_{\rc}(\Gm,\Sym^k\Kl_2), \qquad \coH^1_{\rmid}(\Gm,\Sym^k\Kl_2).
\]
As such, they carry an \emph{irregular Hodge filtration}, constructed in \cite[\S6]{Bibi08} by extending an idea of Deligne \cite{Del8406}. On the other hand, we shall prove in Theorem \ref{th:weightsMHM} that these exponential mixed Hodge structures are indeed classical mixed Hodge structures. When this is the case, the irregular Hodge filtration agrees with the usual Hodge filtration, as we show in Proposition~\ref{prop:EMHSMHS}, so we can rely on the geometric interpretation of the former to compute the~latter.

\begin{thm}\label{th:main}
The mixed Hodge structure $\coH^1(\Gm, \Sym^k\Kl_2)$ has weights $\geq k+1$ and the following numerical data:
\begin{enumerate}
\item
For odd $k$, it is mixed of weights $k+1$ and $2k+2$, with
\[
\dim \coH^1(\Gm, \Sym^k\Kl_2)^{p,q}=
\begin{cases}
1 & \text{if } p+q = k+1\text{ and } p \in \{2,4,\dots, k-1\}, \\
1 & \text{if }p = q = k+1, \\
0 & \text{otherwise}.
\end{cases}
\]

\item\label{th:main:itemeven}
For even $k$, it is mixed of weights $k+1$ and $2k+\nobreak2$ if $k\equiv2\bmod4$, and of weights $k+1$, $k+2,$ and $2k+2$ if \hbox{$k\equiv0\bmod4$}, with
\[
\dim \coH^1(\Gm, \Sym^k\Kl_2)^{p,q}=
\begin{cases}
1 & \text{if } p+q = k+1\text{ and } \min\{p,q\}\in \bigl\{2,4,\dots, 2\flr{\sfrac{(k-1)}{4}}\bigr\}, \\
1 & \text{if $p = q =\sfrac{k}{2}+1$ and $k\equiv0\bmod4$}, \\
1 & \text{if }p = q = k+1, \\
0 & \text{otherwise}.
\end{cases}
\]
\end{enumerate}
Furthermore, the mixed Hodge structure $\coH^1_{\rmid}(\Gm,\Sym^k\Kl_2)$ is pure of weight $k+1$ and is equal to $W_{k+1}\coH^1(\Gm, \Sym^k\Kl_2)$.
\end{thm}

To prove that $\coH^1(\Gm, \Sym^k\Kl_2)$ carries a mixed Hodge structure and compute its Hodge numbers, we first establish the analogous result for the pullback $\wt\Kl_2$ of~$\Kl_2$ by the double cover of $\Gm$. The symmetric power $\Sym^k\wt\Kl_2$ turns out to be the restriction of the Fourier transform of a $\cD$-module on $\Afu$ that underlies a pure Hodge module and, as we explain in Section \ref{subsec:computeweights}, the theory of mixed Hodge modules endows its cohomology with a mixed Hodge structure whose numerical invariants can be computed in terms of nearby cycles. On the other hand, seeing the cohomology of symmetric powers as the alternating part of the cohomology of tensor powers and using a refined form of the K\"unneth formula,
we get an isomorphism
\begin{equation}\label{eq:classical-motives}
\coH^1_{\dR}(\Gm,\Sym^k\Kl_{2}) \simeq \coH^1_{\dR}(\Gm,\Kl_{2}^{\otimes k})^{\mathrm{sign}}
\simeq \coH^{k+1}_{\dR}(\Gm^{k+1}, \fqq{f_k})^{\mathrm{sign}},
\end{equation}
where $f_k$ is the function $x_1+\cdots+x_k+z(\sfrac{1}{x_1}+\cdots+\sfrac{1}{x_k})$ and sign denotes the eigenspace on which the symmetric group $\symgp_k$ acts through the sign. After pullback by the cover $t \mapsto z=t^2$ and the change of coordinates $x_i=ty_i$,
the function $f_k$ takes the form $ \widetilde{f}_k=tg^{\boxplus k}$, where $g^{\boxplus k}$ is the~$k$\nobreakdash-fold Thom-Sebastiani sum $g(y_1)+\cdots+g(y_k)$ of the function $g(y)=y+1/y$
with itself. The cohomology group \eqref{eq:classical-motives} is hence also given by the invariants of
\begin{equation}\label{eq:classical-motives2}
\coH^1_{\dR}(\Gm,\Sym^k\wt\Kl_{2}) \simeq \coH^{1}_{\dR}(\Gm^{k+1}, \fqq{t g^{\boxplus k}})^{\mathrm{sign}}
\end{equation} under the action of $\mu_2$ coming from the cover. On toric varieties such as a compactification of~$\Gm^{k}$ adapted to the function $f_k$, the work of Adolphson and Sperber~\cite{A-S97} and the results of~\hbox{\cite{E-S-Y13, Yu12}} lead to a geometric interpretation of the irregular Hodge filtration that, once we know the Hodge numbers of \eqref{eq:classical-motives2}, enables us to complete the proof of Theorem~\ref{th:main}.

From this circle of ideas we also see that the Hodge structure $\coH^1_{\rmid}(\Gm, \Sym^k\Kl_2)$ has motivic origin, in the sense that it is cut out of the cohomology of an algebraic variety. Indeed, replacing with~$\Afu$ the copy of~$\Gm$ with coordinate~$t$ in $\Gm^{k+1}$ and combining the Gysin and the localization long exact sequences, we obtain the following description in Theorem \ref{Thm:Kl_in_Mcl}: let $\KM \subset \Gm^k$ be the zero locus of $g^{\boxplus k}$, on which the group $\symgp_k \times \mu_2$ acts by permuting the coordinates and sending $y_i$ to $-y_i$. Then there is an isomorphism of pure Hodge structures
\begin{equation}\label{eqn:motivicorigin}
\coH^1_{\rmid}(\Gm,\Sym^k\Kl_2)\cong \gr_{k-1}^W \coH^{k-1}_{\cp}(\KM)^{\mathrm{sign} \times \mu_2}(-1).
\end{equation}
For odd $k$, the hypersurface $\KM$ is smooth and we also obtain an isomorphism
\[
\coH^1_{\rmid}(\Gm,\Sym^k\Kl_2)\cong \gr_{k-1}^W \coH^{k-1}(\KM)^{\mathrm{sign} \times \mu_2}(-1).
\]
The right-hand side of \eqref{eqn:motivicorigin} is the Hodge realization of a pure motive $\Motive_k$ over $\QQ$ and the Galois representations $r_{k, \ell}$ from \eqref{eqn:introGaloisreps} arise as its $\ell$-adic realizations.

That a paper seemingly about $L$-functions bears the title
``Hodge theory of Kloosterman connections'' may come as a surprise.
The reason for this choice is that we see Theorem~\ref{th:main} as the crux of our contribution. Once we know that all Hodge numbers are either zero or one (equivalently, that the Galois representations $r_{k, \ell}$ are \textit{regular}), a recent theorem of Patrikis and Taylor \cite{PT15}, building on previous work of Barnet-Lamb, Gee, Geraghty, and Taylor \cite{BGGT}, implies that the $r_{k, \ell}$ are potentially automorphic, and hence that their $L$-functions meromorphically extend to the complex plane and satisfy a functional equation. As expounded in the sequel paper \cite{F-S-Y20b}, our approach also explains the relation, numerically checked to high precision by Broadhurst and Roberts in many examples, between special values of the $L$\nobreakdash-functions at critical integers and certain determinants of \textit{Bessel moments}
\begin{displaymath}
\int_0^\infty I_0(z)^a K_0(z)^{k-a} z^b \sfrac{dz}{z},
\end{displaymath}
where $I_0(z)$ and $K_0(z)$ are the modified Bessel functions of the first and the second kind.

\subsection{Overview} Briefly, the paper is organized as follows. In the preparatory Section \ref{sec:expmot}, we gather the main properties of Kloosterman connections and their symmetric powers. The mixed Hodge structures are constructed in Section \ref{sec:basicKl}, where we also exhibit their avatars over finite fields. Section \ref{sec:comp-Hodge} is devoted to the proof of Theorem~\ref{th:main}. Finally, in Section~\ref{sec:L-functions} we compute the \'etale realizations of the motives and pull everything together to prove Theorems~\ref{thm:intro-k-odd} and \ref{thm:intro-k-even}. The paper is supplemented by an appendix concerning exponential mixed Hodge structures and the irregular Hodge filtration.

\subsection{Acknowledgments} This paper owes much to David Broadhurst and David Roberts, whose wonderful insights guided us all the way through. We thank Takeshi Saito for suggesting the argument to compute the action of complex conjugation in the proof of Corollary~\ref{Cor:Gamma-factor},
and Ga\"etan Chenevier, Olivier Ta\"ibi, and Daqing Wan for encouraging us, with their questions, to improve some of the results of a previous version. The first author also would like to acknowledge useful conversations with Spencer Bloch, Jean\nobreakdash-Beno\^it Bost, Pierre Colmez, Peter Jossen, Emmanuel Kowalski, Yichen Qin, Duco van Straten, and Tomohide Terasoma. Finally, we thank the three anonymous referees for their careful reading of various versions of this paper and their many suggestions to correct inaccuracies and improve the presentation.

\section{Symmetric powers of Kloosterman connections}\label{sec:expmot}

In this section, we gather the properties of Kloosterman connections and their symmetric powers that are relevant for the construction of the mixed Hodge structures in the next sections. We refer the reader to Appendix~\ref{subsec:notationD} for the notation and results from the theories of $\cD$-modules and mixed Hodge modules that are used in what follows, and to \cite[\S II]{Katz87} for the notion of slopes of a meromorphic connection at an irregular singularity.

\subsection{Structure of Kloosterman connections}\label{subsec:basicKln}

Let~$n\geq1$ be an integer. We first define Kloosterman connections $\Kl_{n+1}$ generalizing~$\Kl_2$ from the introduction. For simplicity,
we work over the base field $\CC$, although all results remain valid over a field of characteristic zero. Let $\Gm$ denote the one-dimensional torus. We endow the product $\Gm^{n+1}$ with coordinates $(z, x)=(z, x_1, \ldots, x_n)$ and consider the diagram
\begin{equation*}\label{eq:diagramKloos}
\begin{array}{c}
\xymatrix{
&\Gm^{n+1}\ar[dr]^{f}\ar[dl]_{\pi}&\\
\Gm&&\Afu,
}
\end{array}
\end{equation*} where $\pi$ is the projection $(z, x) \mto z$ to the first factor and $f$ is the function
\begin{equation}\label{eqn:functionf}
f(\tz,x)=x_1+\cdots+x_n+\frac{\tz}{x_1\cdots x_n}.
\end{equation} We define $\Kl_{n+1}$ as the bounded complex of~$\cD_{\Gm}$\nobreakdash-modules
\begin{equation}\label{eq:Klpf}
\Kl_{n+1}=\pi_\bast \fqq{f},
\end{equation} where $\fqq f=(\cO_{\Gm^{n+1}},\rd+\rd f)$ is the connection attached to the exponential of $f$. After identifying $\Gamma(\Gm{}_{,\tz},\cO_{\Gm})$ with $\CC[\tz,\tz^{-1}]$, we can see $\Kl_{n+1}$ as a complex of $\Cltzm$-modules.

Besides, we consider the cyclic Galois cover
\[
\ramifn\colon {\Gm}_{, t} \to {\Gm}_{,\tz}
\]
induced by $\tz \mto \tt^{n+1}$, which has Galois group $\mu_{n+1}$, and we set
\[
\wtKl_{n+1}=\ramifn^\bast\Kl_{n+1} \simeq \wtpi_\bast \fqq{\wtf},
\]
with $\wtpi(\tt,x)=\tt$ and $\wtf(t, x)=f(t^{n+1}, x)$.
The group $\mu_{n+1}$ acts on $\wtKl_{n+1}$, and hence on the pushforward $\ramifn_+\wtKl_{n+1}$, and the original complex can be retrieved as the invariants
\[
\Kl_{n+1}=(\ramifn_+\wtKl_{n+1})^{\mu_{n+1}}.
\]
Let $g \colon \Gm^n\to\Afu$ be the function defined as
\begin{equation}\label{eq:defg}
g(y_1,\dots,y_n)=f(1,y_1,\dots,y_n)=y_1+\cdots+y_n+\frac1{y_1\cdots y_n}.
\end{equation} Since the change of variables $(\tt, x) \mapsto (\tt, y)=(\tt, x/\tt)$ on $\Gm^{n+1}$ turns $\wtf$ into $\tt g(y)$ and is compatible with projections to the first factor, we also get
\[
\wtKl_{n+1}=\wtpi_\bast \fqq{\tt g(y)}.
\]

\begin{prop}\label{prop:Kln}
The complex $\Kl_{n+1}$ is concentrated in degree zero, \ie the equality
\[
\Kl_{n+1}=\SH^0 \pi_\bast \fqq{f}
\]
holds. Moreover, $\Kl_{n+1}$ has the following properties:
\begin{enumerate}
\item\label{prop:Kln1}
$\Kl_{n+1}$ is the irreducible free $\cO_{\Gm}$-module of rank $n+1$ with connection associated with the hypergeometric differential operator \hbox{$(\tz\partial_\tz)^{n+1}-\tz$.}
\item\label{prop:Kln2}
$\Kl_{n+1}$ has a regular singularity at $\tz=0$ with unipotent monodromy and a single Jordan block, and an irregular singularity of pure slope $1/(n+1)$ at $\tz=\infty$.
\item\label{prop:Kln4}
Let $\Kl_{n+1}^\vee$ be the $\cO_{\Gm}$-module dual to $\Kl_{n+1}$ endowed with the dual connection, and let~$\iota_r$ denote the involution $\tz\mto(-1)^r\tz$. Then there is an isomorphism
\[
\Kl_{n+1}^\vee\simeq\iota_{n+1}^\bast\Kl_{n+1}.
\]
\item\label{prop:Kln5}
$\wtKl_{n+1}$ is the restriction to $\Gm$ of the Fourier transform of a regular holonomic $\cD_{\Afu}$\nobreakdash-module.
\end{enumerate}
\end{prop}

\begin{proof} The arguments of \cite[Th.\,7.4\,\&\,7.8]{Deligne77SGA} can readily be transposed to the complex setting to show that $\Kl_{n+1}$ is a free $\cO_{\Gm}$-module with connection sitting in degree zero. Instead, we give a proof based on the following recursive description. Let $\inv$ be the involution~$\tz\mto 1/\tz$ on $\Gm$ and consider the localized Fourier transformation
\[
\cF=j_0^\bast\FT j_{0\bast} \colon \catD^\rb_\hol(\cD_{\Gm})\to\catD^\rb_\hol(\cD_{\Gm}).
\]
Recall that this functor preserves holonomic modules, and sends regular holonomic modules to smooth holonomic modules on $\Gm$ with a regular singularity at the origin and a possibly irregular singularity with slopes in~$\{0, 1\}$ at infinity. For the sake of the induction, it~will be convenient to set $\Kl_1=j_0^\bast \fqq y$ and start with $n=0$.

\begin{lemma}\label{lem:iteration}
For $n\geq0$, the isomorphism $\Kl_{n+2}\simeq\cF\inv^\bast\Kl_{n+1}
$ holds.
\end{lemma}

\begin{proof}
We set $\tz'=1/\tz$ and $x_0 = z/x_1\cdots x_n$,
so that $p_{\tz'}=1/x_0\cdots x_n$ and
\[
\inv^\bast\Kl_{n+1}=p_{\tz'\bast} E^{x_0+\cdots+x_n}.
\]
Hence, denoting by $\tauz$ a coordinate on a new factor $\Gm$ and by $\pi_{\tauz}$ the corresponding projection,
\[
\cF\inv^\bast\Kl_{n+1}=\pi_{\tauz\bast} E^{x_0+\cdots+x_n+\tauz \tz'}=\pi_{\tauz\bast} E^{x_0+\cdots+x_n+\tauz/x_0\cdots x_n}\simeq\Kl_{n+2}.\qedhere
\]
\end{proof}

Let us now show \eqref{prop:Kln1} and \eqref{prop:Kln2}. Since $\Kl_1$ is a holonomic $\cD_{\Gm}$-module, the same goes for each~$\Kl_{n+1}$ by induction. That $\Kl_{n+1}$ is isomorphic to the free $\cO_{\Gm}$-module of rank $n+1$ with connection defined by $L=(\tz\partial_\tz)^{n+1}-\tz$ is clear for $n=0$. Assuming it true for~$\Kl_{n+1}$, Lemma~\ref{lem:iteration} implies that $\Kl_{n+2}$ is isomorphic~to
\begin{align*}
\cF\inv^\bast\bigl(\cD_{\Gm}/\cD_{\Gm}L\bigr) \simeq \cF\bigl(\cD_{\Gm}/\cD_{\Gm}((-\tz\partial_\tz)^{n+1}z-1))\bigr) \simeq \cD_{\Gm}/\cD_{\Gm}((\tz\partial_\tz)^{n+2}-z),
\end{align*} where the first isomorphism is given by inversion and multiplication by $z$ on the right.
With respect to a suitable basis, the matrix of $\tz\partial_\tz$ acting on $\Kl_{n+1}$ is thus given by
\[
\begin{pmatrix}0&0&\cdots&0&\tz\\1&0&\cdots&0&0\\0&1&\cdots&0&0\\ 0 & 0 &\cdots& 1& 0 \end{pmatrix},
\]
which shows that $z=0$ is a regular singularity with unipotent monodromy and a single Jordan block. It also follows from this description that $z=\infty$ is an irregular singularity of slope~$\sfrac{1}{(n+1)}$, and that $\Kl_{n+1}$ is irreducible. Indeed, $j_{0\bast}\inv^\bast \Kl_{n+1}$ agrees with the intermediate extension $j_{0\dag+}\inv^\bast \Kl_{n+1}$ since~$\inv^\bast \Kl_{n+1}$ has slope $\sfrac{1}{(n+1)}$ at $0$; it is thus an irreducible $\cD_{\Afu}$-module, and Fourier transformation preserves this property.

We note that property \eqref{prop:Kln4} is true for $n=0$. Assuming it holds for $\Kl_{n+1}$, we prove it for $\Kl_{n+2}$. Fourier transformation commutes with duality of $\cD_{\Afu}$-modules up to~$\iota_1$, as follows from \hbox{\cite[Lem.\,V.3.6]{Malgrange91}}, and the latter duality corresponds to the duality of free~$\cO_{\Gm}$\nobreakdash-modules with connection through the pair of functors $j_0^\bast$ and $j_{0\bea}$. Thus Lemma \ref{lem:iteration} yields \eqref{prop:Kln4}.

Finally, the pullback $\wtKl_{n+1}$ of $\Kl_{n+1}$ by the finite morphism $[n+1]$ is also concentrated in degree zero and is smooth of rank $n+1$ on $\Gm$. Consider the $\cD_{\Afu}$\nobreakdash-module
\[
M_{n+1}=\SH^0g_\bast\cO_{\Gm^n}.
\]
Letting $j_0 \colon \Gm\hto\Afu$ denote the inclusion, we claim that there is an isomorphism
\begin{equation}\label{eq:KltildeisFourier}
\wtKl_{n+1}\simeq j_0^\bast\FT M_{n+1}
\end{equation}
of $\Clttm$-modules. This will yield~\eqref{prop:Kln5}. We first observe that $\wtKl_{n+1}$ is isomorphic to~$j_0^\bast\FT(g_\bast\cO_{\Gm^n})$. Indeed, letting $p_t \colon {\Gm}_{, t} \times \Afu_\taut \to {\Gm}_{, t}$ and $p_\taut \colon {\Gm}_{, t} \times \Afu_\taut \to \Afu_\taut$ denote the projections and writing $\wtpi$ as the composition of $\id \times g \colon {\Gm}_{, t} \times \Gm^n \to {\Gm}_{, t} \times \Afu$ and $p_t$, the projection formula gives isomorphisms
\[
\wtpi_\bast\fqq{tg}\simeq {p_\tt}_\bast((\id \times g)_\bast\fqq{tg})\simeq {p_\tt}_\bast(p_\taut^\bast g_\bast\cO_{\Gm^{n}}\otimes\fqq{\tt\taut}) \simeq j_0^\bast\FT(g_\bast\cO_{\Gm^n}).
\]
The claim then follows from the fact that $\wtKl_{n+1}$ is concentrated in degree zero.
\end{proof}

\subsection{Cohomology of \texorpdfstring{$\Sym^k\Kl_{n+1}$}{SymkKln1}}\label{subsec:basicSymkKln}

In what follows, we work in the abelian tensor category of vector bundles with connection on $\Gm$. For each integer $k \geq 1$, we define the $k$-th symmetric power $\Sym^k \Kl_{n+1}$ (\resp $\Sym^k \wtKl_{n+1}$)
as the invariants of the tensor product $\Kl_{n+1}^{\otimes k}$ (\resp $\wtKl_{n+1}^{\otimes k}$) under the action of the symmetric group $\symgp_k$. There are isomorphisms
\begin{align*}\label{eq:splitting_fqq_tilde_f}
\wtKl_{n+1}^{\otimes k} &\simeq \ramifn^\bast\Kl_{n+1}^{\otimes k},&\Sym^k \wtKl_{n+1}&\simeq \ramifn^\bast\Sym^k \Kl_{n+1},\notag
\\
\Kl_{n+1}^{\otimes k} &\simeq (\ramifn_\bast\wtKl_{n+1}^{\otimes k})^{\mu_{n+1}},&\Sym^k \Kl_{n+1}&\simeq (\ramifn_\bast\Sym^k \wtKl_{n+1})^{\mu_{n+1}},\\
\ramifn_\bast \wtKl_{n+1}^{\otimes k}&\simeq \bigoplus_{i=0}^n \tz^{i/(n+1)}\Kl_{n+1}^{\otimes k}, &\ramifn_\bast \Sym^k \wtKl_{n+1}&\simeq \bigoplus_{i=0}^n \tz^{i/(n+1)}\Sym^k \Kl_{n+1},\notag
\end{align*}
where $\tz^{i/(n+1)}\mathcal{E}$ denotes the Kummer twist $(\cO_{\Gm}, \de + \frac{i}{n+1}\sfrac{\de\tz}{\tz}) \otimes \mathcal{E}$ of a vector bundle with connection $\mathcal{E}$. This follows from the decomposition
\[
\ramifn_\bast\cO_{\Gm} = \bigoplus_{i=0}^n (\cO_{\Gm}, \de + \tfrac{i}{n+1}\sfrac{\de\tz}{\tz})
\]
into eigenspaces for the action of the Galois group $\mu_{n+1}$ of the cover $\ramifn$. Note the equality
\[
\rk\Sym^k\Kl_{n+1}=\rk\Sym^k\wtKl_{n+1}=\binom{n+k}{k}.
\]

\begin{prop}\label{prop:irreducibility} The connections~$\Sym^k\Kl_{n+1}$ and $\Sym^k\wtKl_{n+1}$ are irreducible.
\end{prop}

\begin{proof} For each $n\geq 1$, the differential Galois group of $\Kl_{n+1}$ and $\wtKl_{n+1}$ is equal to $\mathrm{SL}_{n+1}(\CC)$ if $n$ is even and to $\mathrm{Sp}_{n+1}(\CC)$ if $n$ is odd (see \cite[Cor.\,4.4.8]{Katz87} taking Proposition \ref{prop:Kln}\,\eqref{prop:Kln1} into account). Since any symmetric power of the standard representation of these groups is irreducible (\cf \eg \cite[\S15.3 \& \S17.3]{FultonHarris}), it follows that $\Sym^k\Kl_{n+1}$ and $\Sym^k\wt\Kl_{n+1}$ are irreducible.
\end{proof}

Let us consider the Laurent polynomial
\begin{equation}\label{eq:f_k}
f_k = \sum_{j=1}^k \biggl(\sum_{i=1}^n x_{ji} + \tz\prod_{i=1}^n\frac{1}{x_{ji}} \biggr) \colon \Gm^{kn+1} \to \Afu
\end{equation}
and let $\pi_k\colon \Gm^{kn+1} \to {\Gm}_{,\tz}$ denote the projection to the coordinate $\tz$. With the above notation, the equalities $f_1=f$ and $\pi_1=\pi$ hold. Let us consider the Cartesian square
\[
\xymatrix@C=1.5cm{
\Gm^{kn+1} \ar[r]^{\wtramifn}\ar[d]_{\wt{\pi}_k} & \Gm^{kn+1} \ar[d]^{\pi_k} \\ {\Gm}_{,\tt} \ar[r]^\ramifn & {\Gm}_{,\tz}}
\]
and set $\wtf_k = \wtramifn^*f_k$, so that again $\wtf_1=\wtf$. By the change of variables $(\tt,y_{ji})=(\tt, x_{ji}/\tt)$ as in the previous section, $\wtf_k$ can be written as
\begin{equation}\label{eqn:relationfandg}
\wtf_k=\tt\cdot g^{\boxplus k},
\end{equation} where $g^{\boxplus k}$ is the $k$-fold Thom-Sebastiani sum of the Laurent polynomial $g$ from \eqref{eq:defg} with itself. There is a natural action of $\mu_{n+1}$ on
$\SH^0\wtramifn_\bast E^{\wtf_k}$ such that
\[
E^{f_k}=(\SH^0\wtramifn_\bast E^{\wtf_k})^{\mu_{n+1}}.
\]

\begin{prop}\label{prop:GTS}
There are isomorphisms
\[
\Kl_{n+1}^{\otimes k} \simeq \SH^0{\pi_k}_\bast\fqq{f_k}={\pi_k}_\bast\fqq{f_k}\quad\text{and}
\quad
\wtKl_{n+1}^{\otimes k} \simeq \SH^0\wtpi_{k+}\fqq{\wtf_k}=\wtpi_{k+}\fqq{\wtf_k}.
\]
\end{prop}

\begin{proof}
We first prove the statement about $\wtKl_{n+1}^{\otimes k}$. Recall the equality \hbox{$\wtf_k=\tt\cdot g^{\boxplus k}$} and consider the complex of $\cD_{\Afu_\taut}$-modules with regular holonomic cohomology
\[
M_{n+1}^{(k)}=g^{\boxplus k}_+\cO_{\Gm^{kn}}.
\]
Arguing as in the proof of Proposition \ref{prop:Kln}, we get an isomorphism $\wtpi_{k+}\fqq{\wtf_k}\simeq j_0^\bast\FT M_{n+1}^{(k)}$. Let~$s_k \colon \Afu\times\cdots\times\Afu\to\Afu$ be the sum map. Writing \hbox{$g^{\boxplus k}=s_k\circ(g\times\cdots\times g)$,} we can identify $M_{n+1}^{(k)}$ with the $k$-fold additive $*$-convolution of $M_{n+1}^{(1)}$ with itself. Since Fourier transformation exchanges additive $*$-convolution and derived tensor product on~$\Afu_t$ (as recalled in Section~\ref{subsec:notationD}), we deduce upon localization an isomorphism of complexes of~$\CC[\tt,\tt^{-1}]$\nobreakdash-modules with connection
\[
j_0^+\FT M_{n+1}^{(k)}\simeq (j_0^+\FT M_{n+1}^{(1)})^{{\overset{\mathrm{L}}\otimes} k} ,
\]
where the tensor product is taken over the coordinate ring $\CC[\tt,\tt^{-1}]$ of ${\Gm}_{,\tt}$. Now, recall from the proof of Proposition \ref{prop:Kln} that the complex $j_0^+\FT M_{n+1}^{(1)}$ is concentrated in degree zero and is a free $\CC[\tt,\tt^{-1}]$-module isomorphic to $\wtKl_{n+1}$. It follows that the above derived tensor product is an ordinary tensor product, isomorphic to $\wtKl_{n+1}^{\otimes k}$.

Let us now consider the case of $\Kl_{n+1}^{\otimes k}$. From the equality $\pi_k\circ(\id\times[n+1])=[n+1]\circ\wtpi_k$, along with the first part, we deduce $\mu_{n+1}$-equivariant isomorphisms
\begin{align*}
[n+1]_+\wtKl_{n+1}^{\otimes k}\simeq[n+1]_+\wtpi_{k+}E^{\wt f_k} &\simeq\pi_{k+}((\id\times[n+1])_+E^{\wt f_k})=\cH^0\pi_{k+}((\id\times[n+1])_+E^{\wt f_k}).
\end{align*}
Moreover, the action on the rightmost term comes from that on $(\id\times[n+1])_+E^{\wt f_k}$, whose invariant submodule is $E^{f_k}$. Taking $\mu_{n+1}$-invariants on the extreme terms, we deduce the first isomorphism of the statement.
\end{proof}

\begin{cor}\label{cor:KummethForTensor}
There are isomorphisms
\[
\coH^1_\dR(\Gm,\Kl_{n+1}^{\otimes k})\simeq \coH^{kn+1}_\dR(\Gm^{kn+1},\fqq{f_k}),\quad \coH^1_\dR(\Gm,\wtKl_{n+1}^{\otimes k})\simeq \coH^{kn+1}_\dR(\Gm^{kn+1},\fqq{\wtf_k}).
\]
\end{cor}

\begin{proof} Letting $q$ and $\wt q$ denote the structure morphisms of ${\Gm}_{,\tz}$ and ${\Gm}_{,\tt}$ respectively, we deduce isomorphisms of complexes
\[
q_\bast \Kl_{n+1}^{\otimes k} \simeq (q\circ\pi_k)_\bast \fqq{f_k}, \qquad q_\bast \wtKl_{n+1}^{\otimes k} \simeq (\wt q\circ\wtpi_k)_\bast\fqq{\wtf_k}
\]
from Proposition~\ref{prop:GTS} along with the isomorphism of functors $(q\circ\pi_k)_\bast\simeq q_\bast\circ{\pi_k}_\bast$ and \hbox{$(\wt q\circ\wtpi_k)_\bast\simeq \wt q_\bast\circ\wtpi_{k+}$.} The statement then follows by taking cohomology in degree zero.
\end{proof}

Since the dual of $\fqq{f_k}$ is $\fqq{-f_k}$, Poincaré-Verdier duality yields a commutative diagram
\[
\begin{array}{c}
\xymatrix@C=1.5cm{
\coH_{\dR,\cp}^{kn+1}(\Gm^{kn+1},\fqq{f_k}) \otimes \coH_\dR^{kn+1}(\Gm^{kn+1},\fqq{-f_k})\ar@<-8ex>[d]\ar[r] &\coH_{\dR,\cp}^{2kn+2}(\Gm^{kn+1}) =\CC\ar@{=}[d]\\
\coH_{\dR}^{kn+1}(\Gm^{kn+1},\fqq{f_k}) \otimes \coH_{\dR, \cp}^{kn+1}(\Gm^{kn+1},\fqq{-f_k})\ar@<-8ex>[u]\ar[r]&\coH_{\dR,\cp}^{2kn+2}(\Gm^{kn+1}) =\CC
}
\end{array}
\]
in which the horizontal arrows are non-degenerate pairings. Denote the image of the first vertical arrow by
$\coH_{\dR,\rmid}^{kn+1}(\Gm^{kn+1},\fqq{f_k})$.
Consider the involution $\iota$ on~$\Gm^{kn+1}$ given by
\[
(z,x_{ji}) \mto ((-1)^{n+1}z,-x_{ji}).
\]
Since $\iota_\bast\fqq{f_k} = \iota^+\fqq{f_k} = \fqq{-f_k}$, it induces an isomorphism from the de Rham cohomology and the de Rham cohomology with compact support of~$\fqq{f_k}$ to those of $\fqq{-f_k}$, from which we deduce a self non-degenerate pairing
\begin{equation}\label{eq:pairing_dR_f_k}
\coH_{\dR,\rmid}^{kn+1}(\Gm^{kn+1},\fqq{f_k}) \otimes
\coH_{\dR,\rmid}^{kn+1}(\Gm^{kn+1},\fqq{f_k})\to\CC.
\end{equation} This pairing is $(-1)^{kn+1}$-symmetric since $\iota$ acts trivially on $\coH_{\dR,\cp}^{2kn+2}(\Gm^{kn+1})$. There is a similarly defined pairing on $\coH_{\dR,\rmid}^{kn+1}(\Gm^{kn+1},\fqq{\wtf_k})$, which induces \eqref{eq:pairing_dR_f_k}
by taking $\mu_{n+1}$-invariants.

In what follows, we consider objects acted upon by one of the groups $\symgp_k$ and $\symgp_k\times\mu_{n+1}$, and we use the uniform notation $\chi \colon G \to \{\pm 1\}$ for the sign character if $G=\symgp_k$ and the product of the sign character on $\symgp_k$ and the trivial character on $\mu_{n+1}$ if $G=\symgp_k\times\mu_{n+1}$. We~also set~$\chi_n=\chi^n$ for $n \geq 1$. Given a representation $V$ of $G$ over a field $K$ of characteristic zero, $V^{G,\chi_n}$ denotes the $\chi_n$-isotypic part of $V$, defined as the image of the idempotent
\[
\frac{1}{|G|}\sum_{\sigma\in G} \chi_n(\sigma)\sigma
\]
in the group ring $K[G]$ acting on $V$.

\begin{prop}\label{prop:HdRSymKln}
The de~Rham cohomology of $\Sym^k\Kl_{n+1}$ and $\Sym^k\wtKl_{n+1}$ is concentrated in degree one. Moreover, there are isomorphisms
\begin{align*}
\coH^1_\dR(\Gm,\Sym^k\Kl_{n+1}) &\simeq
\coH^{kn+1}_\dR(\Gm^{kn+1},\fqq{f_k})^{\symgp_k,\,\chi_n} \simeq
\coH^{kn+1}_\dR(\Gm^{kn+1},\fqq{\wtf_k})^{\symgp_k\times\mu_{n+1},\chi_n}, \\
\coH^1_\dR(\Gm,\Sym^k\wtKl_{n+1}) &\simeq
\coH^{kn+1}_\dR(\Gm^{kn+1},\fqq{\wtf_k})^{\symgp_k,\,\chi_n}.
\end{align*}
\end{prop}

\begin{proof} Since $\Gm$ is an affine curve, there is no cohomology in degree greater than one and it suffices to prove that $\coH^0_{\dR}$ vanishes. This cohomology space being given by the invariants under the action of the differential Galois group, the vanishing follows from the fact that~$\Sym^k\Kl_{n+1}$ and $\Sym^k\wtKl_{n+1}$ are non-trivial irreducible representations (Proposition~\ref{prop:irreducibility}).

Notice, however, that the actions of $\sigma \in \symgp_k$
on the left-hand and right-hand sides of each of the isomorphisms
in Proposition~\ref{prop:GTS}
commute up to the sign $\chi_n(\sigma)$,
which can be seen \eg by representing elements as differential forms
(this is similar to the discussion on the determinant of $\coH^1$ in \cite[\S 12.3.1]{F-J18}). By taking the $G$-action
on the isomorphisms of Corollary \ref{cor:KummethForTensor} into account,
the desired isomorphisms are then clear.
\end{proof}

Let $j\colon \Gm\hto\PP^1$ be the inclusion. Besides the meromorphic extension $j_\bast\Sym^k\Kl_{n+1}$, we consider the following $\cD_{\PP^1}$-modules that also extend the $\cD_{\Gm}$\nobreakdash-mod\-ule $\Sym^k\Kl_{n+1}$ to $\PP^1$:
\begin{equation}\label{eqns:extensions}
\begin{aligned}
j_\bexc\Sym^k\Kl_{n+1}&=\bD j_\bast\Sym^k\Kl_{n+1}^\vee\simeq\iota_{n+1}^\bast\bD j_\bast\Sym^k\Kl_{n+1}\text{ (after \ref{prop:Kln}\eqref{prop:Kln4})},\\
j_\bea\Sym^k\Kl_{n+1}&=\image\bigl[j_\bexc\Sym^k\Kl_{n+1}\to j_\bast\Sym^k\Kl_{n+1}\bigr],
\end{aligned}
\end{equation}
where $\bD$ denotes the duality of $\cD$-modules. The equality
\[
\coH^r_\dR(\Gm,\Sym^k\Kl_{n+1})=\coH^r_\dR(\PP^1,j_\bast\Sym^k\Kl_{n+1})
\]
holds, and we set
\begin{align*}
\coH^r_{\dR,\rc}(\Gm,\Sym^k\Kl_{n+1})&=\coH^r_\dR(\PP^1,j_\bexc\Sym^k\Kl_{n+1}),\\
\coH^r_{\dR,\rmid}(\Gm,\Sym^k\Kl_{n+1})&=\image\bigl[\coH^r_{\dR,\rc}(\Gm,\Sym^k\Kl_{n+1})\to \coH^r_\dR(\Gm,\Sym^k\Kl_{n+1})\bigr].
\end{align*}

\begin{cor}\label{cor:vanishingSymKln}
The de Rham cohomology with compact support $\coH^r_{\dR,\rc}(\Gm,\Sym^k\Kl_{n+1})$ and the middle de Rham cohomology $\coH^r_{\dR,\rmid}(\Gm,\Sym^k\Kl_{n+1})$ vanish for $r\neq1$. Moreover, there are isomorphisms
\begin{equation}\label{eqn:twomiddlesagree}
\begin{aligned}
\coH^1_{\dR, \cp}(\Gm,\Sym^k\Kl_{n+1}) &\simeq
\coH^{kn+1}_{\dR, \cp}(\Gm^{kn+1},\fqq{f_k})^{\symgp_k,\,\chi_n} \\&\simeq
\coH^{kn+1}_{\dR, \cp}(\Gm^{kn+1},\fqq{\wtf_k})^{\symgp_k\times\mu_{n+1},\chi_n}, \\
\coH^1_{\dR,\rmid}(\Gm,\Sym^k\Kl_{n+1})
&=\coH^1_\dR(\PP^1,j_\bea\Sym^k\Kl_{n+1}),
\end{aligned}
\end{equation}
and a perfect $(-1)^{kn+1}$-symmetric pairing
\[
\coH^1_{\dR,\rmid}(\Gm,\Sym^k\Kl_{n+1}) \otimes
\coH^1_{\dR,\rmid}(\Gm,\Sym^k\Kl_{n+1}) \to \CC
\]
induced by Poincar\'e duality. The same statement holds for $\Sym^k\wtKl_{n+1}$.
\end{cor}

\begin{proof}
Combining Poincaré duality for $\cD_{\PP^1}$-modules, the isomorphism \eqref{eqns:extensions}, and the fact that~$\iota_{n+1}^\bast$ induces an isomorphism on de Rham cohomology, we get
\[
\coH^r_{\dR, \rc}(\Gm, \Sym^k\Kl_{n+1}) \simeq \coH^r_\dR(\PP^1, \iota_{n+1}^\bast\bD j_\bast\Sym^k\Kl_{n+1})\simeq \coH^{2-r}_{\dR}(\Gm, \Sym^k\Kl_{n+1})^\vee,
\]
and hence the vanishing of the left-hand side for $r\neq1$ by Proposition~\ref{prop:HdRSymKln}. This immediately implies the vanishing of $\coH^r_{\dR,\rmid}(\Gm,\Sym^k\Kl_{n+1})$ for $r\neq 1$. The first two isomorphisms are proved in the same way as those of Proposition~\ref{prop:HdRSymKln}.

To shorten notation, we write $\coH^r=\coH^r_{\dR}(\PP^1, j_\bea\Sym^k\Kl_{n+1})$. From the fact that the intermediate extension $j_\bea\Sym^k\Kl_{n+1}$ is self-dual up to~$\iota_{n+1}$, we get an isomorphism $\coH^r \simeq (\coH^{2-r})^\vee$ similarly as above. Since the morphism $j_\bexc\Sym^k\Kl_{n+1}\to j_\bast\Sym^k\Kl_{n+1}$ factors through $j_\bea\Sym^k\Kl_{n+1}$, to prove~\eqref{eqn:twomiddlesagree} it suffices to show that the map $\coH^1_{\dR,\rc}(\Gm,\Sym^k\Kl_{n+1})\!\to\! \coH^1$ is surjective, which then implies that $\coH^1\!\to\! \coH^1_\dR(\Gm,\Sym^k\Kl_{n+1})$ is injective by duality. The former appears in the long exact sequence associated with
\[
0 \to \ker[j_\bexc\!\Sym^k\Kl_{n+1} \to j_\bea\Sym^k\Kl_{n+1}] \to j_\bexc\!\Sym^k\Kl_{n+1}\to j_\bea\Sym^k\Kl_{n+1}\to 0
\]
and, since the kernel has punctual support by \cite[Prop.\,2.9.8]{Katz90}, we get the surjectivity.

Finally, we observe that the pairing \eqref{eq:pairing_dR_f_k}
is compatible with the induced actions of $\symgp_k$, which moreover acts trivially on the target $\CC$.
Taking the $\chi_n$-isotypic parts yields the desired self-duality
for $\coH^1_{\dR,\rmid}(\Gm,\Sym^k\Kl_{n+1})$.
The proof for $\Sym^k\wtKl_{n+1}$ is completely parallel.
\end{proof}

\subsection{The inverse Fourier transform of \texorpdfstring{$\Sym^k\wtKl_{n+1}$}{wtSymKln+1}}\label{subsec:FSymkwtKln}

Let $j_0\colon \Gm\hto\Afu$ denote the inclusion.
Recall that $\Sym^k\wtKl_{n+1}$ is the restriction to $\Gm$ of the Fourier transform of a regular holonomic $\cD_{\Afu_\tt}$-module. Applying inverse Fourier transformation to the exact sequence
\begin{equation}\label{eq:intermextwt}
0\to j_{0\bea}\Sym^k\wtKl_{n+1}\to j_{0\bast}\Sym^k\wtKl_{n+1}\to \wt C_{k, n}\to0,
\end{equation}
where $\wt C_{k, n}$ is supported at the origin, we thus get an exact sequence of regular holonomic~$\cD$\nobreakdash-modules on the dual affine line $\Afu_\tau$.
Set
\begin{equation}\label{eqn:defMtilde}
\wtM=\FT^{-1}( j_{0\bea}\Sym^k\wtKl_{n+1}).
\end{equation}
The endofunctor $\Pi\colon N \mapsto N \star j_{0\bexc}\cO_{\Gm}$ on the category of regular holonomic $\cD_{\Afu}$\nobreakdash-modules, where~$\star$ stands for additive $\ast$-convolution (\cf Section \ref{subsec:notationD}), is a projector onto the full subcategory of those with vanishing global de Rham cohomology, and the functors~$\Pi \circ \FT^{-1}$ and~$\FT^{-1}\circ j_{0\bast} j_0^\bast$ are canonically isomorphic (\cite[Prop.\,12.3.5]{Katz90}). Therefore, the equality
\begin{equation}\label{eqn:PiofMtilde}
\Pi(\wtM)=\FT^{-1}(j_{0\bast}\Sym^k\wtKl_{n+1})
\end{equation} holds, and we get an exact sequence of regular holonomic $\cD_{\Afu}$-modules
\begin{equation}\label{eq:wtM}
0\to\wtM\to\Pi(\wtM)\to\wtM'\to0,
\end{equation} where $\wtM'$ is constant (\ie a sum of copies of the trivial $\cD_{\Afu}$-module $\cO_{\Afu}$).

As recalled in Section \ref{subsec:FT_EMHS} of the appendix, the projector $\Pi$ lifts to a projector on the category $\MHM(\Afu)$ of mixed Hodge modules on the affine line $\Afu$, denoted in the same way.

\begin{prop}\label{prop:HwtM}
The exact sequence \eqref{eq:wtM} underlies an exact sequence
\[
0\to\wtM{}^\rH\to\Pi(\wtM^\rH)\to\wtM'{}^\rH\to0
\]
in the category $\MHM(\Afu)$. More precisely, $\wtM{}^\rH$ is the pure Hodge module $W_{kn}\Pi(\wtM^\rH)$ of weight~$kn$ and~$\wtM'{}^\rH$ is a mixed Hodge module of weights $\geq kn+1$, which is pure of weight $2k+1$ if $n=1$.
\end{prop}

\begin{proof} Set $U=\Gm^{kn}$ and recall the Laurent polynomial $g^{\boxplus k}\colon U \to \Afu_\tau$. We consider the associated Gauss\nobreakdash-Manin system $\SH^0g^{\boxplus k}_\bast\cO_U$ and its localized Fourier transform
$j_0^\bast\FT(\SH^0g^{\boxplus k}_\bast\cO_U)$. The second isomorphism of Proposition~\ref{prop:GTS} reads
\[
\wtKl_{n+1}^{\otimes k} \simeq j_0^\bast\FT(\SH^0g^{\boxplus k}_\bast\cO_U).
\]
The Laurent polynomial $g$ is convenient and non\nobreakdash-degenerate with respect to its Newton polytope, and hence so is $g^{\boxplus k}$. The argument of \cite[(3.6) \& (3.9)(c)]{D-L91} extends to the complex setting and shows that the cone of the natural morphism~$g^{\boxplus k}_\bexc\cO_U\to g^{\boxplus k}_\bast\cO_U$ has constant cohomology,\footnote{Indeed, since $h=g^{\boxplus k}$ is convenient and non-degenerate with respect to its Newton polytope, it can be extended to a proper morphism $h_\Sigma\colon Y_\Sigma\to\Afu$ on a smooth toroidal variety $Y_\Sigma$ such that $R^i(h_{\Sigma|\ov Y_\Sigma^\sigma})_*\CC$ is constant for each $i$ and each cone $\sigma$ of the regular fan $\Sigma$. Setting $D=Y_\Sigma\moins\Gm^{kn}$, it follows that $R^i(h_{\Sigma|D})_*\CC$ is constant for each $i$, and thus the assertion for $h$ reduces to that for $h_\Sigma$, which holds because $Y_\Sigma$ is smooth and $h_\Sigma$ is proper and locally acyclic except at a finite number of points.}
and hence the same holds for the kernel and the cokernel
of the induced morphism
\[
\cH^0g^{\boxplus k}_\bexc\cO_U\to \cH^0g^{\boxplus k}_\bast\cO_U.
\]
Letting $\cH^0g^{\boxplus k}_{\bea}\cO_U$ denote its image, it follows that the induced morphisms
\[
\bigl(\cH^0g^{\boxplus k}_\bexc\cO_U\bigr)[\partial_\taut^{-1}]\to\bigl(\cH^0g^{\boxplus k}_\bea\cO_U\bigr)[\partial_\taut^{-1}]\to(\cH^0g^{\boxplus k}_\bast\cO_U)[\partial_\taut^{-1}],
\]
after inverting $\partial_\taut$, are isomorphisms. We finally obtain a morphism
\[
\cH^0g^{\boxplus k}_\bea\cO_U\to\FT^{-1}(j_{0\bast}\wtKl_{n+1}^{\otimes k})
\]
whose kernel and cokernel are constant $\cD_{\Afu}$-modules.

Recall from Section \ref{subsec:notaMHM} that the mixed Hodge module~$\pQQ^\rH_U$ is pure of weight $kn=\dim U$ and that the associated perverse sheaf and filtered $\cD_U$-module are $\QQ_U[kn]$ and $\cO_U$ together with the trivial filtration jumping at the index~$0$. As mixed Hodge modules on $\Afu_\taut$, the proper pushforward $\cH^0\Hm g^{\boxplus k}_!\pQQ^\rH_U$ has weights~\hbox{$\leq kn$},
the pushforward $\cH^0\Hm g^{\boxplus k}_*\pQQ^\rH_U$
has weights~\hbox{$\geq kn$,} and the image~$\cH^0\Hm g^{\boxplus k}_{!*}\pQQ^\rH_U$ of the former in the latter is pure of weight~$kn$, by \hbox{\cite[(4.5.2)]{MSaito87}.} Away from its singularities, $\cH^0\Hm g^{\boxplus k}_{!*}\pQQ^\rH_U$ corresponds to a polarizable variation of pure Hodge structures of weight $kn-1$.

The symmetric group $\symgp_k$ acts on $\Gm^{kn}$ by permuting the index $j$ in the coordinates $y_{ji}$ and this action preserves $g^{\boxplus k}$. Therefore, $\symgp_k$ acts on $\cH^0g^{\boxplus k}_\bast\cO_U$ and hence on its Fourier transform~$\FT \cH^0g^{\boxplus k}_\bast\cO_U$ and its localized Fourier transform $j_0^\bast\FT \cH^0g^{\boxplus k}_\bast\cO_U$. Through the identification
\[
\Sym^k \wtKl_{n+1}\simeq j_0^\bast\FT(\cH^0g^{\boxplus k}_\bast\cO_U)^{\symgp_k,\,\chi_n},
\]
we obtain a morphism
\begin{equation}\label{eq:kappa!{*}}
(\cH^0g^{\boxplus k}_\bea\cO_U)^{\symgp_k,\,\chi_n}\to\FT^{-1}(j_{0\bast}\Sym^k\wtKl_{n+1})=\Pi(\wtM)
\end{equation}
inducing an isomorphism after inverting $\partial_\taut$, a property which implies the isomorphisms
\begin{equation}\label{eq:Pigbox}
\Pi(\wtM)\simeq\Pi((\cH^0g^{\boxplus k}_\bea\cO_U)^{\symgp_k,\,\chi_n})\simeq\Pi(\cH^0g^{\boxplus k}_\bea\cO_U)^{\symgp_k,\,\chi_n}.
\end{equation}

\begin{lemma}\label{lem:irreducibility}
The $\cD_{\Afu}$-module $\wt M$ is irreducible and is equal to the image of \eqref{eq:kappa!{*}}.
\end{lemma}

\begin{proof} Since $\Sym^k\wtKl_{n+1}$ is irreducible by Proposition \ref{prop:irreducibility}, so are $j_{0\bea}\Sym^k\wtKl_{n+1}$ and its inverse Fourier transform $\wt M$. Besides, the left-hand side $\wt M_0$ of \eqref{eq:kappa!{*}} underlies a pure Hodge module $\wt M_0^\rH$ on the affine line, which decomposes as a direct sum $\wt M_1^\rH\oplus\wt M_2^\rH$ in which the~$\cD$\nobreakdash-module $\wt M_1$ underlying $\wt M_1^\rH$ is the maximal constant submodule of $\wt M_0$. The image of~\eqref{eq:kappa!{*}} is thus isomorphic to the $\cD$-module $\wt M_2$ underlying $\wt M_2^\rH$. Since $\wt M_2^\rH$ is pure, $\wt M_2$ is a direct sum of irreducible $\cD$-modules, and hence $\FT(\wt M_2)$ as well. It follows that $\FT(\wt M_2)$ is an intermediate extension at the origin, in the sense that the equality $\FT(\wt M_2)=j_{0\dag+}j_0^+\FT(\wt M_2)$ holds. On the other hand, \eqref{eq:Pigbox} gives the equality $j_0^+\FT(\wt M_2)=\Sym^k\wtKl_{n+1}$, from which we get $\FT(\wt M_2)=j_{0\dag+}\Sym^k\wtKl_{n+1}=\FT(\wt M)$ as we wanted to show.
\end{proof}

The $\chi_n$-isotypic component $(\cH^0\Hm g^{\boxplus k}_{!*}\,\pQQ^\rH_U)^{\symgp_k,\,\chi_n}$ is a pure object of weight $kn$ of $\MHM(\Afu)$. Upon application of the projector $\Pi$, we obtain an object $\Pi(\cH^0\Hm g^{\boxplus k}_{!*}\, \pQQ^\rH_U)^{\symgp_k,\,\chi_n}$ of $\MHM(\Afu)$ whose underlying $\Cltaut$-module is $\Pi(\wt M)$. The image of the lift of \eqref{eq:kappa!{*}} to $\MHM(\Afu)$ is a pure Hodge module of weight $kn$ that lifts~$\wt M$ to $\MHM(\Afu)$ and is denoted by $\wt M^\rH$. Therefore, $\Pi(\cH^0\Hm g^{\boxplus k}_{!*}\,\pQQ^\rH_U)^{\symgp_k,\,\chi_n}=\Pi(\wt M^\rH)$ holds. We denote by $\wtM'{}^\rH$ the quotient object in $\MHM(\Afu)$.

It remains to check the weight properties. We first claim that the equality $\wt M^\rH=W_{kn}\Pi(\wt M^\rH)$ holds. Otherwise, the non-zero quotient object $W_{kn}\Pi(\wt M^\rH)/\wt M^\rH$, which is constant, would be a direct summand of $W_{kn}\Pi(\wt M^\rH)$, and hence $\Pi(\wt M)$ would have a non-zero constant submodule, which contradicts the vanishing of its global de Rham cohomology. It follows that $\wtM'{}^\rH$ has weights $\geq kn+1$ and it remains to show that it is pure of weight $2k+1$ if $n=1$. This is equivalent to showing that the nearby cycles at infinity $\psi_{1/\taut}\wtM'{}^\rH$ have weight $2k$ (note that~$\wtM'{}^\rH$ extends smoothly at infinity). By \cite[Prop.\,A.1]{M-T09}, although $\Pi(\wtM^\rH)$ is mixed, the weight filtration~$W$ on~$\psi_{1/\taut,1}\Pi(\wtM^\rH)$ is nevertheless identified with a shifted monodromy filtration, namely the monodromy filtration centered at~$k$. Moreover, $\psi_{1/\taut}\wtM'{}^\rH$ is identified with the primitive part~$\rP_{k}\gr_{2k}^W\psi_{1/\taut,1}\Pi(\wtM^\rH)$ (\cf \cite[proof of Prop.\,A.3]{M-T09}), and hence is pure~of~weight~$2k$, since the monodromy of $\psi_{1/\taut,1}\Pi(\wtM^\rH)$ has only one Jordan block, by Corollary~\ref{cor:wtM} below.
\end{proof}

\section{Motives of symmetric power moments of Kloosterman connections}\label{sec:basicKl}

The main goal of this section is to prove that the middle de Rham cohomology of the connection $\Sym^k\Kl_{n+1}$ is the de Rham realization of a Nori motive over $\QQ$ (see \cite{HMS17}). The expressions of the various cohomology spaces obtained in Section \ref{subsec:basicSymkKln} naturally exhibit them as the de Rham fibers of certain exponential mixed Hodge structures, and we begin by showing in Theorem~\ref{th:weightsMHM} that they are in fact usual mixed Hodge structures. Building on this result, we then obtain in Theorem~\ref{Thm:Kl_in_Mcl} a geometric description of the middle mixed Hodge structure in terms of a hypersurface \hbox{$\KM\subset\Gm^{kn}$} acted upon by the group $\symgp_k\times\mu_{n+1}$. This leads us to \emph{define} the middle motive $\Motive_k$ of $\Sym^k\Kl_{n+1}$ as
\begin{equation}\label{eq:Mk}
\Motive_k= \gr^W_{kn+1}[\NM_\cp^{kn-1}(\KM)(-1)]^{\symgp_k\times\mu_{n+1},\chi_n},
\end{equation}
where in the right-hand side $W_\bbullet$ stands for the weight filtration, which exists at the motivic level by \cite[Th.\,10.2.5]{HMS17}.
In the second part, we establish the counterpart of these results for the étale cohomology of symmetric powers of the $\ell$-adic Kloosterman sheaf and the rigid cohomology of symmetric powers of the Kloosterman $F$-isocrystal on $\Gm$ over a finite field. They turn out to be easier to prove because a full formalism of weights is at hand regardless of the ramification of the coefficients, whereas in the Hodge setting
we need to resort the interpretation of~$\Sym^k\Kl_{n+1}$ as the localized Fourier transform of a mixed Hodge module in order to remain within the framework of exponential mixed Hodge structures.

\subsection{The mixed Hodge structures attached to \texorpdfstring{$\Sym^k\Kl_{n+1}$}{SymkKln1}}\label{subsec:expmotSymkKln}

Throughout this section,
we work in the abelian category $\EMHS$
of exponential mixed Hodge structures,
which is a full subcategory of the category of mixed Hodge modules
over the affine line $\Afu$ (see Section~\ref{subsec:FT_EMHS} for a review). It contains the category of mixed Hodge structure as a full subcategory (Lemma~\ref{lem:weightsMHS}) and objects $\coH^r(U,f),$ $\coH_\cp^r(U,f),$ and $\coH_\rmid^r(U,f)$ associated with a regular function $f\colon U\to\Afu$ and an integer $r$,
whose de Rham fibers are the cohomology spaces $\coH_\dR^r(U,\fqq{f}),$ $\coH_{\dR,\cp}^r(U,\fqq{f}),$ and $\coH_{\dR,\rmid}^r(U,\fqq{f})$ respectively (see Section \ref{subsec:EMHS_Uf}).

Recall from~\eqref{eq:f_k} the Laurent polynomial $f_k$ on the torus $\Gm^{kn+1}$ with coordinates $(z, x_{ji})$ and its pullback $\wtf_k$ to the torus~$\Gm^{kn+1}$ with coordinates $(t, x_{ji})$ by the map $\tz \mapsto \tt^{n+1}$. The group $\symgp_k \times \mu_{n+1}$ acts on~$\Gm^{kn+1}= {\Gm}_{,\tt} \times (\Gm^n)^k$ by permutation of the copies of $\Gm^n$ and multiplication on the coordinate $\tt$. As this action leaves $\wtf_k$ invariant, the objects $\coH^{kn+1}(\Gm^{kn+1}, \wtf_k)$ and $\coH^{kn+1}_{\cp}(\Gm^{kn+1}, \wtf_k)$ of the category $\EMHS$ inherit an action of $\symgp_k \times \mu_{n+1}$.

For $? = \emptyset, \cp, \rmid$, we define the objects
\begin{equation*}\label{def:motSymkKln}
\begin{aligned}
\coH_?^1(\Gm, \Sym^k\Kl_{n+1})&=\coH_?^{kn+1}(\Gm^{kn+1}, f_k)^{\symgp_k,\,\chi_n}=\coH_?^{kn+1}(\Gm^{kn+1}, \wtf_k)^{\symgp_k \times \mu_{n+1},\,\chi_n}, \\ \coH_?^1(\Gm, \Sym^k\wtKl_{n+1})&=\coH_?^{kn+1}(\Gm^{kn+1}, \wtf_k)^{\symgp_k,\,\chi_n}
\end{aligned}
\end{equation*} of the category $\EMHS$, the de Rham fibers of which are nothing but $\coH_{\dR, ?}^1(\Gm, \Sym^k\Kl_{n+1})$ and~$\coH_{\dR, ?}^1(\Gm, \Sym^k\wt\Kl_{n+1})$ in view of Proposition~\ref{prop:HdRSymKln} and Corollary~\ref{cor:vanishingSymKln}.

\begin{thm}\label{th:weightsMHM} The exponential mixed Hodge structures
\[
\coH^1(\Gm,\Sym^k\Kl_{n+1}), \quad \coH^1_{\rc}(\Gm,\Sym^k\Kl_{n+1}), \quad \text{and} \quad \coH^1_{\rmid}(\Gm,\Sym^k\Kl_{n+1})
\]
are usual mixed Hodge structures of weights $\geq kn+1$, $\leq kn+1$, and $kn+1$ respectively, and the natural morphisms between them are morphisms of mixed Hodge structures.

Moreover, the induced map
\begin{equation}\label{eqn:isom-graded}
\gr^W_{kn+1}\coH^1_{\rc}(\Gm,\Sym^k\Kl_{n+1}) \longrightarrow \gr^W_{kn+1}\coH^1(\Gm,\Sym^k\Kl_{n+1})
\end{equation} is an isomorphism, and $\coH^1_{\rmid}(\Gm,\Sym^k\Kl_{n+1})$ is equal to its image. This pure Hodge structure of weight $kn+1$ is equipped with a $(-1)^{kn+1}$-symmetric pairing
\begin{equation}\label{eq:pairing-motive}
\coH^1_{\rmid}(\Gm,\Sym^k\Kl_{n+1}) \otimes \coH^1_{\rmid}(\Gm,\Sym^k\Kl_{n+1}) \to \QQ(-kn-1).
\end{equation}
The same results hold for $\Sym^k\wtKl_{n+1}$.
\end{thm}

\begin{proof}
Once we know that $\coH^1_{\rc}(\Gm,\Sym^k\Kl_{n+1})$ and $\coH^1(\Gm,\Sym^k\Kl_{n+1})$ are both mixed Hodge structures, the natural morphism between them is a morphism of Hodge structures since $\MHS$ is a full subcategory of $\EMHS$ by Lemma \ref{lem:weightsMHS}. Its image $\coH^1_{\rmid}(\Gm,\Sym^k\Kl_{n+1})$ is therefore a usual mixed Hodge structure as well. The pairing \eqref{eq:pairing-motive} is induced by Poincaré duality and the involution $\iota_{n+1}$ as in Corollary \ref{cor:vanishingSymKln} for the de Rham fiber.

Since the exponential mixed Hodge structures attached to $\Sym^k\Kl_{n+1}$ and $\Sym^k\wtKl_{n+1}$ are defined as the $\chi_n$-isotypic components of $\coH_?^{kn+1}(\Gm^{kn+1}, \wtf_k)$ for the action of $\symgp_k \times \mu_{n+1}$ and~$\symgp_k$ respectively, it suffices to prove the result for the latter. For $\coH^{kn+1}(\Gm^{kn+1}, \wtf_k)$ and~$\coH_{\cp}^{kn+1}(\Gm^{kn+1}, \wtf_k)$, we apply Theorems~\ref{th:EMHSMHStg}\eqref{th:EMHSlocMHS2} and~\ref{th:EMHSMHStg}\eqref{th:EMHSlocMHS3} respectively with $V=\Gm^{kn}$, $M_V^\rH=\cO_V^\rH$, $f=\wtf_k$, and~$r=kn+1$. The weight properties follow from Proposition \ref{prop:weightsHdUf}, taking Lemma \ref{lem:weightsMHS} into account.

We will prove that the analogue of the map \eqref{eqn:isom-graded} for $\Sym^k\wtKl_{n+1}$ is an isomorphism and deduce the result for $\Sym^k\Kl_{n+1}$ by taking $\mu_{n+1}$-invariants. According to Proposition \ref{prop:weightsHdUf}, it suffices to establish the equality
\begin{equation}\label{eqn:purityequality}
\coH^1_\rmid(\Gm,\Sym^k\wtKl_{n+1})=W_{kn+1}\coH^1(\Gm,\Sym^k\wtKl_{n+1}).
\end{equation}
The inclusion $\subset$ follows from the purity of $\coH^1_\rmid(\Gm,\Sym^k\wtKl_{n+1})$. Since we only deal with weight properties, we may as well work with the de~Rham fibers of the corresponding exponential mixed Hodge structures. Recall that the $\cD_{\Afu}$\nobreakdash-module $\wt M$ defined in \eqref{eqn:defMtilde} underlies a pure Hodge module $\wt M^\rH$ of weight $kn$ by Proposition \ref{prop:HwtM}. Letting $j_\infty\colon\Afu_t\hto\PP^1_t$ denote the inclusion, the map $\coH^1_{\dR,\rmid}(\Gm,\Sym^k\wtKl_{n+1})\hto\coH^1_\dR(\Gm,\Sym^k\wtKl_{n+1})$ decomposes as
\[
\coH^1_\dR(\PP^1_t,j_{\infty\dag+}\FT\wt M)\hto\coH^1_\dR(\Afu_t,\FT\wt M)\hto\coH^1_\dR(\Afu_t,\FT\Pi(\wt M)).
\]
From Proposition \ref{prop:HwtM} and Corollary \ref{cor:weightssimple}\eqref{cor:weightssimple1}, we deduce that the second inclusion induces an equality between the subspaces $W_{kn+1}$. Corollary \ref{cor:weightssimple}\eqref{cor:weightssimple3} identifies the weight filtration of $\coH^1_\dR(\Afu_t,\FT\wt M)$ with that of
\[
\coker[\wt\rN_\tau\colon\psi_{\tau,1}\wt M^\rH\to \psi_{\tau,1}\wt M^\rH(-1)].
\]
Recall also that the weight filtration on $\psi_{\tau,1}\wt M^\rH$ is the monodromy filtration associated with~$\wt\rN_\tau$ centered at $kn-1$ ($kn$ is the pure weight of $\wt M^\rH$). Then $\gr^W\coker\wt\rN_\tau$ is the direct sum of the primitive parts of the Lefschetz decomposition of $\gr^W\psi_{\tau,1}\wt M^\rH(-1)$. The weight of the component corresponding to the primitive part $\rP_0$ (Jordan blocks of $\wt\rN_\tau$ of size one) is $kn+1$, and the weight for the component corresponding to $\rP_i$ ($i\geq1$) is $kn+1+i$.

Proving the equality \eqref{eqn:purityequality} amounts therefore to proving the equality of dimensions
\[
\dim\coH^1_\dR(\PP^1_t,j_{\infty\dag+}\FT\wt M)=\dim \rP_0.
\]
It will be easier to deal with codimensions. Note that the cokernel of $j_{\infty\dag+}\FT\wt M\hto j_{\infty+}\FT\wt M$ is a $\cD_{\PP^1_t}$-module supported at $t=\infty$, and hence takes the form $C[\partial_{t'}]$ for some finite-dimensional vector space $C$, where $t'$ denotes the coordinate $1/t$ at infinity. Then $\dim C$ is the codimension of $\coH^1_\dR(\PP^1_t,j_{\infty\dag+}\FT\wt M)$ in $\coH^1_\dR(\Afu_t,\FT\wt M)=\coH^1_\dR(\PP^1_t,j_{\infty+}\FT\wt M)$, and we aim at proving the equality
\begin{equation}\label{eq:CPi}
\dim C=\sum_{i\geq1}\dim\rP_i.
\end{equation}
Since $\wt M$ is irreducible (Lemma \ref{lem:irreducibility}), it is an intermediate extension at the origin. Therefore, $\sum_{i\geq1}\dim\rP_i$ is the number of Jordan blocks of $\wt\rN_\tau$ acting on the vanishing cycle space
\[
\phi_{\tau,1}\wt M=\image[\wt\rN_\tau\colon\psi_{\tau,1}\wt M\to \psi_{\tau,1}\wt M]
\]
and, with an argument already used for $\psi_{\tau,1}\wt M$, the following equality holds:
\[
\sum_{i\geq1}\dim\rP_i=\dim\coker[\wt\rN_\tau\colon\phi_{\tau,1}\wt M\to\phi_{\tau,1}\wt M].
\]

On the other hand, the module $C[\partial_{t'}]$ can be computed by replacing $j_{\infty+}\FT\wt M$ with its formalization at infinity $(j_{\infty+}\FT\wt M)^\wedge=\CC[\![t']\!]\otimes_{\CC[t']} j_{\infty+}\FT\wt M$. Decomposing the latter as the direct sum of its regular and irregular parts and taking the equalities
\[
(j_{\infty+}\FT\wt M)^\wedge_\irr=(j_{\infty\dag}\FT\wt M)^\wedge_\irr=(j_{\infty\dag+}\FT\wt M)^\wedge_\irr
\]
into account, we see that the irregular part does not contribute to $C[\partial_{t'}]$, hence an isomorphism
\[
C\simeq \coker\bigl[\phi_{t',1}(j_{\infty\dag+}\FT\wt M)^\wedge_\reg\hto \phi_{t',1}(j_{\infty+}\FT\wt M)^\wedge_\reg\bigr].
\]
The target of this morphism is isomorphic to $\psi_{t',1}(j_{\infty+}\FT\wt M)^\wedge_\reg$ and, by definition of the intermediate extension, the source is isomorphic to the image of
\[
\wh\rN_{t'}\colon\psi_{t',1}(j_{\infty+}\FT\wt M)^\wedge_\reg\to\psi_{t',1}(j_{\infty+}\FT\wt M)^\wedge_\reg,
\]
the morphism above being the inclusion. Therefore, we get the equality
\[
\dim C=\dim\coker\Bigl[\wh\rN_{t'}\colon\psi_{t',1}(j_{\infty+}\FT\wt M)^\wedge_\reg\to\psi_{t',1}(j_{\infty+}\FT\wt M)^\wedge_\reg\Bigr].
\]
Finally, the stationary phase formula identifies $(\phi_{\tau,1}\wt M,\wt\rN_\tau)$ with $(\psi_{t',1}(j_{\infty+}\FT\wt M)^\wedge_\reg,\wh\rN_{t'})$ since $(j_{\infty+}\FT\wt M)^\wedge_\reg$ is the microlocalization of $\wt M$ at the origin (\cf \eg \cite[Prop.\,2.3]{Bibi96bb}). This implies the equality \eqref{eq:CPi}, thus concluding the proof.
\end{proof}

\begin{remark} For $k=n=1$, there are isomorphisms
\[
\coH^1(\Gm, \Kl_2) \simeq \coH^2(\Gm^2, x+y) \simeq \QQ(-2) \quad\text{and}\quad \coH^1_\cp(\Gm, \Kl_2) \simeq \coH^2_{\cp}(\Gm^2, x+y) \simeq \QQ(0)
\]
by the change of variables $(x, z) \mapsto (x, y)=(x, z/x)$ and Example \ref{exem:Gmx}. Hence, the middle Hodge structure $\coH_{\rmid}^1(\Gm, \Kl_2)$ vanishes.
\end{remark}

\begin{thm}\label{Thm:Kl_in_Mcl} Assume $kn \geq 2$. Consider the hypersurface $\KM \subset \Gm^{kn}$ over $\QQ$ defined as the zero locus of the Laurent polynomial
\begin{equation}\label{eq:g_boxplus_k}
g^{\boxplus k}(y)=\sum_{j=1}^k \biggl(\sum_{i=1}^n y_{ji} + \prod_{i=1}^n \sfrac{1}{y_{ji}} \biggr),
\end{equation}
on which the group $\symgp_k\times\mu_{n+1}$ acts
by permuting the indices $j$ in $y_{ji}$
and by $y_{ji} \mto \zeta^{-1}y_{ji}$ for each $\zeta \in \mu_{n+1}$. There is a canonical isomorphism of pure Hodge structures
\begin{align*}
\coH^1_{\rmid}(\Gm,\Sym^k\Kl_{n+1})&\cong\gr_{kn+1}^W[\coH^{kn-1}_{\cp}(\KM)(-1)]^{\symgp_k \times \mu_{n+1},\,\chi_n}\\
&\cong \gr_{kn+1}^W\coH^{kn+1}_{\KM}(\Gm^{kn})^{\symgp_k \times \mu_{n+1},\,\chi_n}.
\end{align*} If $\KM$ is smooth, the last term is also isomorphic to $\gr_{kn+1}^W[\coH^{kn-1}(\KM)(-1)]^{\symgp_k \times \mu_{n+1},\,\chi_n}.$
\end{thm}

\begin{proof} Consider the open immersion $(\Gm^{kn+1},\wtf_k) \subset (\Afu_{\tt}\times\Gm^{kn},\wtf_k)$ of pairs compatible with the action of $\symgp_k\times\mu_{n+1}$. On noting the isomorphism
\[
(\Afu_{\tt}\times\Gm^{kn},\wtf_k)\setminus(\Gm^{kn+1},\wtf_k) \simeq (\Gm^{kn}, 0),
\]
we obtain a diagram with exact rows and a commutative square
\[
\xymatrix@C=.35cm{
\cdots\ar[r]&\coH^{kn-1}(\Gm^{kn})(-1)\ar[r]& \coH^{kn+1}(\Afu_{\tt}\times\Gm^{kn},\wtf_k) \ar[r]&\coH^{kn+1}(\Gm^{kn+1},\wtf_k)\ar[r]&\coH^{kn}(\Gm^{kn})(-1)\ar[r]&0\\
\cdots&\coH_\cp^{kn+1}(\Gm^{kn})\ar[l]& \coH_\cp^{kn+1}(\Afu_{\tt}\times\Gm^{kn},\wtf_k)\ar[u]\ar@{}[ur]|\circlearrowright\ar[l]& \coH_\cp^{kn+1}(\Gm^{kn+1},\wtf_k)\ar[u]\ar[l]&\ar[l]\coH_\cp^{kn}(\Gm^{kn})&\ar[l] 0
}
\]
as in Example \ref{exem:locexactsequence} from the appendix. Exactness at the rightmost term of both rows follows from the fact that~$\Afu_{\tt}\times\Gm^{kn}$ is an affine variety of dimension $kn+1$. In the above diagram, the four terms involving $\Gm^{kn}$ on the corners are pure of weights~$2kn$ (top left), $2kn+2$ (top right), $2$ (bottom left), and $0$ (bottom right). Moreover, by Proposition \ref{prop:weightsHdUf} the two middle upper terms have weights $\geq kn+1$ and the middle lower terms have weights $\leq kn+1$. Since $W_\bbullet$ is an exact functor, the commutative square yields a commutative square with isomorphisms
\[
\xymatrix{
\gr_{kn+1}^W\coH^{kn+1}(\Afu_{\tt}\times\Gm^{kn},\wtf_k) \ar[r]^-\sim&\gr_{kn+1}^W\coH^{kn+1}(\Gm^{kn+1},\wtf_k)\\
\gr_{kn+1}^W\coH_\cp^{kn+1}(\Afu_{\tt}\times\Gm^{kn},\wtf_k)\ar[u]\ar@{}[ur]|\circlearrowright&\gr_{kn+1}^W\coH_\cp^{kn+1}(\Gm^{kn+1},\wtf_k)\ar[u]\ar[l]_-\sim
}
\]
of pure Hodge structures of weight $kn+1$ that are equivariant for the action of $\symgp_k\times\mu_{n+1}$. After taking the $\chi_n$-isotypic components, the right vertical map is an isomorphism by Theorem~\ref{th:weightsMHM}, with source and target isomorphic to $\coH^1_{\rmid}(\Gm,\Sym^k\Kl_{n+1})$. The same then holds for the remaining map.

Recall that the change of variables $(t, x) \mapsto (t, y)$ on $\Gm^{kn+1}$ given by $x_{ji}=\tt y_{ji}$ transforms the Laurent polynomial $\wtf_k(t, x)$ into $t g^{\boxplus k}(y)$. On the new coordinates, $\symgp_k$ permutes the indices~$j$ in $y_{ji}$ and leaves $t$ invariant, while $\zeta \in \mu_{n+1}$ acts through $\tt \mapsto \zeta \tt$ and
$y_{ji} \mapsto \zeta^{-1} y_{ji}$.
Applying Example~\ref{exem:cK}, we obtain isomorphisms
\[
\coH_\cp^{kn+1}(\Afu\times\Gm^{kn},\wtf_k)
\simeq \coH_\cp^{kn+1}(\Afu\times \KM)\simeq \coH^2_{\cp}(\Afu) \otimes \coH^{kn-1}_{\cp}(\KM)\simeq \coH^{kn-1}_{\cp}(\KM)(-1)
\]
and these isomorphisms are equivariant since $\mu_{n+1}$ acts trivially on $\coH^2_{\cp}(\Afu)$. Similarly, Example~\ref{exem:cK} gives an equivariant isomorphism
\[
\coH^{kn+1}(\Afu\times\Gm^{kn},\wtf_k)\simeq \coH^{kn+1}_{\KM}(\Gm^{kn}),
\]
and the right-hand side is also isomorphic to $\coH^{kn-1}(\KM)(-1)$ if $\KM$ is smooth.
\end{proof}

\begin{remark}\label{rem:Ksmooth}
The hypersurface $\KM$ has at worst isolated singularities, its $\ol\QQ$-singular points being those with coordinates $y_{ji} = \zeta_j$ for all $1 \leq j \leq k$ and $1 \leq i \leq n$, where $\zeta_j \in \mu_{n+1}$ are roots of unity
satisfying $\sum_{j=1}^k \zeta_j = 0$.
(This is reminiscent of the computation of the Swan conductor at infinity of the analogue of $\Sym^k\Kl_{n+1}$ over finite fields in \cite[Th.\,3.1]{FW05}.)
For~example, if $n+1$ is a prime number,
then $\KM$ is singular if and only if $k$ is a multiple of~$n+1$; this includes the case $n=1$, in which~$\KM$ is singular if and only if $k$ is even.
\end{remark}

\subsection{The case of characteristic \texorpdfstring{$p$}{p}}

In this section, we prove an analogue of Theorem \ref{Thm:Kl_in_Mcl} for the étale and rigid cohomology of symmetric powers of Kloosterman sheaves over finite fields.

\subsubsection{\'Etale cohomology of symmetric powers of $\ell$-adic Kloosterman sheaves}
\label{sect:etale_char_p}

Let $p$ and $\ell$ be distinct prime numbers, $\ol\FF_p$ an algebraic closure of the finite field $\FF_p$, and $\ol{\QQ}_\ell$ an algebraic closure of the field of $\ell$-adic numbers ${\QQ}_\ell$. All weights of $\ell$-adic sheaves are understood to be considered with respect to a fixed isomorphism from $\ol\QQ_\ell$ to $\CC$.
Let $\zeta$ be a primitive $p$-th root of unity in $\ol{\QQ}_\ell$
and $\psi\colon \FF_p \to \QQ_\ell(\zeta)^\times$ a non\nobreakdash-trivial additive character. The Artin\nobreakdash-Schreier sheaf~$\AS_\psi$ is the rank\nobreakdash-one lisse \'etale sheaf with coefficients in $\QQ_\ell(\zeta)$ on the affine line $\Afu_{\FF_p}$ on which geometric Frobenius acts as multiplication by $\psi(\tr_{\FF_q/\FF_p}(y))$ at each closed point $y$ defined over $\FF_q$. Given a regular function $h \colon X \to \Afu$ on a smooth algebraic variety $X$ over~$\FF_p$, we set $\AS_{\psi(h)}= h^*\AS_\psi$.

Recall the function $f \colon \Gm^{n+1} \to \Afu$ from \eqref{eqn:functionf} and the projection $\pi \colon \Gm^{n+1} \to {\Gm}_{,\tz}$ to the first coordinate, which we now view as defined over $\FF_p$. For each $n \geq 1$, we define a complex of $\ell$-adic sheaves with coefficients in $\QQ_\ell(\zeta)$ on $\Gm$ in a similar way as we did in \eqref{eq:Klpf}, namely:
\[
\Kl_{n+1}=\Rder\pi_\ast \AS_{\psi(f)}[n].
\]
Deligne proved in \cite[Th.\,7.4\,\&\,7.8]{Deligne77SGA} that $\Kl_{n+1}$ is a lisse sheaf of rank $n+1$ concentrated in degree zero that is pure of weight $n$; we shall call it the \emph{$\ell$-adic Kloosterman sheaf}. At each closed point $z\in \Gm(\FF_q)$, the action of geometric Frobenius has trace $(-1)^n \Kl_{n+1}(z;q)$, where
\[
\Kl_{n+1}(z; q)=
\sum_{x_1,\dots, x_n \in \FF_q^\times}
\psi(\mathrm{tr}_{\FF_q \slash \FF_p}
	(x_1+\cdots +x_n+\sfrac{z}{x_1\cdots x_n}))
\]
is the Kloosterman sum in many variables generalizing \eqref{eq:Kl_sum}. Similarly, we define
\[
\wtKl_{n+1}=[n+1]^\ast \Kl_{n+1}\simeq \Rder\wtpi_\ast \AS_{\psi(\wt f)}[n],
\]
where $\wtf(\tt, x)=f(t^{n+1}, x)$ and $\wtpi$ is the projection $(\tt, x) \mapsto \tt$. Note that $\wtKl_{n+1}$ is also concentrated in degree zero. If $p$ divides $n+1$, say $n+1 = p^rm$ with $(p,m) = 1$, the cover~$[n+1]$ of~$\Gm$ factors as the composition of the \'etale cyclic cover $[m]$
and the purely inseparable cover~$[p^r]$. The latter induces an equivalence of categories of étale covers (\cf \eg \cite[Prop.\,3.16]{FK88}), and hence does not change the cohomology. In that case, the action of $\mu_{n+1}$ on $\wtKl_{n+1}$ factors through $\mu_m$.

The proof of Proposition \ref{prop:GTS} adapts verbatim to this setting. Indeed, $\wtKl_{n+1}$ is the localized Fourier transform $j_0^\ast \FT_\psi$ of $\Rder g_\ast \QQ_\ell(\zeta)_{\Gm^{kn}}[n]$, which has thus non-constant cohomology only in degree~zero. From the fact that Fourier transformation exchanges tensor product and additive convolution (\cf \eg \cite[Prop.\,1.2.2.7]{Laumon}), we obtain isomorphisms $\wtKl_{n+1}^{\otimes k} \simeq \Rder\wtpi_{k\ast}\AS_{\psi(\wtf_k)}$ and, arguing as in Corollary~\ref{cor:KummethForTensor} and Proposition \ref{prop:HdRSymKln},
\begin{displaymath}
\coH_{\et,?}^1(\mathbb{G}_{\mathrm{m}, \ol\FF_p}, \Sym^k\Kl_{n+1})
\simeq \coH_{\et,?}^{kn+1}(\mathbb{G}^{\scriptstyle{kn+1}}_{\mathrm{m}, \ol\FF_p}, \AS_{\psi(\wtf_k)})^{
	\symgp_k\times\mu_{n+1},\chi_n}
\end{displaymath}
for $? = \emptyset,\cp,\rmid$. The fact that $\Sym^k\Kl_{n+1}$ has cohomology concentrated in degree one still holds for odd $p$ or if $p=2$ and $n$ is odd, since in those cases the geometric monodromy group of $\Kl_{n+1}$ is either $\mathrm{SL}_{n+1}$ or $\mathrm{Sp}_{n+1}$ over $\QQ_\ell(\zeta)$ by \cite[Th.\,11.1]{KatzGK}. However, if $p=2$ and $n$ is even, the geometric monodromy group is either $\mathrm{SO}_{n+1}$ or $G_2$ and the symmetric powers of the standard representation may contain a copy of the trivial representation.

Besides, the natural morphism from étale cohomology with compact support to usual étale cohomology fits into a $\Gal(\ol\FF_p/\FF_p)$-equivariant long exact sequence
\begin{equation}\label{eqn:exactsequenceinertia}
\begin{aligned}
&(\Sym^k\Kl_{n+1})^{I_0}_{\ol\eta} \oplus (\Sym^k\Kl_{n+1})^{I_\infty}_{\ol\eta} \to \coH^1_{\et, \cp}(\mathbb{G}_{\mathrm{m}, \ol\FF_p}, \Sym^k\Kl_{n+1}) \\
&\to \coH^1_{\et}(\mathbb{G}_{\mathrm{m}, \ol\FF_p}, \Sym^k\Kl_{n+1}) \to ((\Sym^k\Kl_{n+1})_{\ol\eta})_{I_0}(-1) \oplus ((\Sym^k\Kl_{n+1})_{\ol\eta})_{I_\infty}(-1),
\end{aligned}
\end{equation} where $I_0$ and $I_\infty$ stand for the inertia groups at $0$ and $\infty$ acting on a geometric generic fiber of~$\Sym^k\Kl_{n+1}$ (\cf \eg \cite[2.0.7]{KatzGK}). Since this sheaf is pure of weight $kn$, it follows from Weil~II \hbox{\cite[Lem.\,1.8.1]{WeilII}} that the leftmost term in the exact sequence is mixed of weights $\leq kn$. By duality, the rightmost term is then mixed of weights $\geq kn+2$, hence an isomorphism
\[
\gr^W_{kn+1}\coH^1_{\et, \cp}(\mathbb{G}_{\mathrm{m}, \ol\FF_p}, \Sym^k\Kl_{n+1}) \stackrel{\sim}{\to} \gr^W_{kn+1}\coH^1_{\et}(\mathbb{G}_{\mathrm{m}, \ol\FF_p}, \Sym^k\Kl_{n+1}),
\]
the image of which is the étale middle cohomology (this is also proved in \cite[Lem.\,2.5.4]{Yun15}).

Let $\KM \subset \Gm^{kn}$ be the hypersurface over $\FF_p$ defined by the same equation as in \eqref{eq:g_boxplus_k}. From the Gysin long exact sequence for étale cohomology and the localization exact sequence for étale cohomology with compact support, we get as in the proof of Theorem \ref{Thm:Kl_in_Mcl} an isomorphism
\begin{align*}
\coH_{\et,\rmid}^1(\mathbb{G}_{\mathrm{m}, \ol\FF_p}, \Sym^k\Kl_{n+1}) &\cong \gr^W_{kn+1}\coH_{\et,\cp}^{kn+1}((\Afu\times\Gm^{kn})_{\ol\FF_p}, \AS_{\psi(\wtf_k)})^{
	\symgp_k\times\mu_{n+1},\chi_n} \\ &\cong \gr^W_{kn+1}\coH_{\et}^{kn+1}((\Afu\times\Gm^{kn})_{\ol\FF_p}, \AS_{\psi(\wtf_k)})^{
	\symgp_k\times\mu_{n+1},\chi_n}.
\end{align*}
Finally, taking into account that the function $\tt g^{\boxplus k}$ vanishes at $\Afu \times \KM$, the localization exact sequence associated with the closed immersion $\Afu \times \KM \hookrightarrow \Afu \times \Gm^{kn}$ reads
\begin{align*}
\cdots\to &\coH_{\et,\cp}^{kn+1}((\Afu\times(\Gm^{kn}\setminus\KM))_{\ol\FF_p}, \AS_{\psi(tg^{\boxplus k})}) \to \coH_{\et,\cp}^{kn+1}((\Afu\times\Gm^{kn})_{\ol\FF_p}, \AS_{\psi(tg^{\boxplus k})})\\
&\to \coH_{\et,\cp}^{kn+1}((\Afu\times\KM)_{\ol\FF_p}, \QQ_\ell(\zeta)) \to \coH_{\et,\cp}^{kn+2}((\Afu\times(\Gm^{kn}\setminus\KM))_{\ol\FF_p}, \AS_{\psi(tg^{\boxplus k})}) \to \cdots
\end{align*} Since $g^{\boxplus k}$ is invertible on $\Gm^{kn}\setminus\KM$, the change of variables $(t, y_{ji}) \mapsto (s, y_{ji})=(tg^{\boxplus k}, y_{ji})$ and the K\"unneth formula yield an isomorphism
\begin{align*}
\coH_{\et,\cp}^{r}((\Afu\times(\Gm^{kn}\setminus\KM))_{\ol\FF_p}, \AS_{\psi(tg^{\boxplus k})}) &\cong \coH_{\et,\cp}^{r}((\Afu\times(\Gm^{kn}\setminus\KM))_{\ol\FF_p}, \AS_{\psi(s)}) \\ &\cong \bigoplus_{a+b=r} \coH_{\et,\cp}^{a}(\Afu_{\ol\FF_p}, \AS_{\psi}) \otimes \coH_{\et,\cp}^{b}((\Gm^{kn}\setminus\KM)_{\ol\FF_p}, \QQ_\ell(\zeta)),
\end{align*} which shows the vanishing of these groups in all degrees $r$ since the Artin-Schreier sheaf $\AS_{\psi}$ has trivial cohomology (with compact support). Therefore,
\[
\coH_{\et,\cp}^{kn+1}((\Afu\times\Gm^{kn})_{\ol\FF_p}, \AS_{\psi(tg^{\boxplus k})}) \cong \coH_{\et,\cp}^{kn+1}((\Afu\times\KM)_{\ol\FF_p}, \QQ_\ell(\zeta)) \cong \coH_{\et,\cp}^{kn-1}(\KM_{\ol\FF_p}, \QQ_\ell(\zeta))(-1).
\]
A similar argument involving the localization sequence for usual étale cohomology gives
\[
\coH_{\et}^{kn+1}((\Afu\times\Gm^{kn})_{\ol\FF_p}, \AS_{\psi(tg^{\boxplus k})}) \cong \coH_{\et, \KM_{\ol\FF_p}}^{kn+1}(\Gm^{kn}, \QQ_\ell(\zeta)),
\]
and the right-hand side is isomorphic to $\coH_{\et}^{kn-1}(\KM_{\ol\FF_p}, \QQ_\ell(\zeta))(-1)$ if $\KM$ is smooth.
Passing to the $\chi_n$-isotypic components and pulling everything together, we have thus proved:

\begin{thm}\label{thm:etalerealisation} Assume $kn \geq 2$. There is a canonical isomorphism of $\Gal(\ol\FF_p /\FF_p)$-modules
\begin{align*}
\coH_{\et,\rmid}^1(\mathbb{G}_{\mathrm{m}, \ol\FF_p}, \Sym^k\Kl_{n+1})&\cong\gr_{kn+1}^W[\coH^{kn-1}_{\et, \cp}(\KM_{\ol\FF_p}, \QQ_\ell(\zeta))(-1)]^{\symgp_k \times \mu_{n+1},\,\chi_n}\\
&\cong \gr_{kn+1}^W\coH_{\et, \KM_{\ol\FF_p}}^{kn+1}(\Gm^{kn}, \QQ_\ell(\zeta))^{\symgp_k \times \mu_{n+1},\,\chi_n}.
\end{align*}
If $\KM$ is smooth, the last term is also isomorphic to $\gr_{kn+1}^W[\coH^{kn-1}_{\et}(\KM_{\ol\FF_p}, \QQ_\ell(\zeta))(-1)]^{\symgp_k \times \mu_{n+1},\,\chi_n}$.
\end{thm}

\subsubsection{Rigid cohomology of symmetric powers of Kloosterman $F$-isocrystals}
\label{sect:rigid_char_p}

Let $\ol \QQ_p$ be an algebraic closure of $\QQ_p$ and let $\varpi \in \ol\QQ_p$ be an element satisfying $\varpi^{p-1} = -p$. As such elements are in bijection with primitive $p$-th roots of unity in $\ol\QQ_p$, this amounts to choosing a non\nobreakdash-trivial $\ol\QQ_p$-valued additive character on $\FF_p$. The analogue of the Artin-Schreier sheaf is then Dwork's $F$-isocrystal $\AS_\varpi$, which is the rank-one connection $\rd+\varpi\rd z$ with Frobenius structure \hbox{$\exp(\varpi(z^p-z))$} on the overconvergent structure sheaf of $\Afu$ over~\hbox{$K=\QQ_p(\varpi)$}. Given a regular function $h \colon X \to \Afu$ on a smooth algebraic variety $X$ over $\FF_p$, we set $\AS_{\varpi h}= h^*\AS_\varpi$. We refer the reader to Kedlaya's paper \cite{Ked06} for a summary of the properties of rigid cohomology of overconvergent isocrystals that are used below, namely the fact that these groups carry a weight filtration for which the main results of Weil II hold.

The \emph{Kloosterman $F$-isocrystal} $\Kl_{n+1}$ is the overconvergent $F$-isocrystal
\[
\Kl_{n+1}=\Rder \pi_{\mathrm{rig}\ast}\AS_{\varpi f}[n],
\]
which has rank $n+1$ and is pure of weight $n$ (see \cite[\S 1, namely 1.5]{Crew_Kl} for a detailed discussion). Similarly to the $\ell$-adic case, the interpretation of $\wtKl_{n+1}=[n+1]^\ast \Kl_{n+1}$ as a localized Fourier transform $j_0^\ast\FT_\varpi$ yields isomorphisms
\begin{equation}\label{eqn:isomrigidmiddlebis}
\coH_{\rig,?}^1(\Gm/K, \Sym^k\Kl_{n+1})
= \coH_{\rig,?}^{kn+1}(\Gm^{kn+1}/K, \AS_{\varpi\wtf_k})^{
	\symgp_k\times\mu_{n+1},\chi_n}
\end{equation}
for $? = \emptyset,\cp,\rmid$, and combining the analogue of the long exact sequence \eqref{eqn:exactsequenceinertia} relating rigid cohomology and rigid cohomology with compact support (\cf \cite[(2.5.1)]{Ked06}) with the fact that the local contributions from $0$ and $\infty$ are mixed of weights $\leq kn$ by \cite[Prop.\,5.1.4]{Ked06}, we get
\begin{equation}\label{eqn:isomrigidmiddle}
\coH_{\rig,\rmid}^1(\Gm/K, \Sym^k\Kl_{n+1})
\cong \gr^W_{kn+1}\coH_{\rig,\cp}^{kn+1}(\Gm^{kn+1}/K, \AS_{\varpi\wtf_k})^{
	\symgp_k\times\mu_{n+1},\chi_n}.
\end{equation} Finally, the same argument as in the $\ell$-adic case, considering the localization exact sequences for rigid cohomology and rigid cohomology with compact support (\cite[Prop.\,8.2.18(ii)]{LeStum07}) and the Gysin isomorphism \cite[Th.\,4.1.1]{Tsuzuki}, allows one to express the right-hand side of \eqref{eqn:isomrigidmiddle} in terms of the hypersurface $\KM$, obtaining
\begin{align*}
\coH_{\rig,\rmid}^1(\Gm/K, \Sym^k\Kl_{n+1})&\cong\gr_{kn+1}^W[\coH^{kn-1}_{\rig, \cp}(\KM/K)(-1)]^{\symgp_k \times \mu_{n+1},\,\chi_n}\\
&\cong \gr_{kn+1}^W\coH^{kn+1}_{\rig, \KM/K}(\Gm^{kn}/K)^{\symgp_k \times \mu_{n+1},\,\chi_n},
\end{align*}
as well as an isomorphism with $\gr_{kn+1}^W[\coH^{kn-1}_{\rig}(\KM/K)(-1)]^{\symgp_k \times \mu_{n+1},\,\chi_n}$ if $\KM$ is smooth.

\section{Computation of the Hodge filtration}
\label{sec:comp-Hodge}

From now on, we focus on the case $n=1$. In this section, we construct a basis of the de Rham cohomology of the symmetric power~$\Sym^k\Kl_2$ that is well enough behaved with respect to the Hodge filtration for us to be able to prove Theorem~\ref{th:main} from the introduction.

\subsection{Structure of \texorpdfstring{$\Kl_2$}{Kl2}, \texorpdfstring{$\wtKl_2$}{wtKl2}, and their symmetric powers}

In preparation for the proof, we first make some of the statements of Section \ref{sec:expmot} more precise for $n=1$. Among other things, this will allow us to compute the irregularity number (as defined \eg in \cite[Chap.\,III]{Malgrange91}) of $\Sym^k \Kl_2$ and $\Sym^k \wt\Kl_2$ at infinity.

\subsubsection*{Explicit bases of $\Kl_2$ and $\wtKl_2$}

Recall that $\Kl_2$ is a rank-two $\CC[\tz, \tz^{-1}]$-module with connection that has a regular singularity at $\tz=0$ and an irregular singularity with slope $1/2$ at $\tz=\infty$. It is the localized Fourier transform $\Kl_2=\cF(E^{1/y})$ of the rank-one $\CC[y,y^{-1}]$\nobreakdash-module with connection~$E^{1/y}$ generated by $\rme^{1/y}$. Arguing as in the proof of Proposition \ref{prop:Kln}, this interpretation shows that the generator $\basis_0=\rme^{1/y}/y$ of $E^{1/y}$, viewed as an element of $j_0^\bast\FT E^{1/y}$, satisfies
\[
(\tz\partial_\tz)^2\basis_0=\tz\basis_0.
\]
We set $\basis_1=\tz\partial_\tz \basis_0$, so that $\basis_1=\rme^{1/y}/y^2$ holds.
Then $\{\basis_0,\basis_1\}$ is a $\CC[\tz,\tz^{-1}]$-basis of $\Kl_2$
and the matrix of the connection in this basis is given by
\begin{equation}\label{eqn:connection1}
\arraycolsep3.5pt
\tz\partial_\tz(\basis_0,\basis_1)=(\basis_0,\basis_1)\begin{pmatrix}0&\tz\\1&0\end{pmatrix}.
\end{equation}

On the other hand, by its definition as $\SH^0\pi_\bast \fqq{f}$ with $f(\tz,x)=x+\tz/x$ and $\pi(z, x)=z$, the $\CC[z,z^{-1}]$-module $\Kl_2$ is equal to the cokernel
of the relative de~Rham complex
\[
\CC[x, x^{-1}, z, z^{-1}] \To{\de + \pt_x f \de x}
	\CC[x, x^{-1}, z, z^{-1}]\de x.
\]
The elements $\basis_0$ and $\basis_1$ are respectively
given by the classes of $\rd x/x$ and $x\rd x/x$
in this cokernel.

The basis $\{\basis_0,\basis_1\}$ lifts to a $\CC[\tt,\tt^{-1}]$-basis
$\{\wtbasis_0,\wtbasis_1\}$ of $\wtKl_2$ satisfying
\begin{equation}\label{eqn:connection2}
\tfrac12\tt\partial_{\tt}\wtbasis_0=\wtbasis_1,\quad \tfrac12\tt\partial_{\tt}\wtbasis_1=\tt^2\wtbasis_0.
\end{equation}

\subsubsection*{Duality}
Let $\{\basis_0^\vee,\basis_1^\vee\}$ denote the basis of $\Kl_2^\vee$ dual to $\{\basis_0,\basis_1\}$. The matrix of the dual connection in this basis is equal to $-\begin{smallpmatrix}0&1\\z&0\end{smallpmatrix}$, and hence there is an isomorphism $\Kl_2\isom\Kl_2^\vee$ given by~\hbox{$(\basis_0,\basis_1)\mto(\basis_1^\vee,-\basis_0^\vee)$}. It induces a non-degenerate skew-symmetric pairing
\[
\langle\,,\,\rangle\colon \Kl_2\otimes\Kl_2\to\CC[\tz,\tz^{-1}], \qquad \langle\basis_0,\basis_1\rangle=-\langle\basis_1,\basis_0\rangle=1,
\]
which is compatible with the connection. There is a similar formula for $\wtKl_2$.

\subsubsection*{Structure at infinity}

Recall the cover $[2]\colon {\Gm}_{, t} \to {\Gm}_{,\tz}$ induced by $t \mto z=t^2$. Since the irregular singularity of $\Kl_2$ at $\tz=\infty$ has slope $1/2$, the formal stationary phase formula (\cf\cite[Th.\,5.1]{Bibi07a}) gives the formal structure at infinity
\[
\wh\Kl_2=\Cptzm\otimes_{\CC[\tz,\tz^{-1}]}\Kl_2\simeq\ramiftwo_\bast(E^{2\tt}\otimes L_{-1}),
\]
where $L_{-1}$ is the rank-one $\Cpttm$-module with connection $(\Cpttm,\rd+\frac12\rd(\tt^{-1})/\tt^{-1})$, with monodromy $-1$. Let $\sfi$ be a square root of~$-1$ and let $L_{\sfi}=(\Cptzm,\rd+\frac14\rd(\tz^{-1})/\tz^{-1})$ denote the rank-one meromorphic connection with monodromy $\sfi$ around $\infty$, so that $L_{-1}=\ramiftwo^\bast L_{\sfi}$ holds. Then there is also an isomorphism
\begin{equation}\label{eq:formalKl2}
\wh\Kl_2\simeq(\ramiftwo_\bast E^{2\tt})\otimes L_{\sfi}.
\end{equation}
The pullback $\wtKl_2=\ramiftwo^\bast\Kl_2$ has slope one at $\tt=\infty$ and unipotent monodromy with one Jordan block at the origin.
Its formal structure is given by
\[
\wh{\wtKl_2}=\Cpttm\otimes_{\CC[\tt,\tt^{-1}]}\wtKl_2=\ramiftwo^\bast(\ramiftwo_\bast E^{2\tt})\otimes L_{-1}\simeq (E^{2\tt} \oplus E^{-2\tt})\otimes L_{-1}.
\]
The module $\wh{\wtKl_2}$ is endowed with the $\mu_2$-action induced by $\tt\mto-\tt$
that exchanges both summands,
whose invariant submodule is identified with $\ramiftwo_\bast E^{2\tt}$ as a $\Cptzm$-module.

\subsubsection*{$\wtKl_2$ as a localized Fourier transform}

Recall the map $g\colon \Gm\to\Afu_\taut$ defined by
$y \mto y+1/y$. Its critical points are $y=\pm1$ and its critical values are $\taut=\pm2$. There is a decomposition $g_\bast\cO_{\Gm}=\cO_{\Afu}\oplus M_2$, where $M_2$ is an irreducible $\Cltaut$-module with regular singularities at $\taut=\pm2$ and $\taut=\infty$. In fact, $M_2$ is a rank-one free $\CC[\taut,(\taut\pm2)^{-1}]$-module with connection, with monodromy $-\id$ around each point $\taut=\pm2$ and monodromy $\id$ around $\infty$. Its Fourier transform $\FT(M_2)$ is then an irreducible $\Cltt$-module with a regular singularity at $\tt=0$ and an irregular singularity of slope one at $\tt=\infty$. Equation~\eqref{eq:KltildeisFourier} gives in particular:
\begin{equation*}\label{eq:Kltilde2}
\wtKl_2=j_0^\bast\FT(g_\bast\cO_{\Gm})=j_0^\bast\FT(M_2).
\end{equation*}
Therefore, $\wtKl_2$ is an irreducible $\CC[\tt,\tt^{-1}]$-module with connection and $\FT(M_2)$ is equal to the intermediate extension $j_{0\bea}\wtKl_2$. It follows from the formulas \eqref{eqn:connection1} and \eqref{eqn:connection2} for the connection on $\Kl_2$ and $\wtKl_2$ that the monodromy of the nearby cycles $\psi_z\Kl_2$ and $\psi_{\tt}\wtKl_2$ (two-dimensional vector spaces) is the unipotent automorphism with one Jordan block of size two.

\subsubsection*{Bases of $\Sym^k\Kl_2$ and $\Sym^k\wtKl_2$}

Out of the bases $\{\basis_0, \basis_1\}$ of $\Kl_2$ and $\{\wtbasis_0,\wtbasis_1 \}$ of $\wt\Kl_2$, we obtain a $\CC[\tz,\tz^{-1}]$-basis $\{\symbasis_a\}_{0 \leq a \leq k}$ of $\Sym^k\Kl_2$ and a $\CC[\tt,\tt^{-1}]$-basis $\{\wtsymbasis_a\}_{0 \leq a \leq k}$ of $\Sym^k\wtKl_2$ by considering the monomials
\begin{equation}\label{eq:basisSymKl2}
\symbasis_a=\basis_0^{k-a}\basis_1^a \quad\text{and}\quad \wtsymbasis_a=\wtbasis_0{}^{k-a}\wtbasis_1{}^a \quad (0\leq a\leq k).
\end{equation}
In the bases $\{\symbasis_a\}$ and $\{ \wtsymbasis_a\}$,
the connections of $\Sym^k\Kl_2$ and $\Sym^k\wtKl_2$ are given by
\begin{equation}\label{eq:tdteta}
\begin{split}
\tz\partial_\tz\symbasis_a&=(k-a)\symbasis_{a+1}+a\tz\symbasis_{a-1},\\
\tfrac12 \tt\partial_{\tt}\wtsymbasis_a&=(k-a)\wtsymbasis_{a+1}+a\tt^2\wtsymbasis_{a-1},
\end{split}
\end{equation} with the convention $\symbasis_b=\wtsymbasis_b=0$ for $b<0$ or $b>k$.

We gather in the following proposition the properties of $\Sym^k\Kl_2$ and $\Sym^k\wtKl_2$.

\begin{prop}\label{prop:SymkKl}
The free $\CC[\tz,\tz^{-1}]$-module with connection~$\Sym^k\Kl_2$ and the free $\CC[\tt,\tt^{-1}]$-module with connection $\Sym^k\wtKl_2$ satisfy the following properties:
\begin{enumerate}
\item\label{prop:SymkKl1}
$\rk\Sym^k\Kl_2=\rk\Sym^k\wtKl_2=k+1$ and the monodromy of $\Sym^k\Kl_2$ around $\tz=0$ (\resp of $\Sym^k\wtKl_2$ around $\tt=0$) is unipotent with only one Jordan block of size $k+1$.

\item\label{prop:SymkKl2b}
$\Sym^k\Kl_2$ and $\Sym^k\wtKl_2$ are endowed with a $(-1)^k$-symmetric non-degenerate pairing.
\item\label{prop:SymkKl4}
The formal structure of $\Sym^k\Kl_2$ at infinity is given by
\[
\Sym^k\wh\Kl_2\simeq
\begin{cases}
\bigoplus_{j=0}^{(k-1)/2}(\ramiftwo_\bast E^{2(2j-k)\tt})\otimes L_{\sfi}^{\otimes k}&\text{if $k$ is odd},\\[5pt]
L_{\sfi}^{\otimes k}\oplus\bigoplus_{j=0}^{k/2-1}(\ramiftwo_\bast E^{2(2j-k)\tt})\otimes L_{\sfi}^{\otimes k}&\text{if $k$ is even}.
\end{cases}
\]
In particular, $\Sym^k \Kl_2$ has irregularity number $\irr_\infty(\Sym^k \Kl_2)=\flr{(k+1)/2}$ and, for odd~$k$, $\Sym^k\wh\Kl_2$ is purely irregular at infinity.
\end{enumerate}
\end{prop}

\begin{proof}
\eqref{prop:SymkKl1} follows from the property that the $k$-th symmetric power of the standard representation of~$\sld$ is an irreducible representation of dimension $k+1$.

\eqref{prop:SymkKl2b}
The $(-1)^k$-symmetric non-degenerate pairings of $\Sym^k\Kl_2$ and $\Sym^k\wtKl_2$ are induced by the skew-symmetric self-duality of $\Kl_2$ and $\wtKl_2$ respectively.

\eqref{prop:SymkKl4}
The formal structure of $\wh\Kl_2$ at infinity given by \eqref{eq:formalKl2} implies the following:

\begin{lemma}\label{lem:SymwhKl2}
There is an isomorphism $\Sym^k\wh\Kl_2\simeq\Sym^k(\ramiftwo_\bast E^{2\tt})\otimes L_{\sfi}^{\otimes k}$. Moreover, $\Sym^k(\ramiftwo_\bast E^{2\tt})$ is the invariant submodule of $\ramiftwo^\bast\Sym^k(\ramiftwo_\bast E^{2\tt})\simeq\Sym^k(\ramiftwo^\bast\ramiftwo_\bast E^{2\tt})$ under $\tt \mapsto -\tt$. The latter $\Cpttm$-module with connection decomposes as $\bigoplus_{j=0}^k E^{2(2j-k)\tt}$.
\end{lemma}

We know that $\Sym^k(\ramiftwo_\bast E^{2\tt})$ is the invariant part of $\bigoplus_{j=0}^k E^{2(2j-k)\tt}$ under $\tt\mto-\tt$ . On the one hand, the invariant part of $E^{a\tt}\oplus E^{-a\tt}$ is $\ramiftwo_\bast E^{a\tt}=\ramiftwo_\bast E^{-a\tt}$ for $a\neq0$. On the other hand, the invariant part of $(\Cpttm,\rd)$ is $(\Cptzm,\rd)$. Therefore, there is an isomorphism
\begin{equation*}\label{eq:formsymk}
\Sym^k(\ramiftwo_\bast E^{2\tt})\simeq
\begin{cases}
\bigoplus_{j=0}^{(k-1)/2}\ramiftwo_\bast E^{2(2j-k)\tt}&\text{if $k$ is odd},\\[5pt]
(\Cptzm,\rd)\oplus\bigoplus_{j=0}^{k/2-1}\ramiftwo_\bast E^{2(2j-k)\tt}&\text{if $k$ is even},
\end{cases}
\end{equation*}
from which \ref{prop:SymkKl}\eqref{prop:SymkKl4} follows.
\end{proof}

It follows from the proposition that the formal regular component $(\Sym^k\wh\Kl_2)_\reg$ has rank zero for odd $k$ and rank one for even $k$, and in the latter case the formal monodromy has eigenvalue one if and only if $k\equiv0\bmod4$. From the proof, we also get the irregularity number
\begin{equation}\label{eq:irregSymk}
\irr_\infty(\Sym^k\wtKl_2)=\begin{cases}
k+1&\text{if $k$ is odd},\\
k&\text{if $k$ is even}.
\end{cases}
\end{equation}
Similarly, the formalization of $\Sym^k\wtKl_2$ at infinity is given by
\begin{equation}\label{eq:formsymkwt}
\Sym^k\wh{\wtKl_2}\simeq\bigoplus_{j=0}^{k} E^{2(2j-k)\tt}\otimes L_{-1}^{\otimes k}
\end{equation} both for odd and even values of $k$.

\begin{cor}\label{cor:subext} Let $j_0\colon \Gm\hto\Afu_\tz$ and $j_\infty\colon \Gm\hto\Afu_{\sfrac{1}{\tz}}$ denote the inclusions.
\begin{enumerate}
\item\label{cor:subext1}
The natural morphism
\[
j_{\infty\bea}\Sym^k\Kl_2\to j_{\infty\bast}\Sym^k\Kl_2
\]
is an isomorphism if $k\not\equiv\nobreak0\bmod4$. The same holds for $\Sym^k\wtKl_2$ if $k\not\equiv\nobreak0\bmod2$.
\item\label{cor:subext2}
Let $N$ be a proper $\Cltz$-submodule of $j_{0\bast}\Sym^k\Kl_2$ satisfying $j_0^\bast N=\Sym^k\Kl_2$. Then the equality $N=j_{0\bea}\Sym^k\Kl_2$ holds. The same is true for $\Sym^k \wtKl_2$.
\end{enumerate}
\end{cor}

\begin{proof}
Statement \eqref{cor:subext1} follows from \ref{prop:SymkKl}\eqref{prop:SymkKl4}. For \eqref{cor:subext2}, the question is local analytic around the point $\tz=0$. Set $E=\psi_{\tz,1}\Sym^k\Kl_2$ and let $\rN$ denote the nilpotent endomorphism on nearby cycles (\cf Section \ref{subsec:notationD}). Giving $N$ is equivalent to giving a subspace~$F$ of $E$ stable by $\rN$, together with two morphisms $E\To{\can}F\To{\var}E$ commuting with~$\rN$ such that $\var\circ\can=\rN$ and that $\var$ is the natural inclusion. Hence there is an inclusion $F\supset\image\can=\image\rN$. Since $\rN$ has only one Jordan block, $\image\rN$ has codimension one in $E$ and since $F\neq E$ by the properness assumption, this implies $F=\image\rN$, as wanted.
\end{proof}

\subsubsection*{The inverse Fourier transform of \texorpdfstring{$\Sym^k\wtKl_2$}{wtSymKl2}}

Some results of Section \ref{subsec:FSymkwtKln} can be made more explicit for $n=1$. Namely, applying the vanishing cycle functor at $\tt=0$ to the exact sequence~\eqref{eq:intermextwt} with $n=1$ we find the exact sequence
\[
0\to \phi_{\tt,1}j_{0\bea}(\Sym^k\wtKl_2)\to\phi_{\tt,1}j_{0\bast}(\Sym^k\wtKl_2)\to \phi_{\tt,1}\wt C_{k,1}\to0.
\]
By definition of the functors $j_{0\bast}$ and $j_{0\bea}$, the middle term $\phi_{\tt,1}j_{0\bast}(\Sym^k\wtKl_2)$ is canonically identified with the nearby cycle module $\psi_{\tt,1}j_{0\bast}(\Sym^k\wtKl_2)$, and $\phi_{\tt,1}j_{0\bea}(\Sym^k\wtKl_2)$ is then identified with the subspace $\image\wtrN$, where $\wtrN$ is the nilpotent endomorphism acting on $\psi_{\tt,1}j_{0\bast}(\Sym^k\wtKl_2)$. Since $\wtrN$ has only one Jordan block of size $k+1$, it follows that $\phi_{\tt,1}\wt C_{k,1}$ is one-dimensional, and that $\wtrN$ acting on $\phi_{\tt,1}j_{0\bea}(\Sym^k\wtKl_2)$ has only one Jordan block of size~$k$.

Let us consider the exact sequence \eqref{eq:wtM} in the present setting. The origin $\taut=0$ is a singular point for $\wtM$ and $\Pi(\wtM)$ if and only if the formal regular component of $\Sym^k\wtKl_2$ at infinity is non-zero, and then $\dim\phi_{\taut}\wtM=\dim\phi_{\taut}\Pi(\wtM)$ is equal to the rank of this formal regular component. Arguing as for $\Kl_2$, this rank is equal to zero if $k$ is odd and to one if $k$ is even, and in the latter case the eigenvalue of the corresponding formal monodromy is $(-1)^k$.

Let us summarize the properties of $\wtM$ and $\Pi(\wtM)$.

\begin{cor}\label{cor:wtM}
Let $\wtM$ be the inverse Fourier transform of $j_{0\bea}\Sym^k\wtKl_{n+1}$.
\begin{enumerate}
\item\label{cor:wtM1}
$\Pi(\wtM)$ is a regular holonomic $\Cltaut$-module, generically of rank $k+1$ with singularities at the points $\taut=2(2j-k)$ for $j=0,\dots,k$. The vanishing cycle space at each of these singularities has rank one with local monodromy equal to $(-1)^k\id$. At $\taut=\infty$, the monodromy is unipotent, with only one Jordan block of size $k+1$.
\item\label{cor:wtM2}
$\wtM$ is a regular holonomic $\Cltaut$-module, generically of rank $k$, with singularities at the points $\taut=2(2j-k)$ for $j=0,\dots,k$. The vanishing cycle space at each of these singularities has rank one with local monodromy equal to $(-1)^k\id$. At~$\taut=\infty$, the monodromy is unipotent, with only one Jordan block of size $k$.
\end{enumerate}
\end{cor}

\subsection{De Rham cohomology on \texorpdfstring{$\Gm$}{Gm}}
As we saw in Proposition~\ref{prop:HdRSymKln}, the de Rham cohomology of $\Sym^k \Kl_2$ is concentrated in degree one. Thanks to the analogue of the Grothen\-dieck\nobreakdash-Ogg\nobreakdash-Shafarevich formula for vector bundles with connection (see \eg \cite[Th.\,2.9.9]{Katz90}), the dimension of $\coH_\dR^1(\Gm,\Sym^k\Kl_2)$ is equal to the irregularity number of $\Sym^k\Kl_2$ at infinity. From Proposition~\ref{prop:SymkKl}\eqref{prop:SymkKl4}, we thus obtain
\begin{equation}\label{eq:dimSymk}
\dim \coH_\dR^1(\Gm,\Sym^k\Kl_2)=\irr_\infty(\Sym^k\Kl_2)=\flr{\frac{k+1}{2}}.
\end{equation}
Similarly, using \eqref{eq:irregSymk} we get
\begin{equation*}\label{eq:dimtildeSymk}
\dim \coH_\dR^1(\Gm,\Sym^k\wtKl_2)=\irr_\infty(\Sym^k\wtKl_2)=\begin{cases}
k+1&\text{if $k$ is odd},\\
k&\text{if $k$ is even}.
\end{cases}
\end{equation*}
By self-duality (Proposition~\ref{prop:SymkKl}\eqref{prop:SymkKl2b}) and Poincaré duality, there are isomorphisms
\begin{equation}
\begin{split}
\coH_{\dR,\rc}^1(\Gm,\Sym^k\Kl_2)&\simeq \coH_\dR^1(\Gm,\Sym^k\Kl_2)^\vee,\\
\coH_{\dR,\rc}^1(\Gm,\Sym^k\wtKl_2)&\simeq \coH_\dR^1(\Gm,\Sym^k\wtKl_2)^\vee.
\end{split}
\end{equation}

We consider the intermediate extension $\cD_{\PP^1}$-modules
$j_\bea\Sym^k\Kl_2$ and $j_\bea\Sym^k\wt\Kl_2$
with respect to the inclusion $j\colon \Gm\hto\PP^1$, which according to Corollary~\ref{cor:vanishingSymKln} compute the middle de Rham cohomology. Recall from \loccit that it is also concentrated in degree one.

\begin{prop}\label{prop:dimH1mid} Let $\delta_{4\ZZ}$ denote the characteristic function of multiples of $4$. We have:
\begin{equation}\label{eqn:dimensionHmid}
\dim \coH_{\dR,\rmid}^1(\Gm,\Sym^k\Kl_2)
= \flr{\dfrac{k-1}{2}} -\delta_{4\ZZ}(k)
=
\begin{cases}
\dfrac{k-1}2&\text{if $k$ is odd},\\[7pt]
2\flr{\dfrac{k-1}{4}} & \text{if $k$ is even},
\end{cases}
\end{equation}
\[
\dim \coH_{\dR,\rmid}^1(\Gm,\Sym^k\wtKl_2)=
\begin{cases}
k&\text{if $k$ is odd},\\
k-2&\text{if $k$ is even}.
\end{cases}
\]
\end{prop}

\begin{proof}
We first consider the intermediate extension by $j_0\colon \Gm\hto\Afu_\tz$. Corollary~\ref{cor:subext}\eqref{cor:subext2} and its proof imply that the cokernel of the injective morphism of $\Cltz$-modules
\[
j_{0\bea}\Sym^k\Kl_2\to j_{0\bast}\Sym^k\Kl_2
\]
is equal to $i_{0\bast}\CC$, where $i_0\colon \{0\}\hto\Afu_\tz$ is the inclusion. Besides, for the intermediate extension by $j_\infty\colon \Gm\hto\Afu_{1/\tz}$, we note that the natural morphism
$j_{\infty\bea}\Sym^k\Kl_2\to j_{\infty\bast}\Sym^k\Kl_2$
is an isomorphism if the formal completion of $\Sym^k\Kl_2$ at~$\infty$
is purely irregular or has no monodromy invariants,
that~is, if $k\not\equiv0\bmod4$ according to Proposition~\ref{prop:SymkKl}\eqref{prop:SymkKl4}.
Otherwise, since the formal regular component has rank one and monodromy equal to the identity, this morphism is injective with cokernel isomorphic to $i_{\infty\bast}\CC$, where $i_\infty \colon \{\infty\}\hto\Afu_{1/\tz}$ is the inclusion. Therefore, the equality
\begin{equation}\label{eq:codimmid}
\dim \coH_{\dR,\rmid}^1(\Gm,\Sym^k\Kl_2)=
\begin{cases}
\dim \coH_\dR^1(\Gm,\Sym^k\Kl_2)-1&\text{if }4\nmid k,\\
\dim \coH_\dR^1(\Gm,\Sym^k\Kl_2)-2&\text{if }4 \mid k,
\end{cases}
\end{equation}
holds, and we conclude from \eqref{eq:dimSymk}.
The proof for $\Sym^k\wtKl_2$ is similar.
\end{proof}

We can now give explicit bases of the de Rham cohomology of $\Sym^k\Kl_2$ and $\Sym^k\wtKl_2$.

\begin{prop}\label{prop:Coh_grM}
The space $\coH_\dR^1(\Gm,\Sym^k\Kl_2)$ has a basis consisting of the classes
\[
\tz^j\basis_0^k\,\frac{\rd\tz}{\tz},\quad 0\leq j<\flr{\frac{k+1}{2}},
\]
and the space $\coH_\dR^1(\Gm,\Sym^k\wtKl_2)$ has a basis consisting of the classes
\[
\tt^j\wtbasis_0{}^k\,\frac{\rd\tt}{\tt},\quad 0\leq j<2\flr{\frac{k+1}{2}}.
\]
\end{prop}

\begin{proof}
We will only consider the case of $\Sym^k\Kl_2$, that of $\Sym^k\wtKl_2$ being similar by replacing the grading below with the one for which $\deg\tt=1$. The space $\coH_\dR^r(\Gm, \Sym^k\Kl_2)$ is identified with the cohomology of the two-term complex
\[
G \To{\tz\partial_\tz} G,\quad
G= \text{the $\Cltzm$-module $\Sym^k\Kl_2$}.
\]
Therefore, the map $\tz\partial_\tz$ is injective and the goal is to find a basis of its cokernel. Recall the $\CC[\tz, \tz^{-1}]$-basis $\{u_0, \dots, u_k\}$ of $G$ from \eqref{eq:basisSymKl2} and consider the $\CC[\tz]$-submodule
\[
G^\bast = \bigoplus_{i=0}^k \CC[\tz]\symbasis_i \subset G.
\]
Formula \eqref{eq:tdteta} shows that $G^\bast $ is stable under the action of~$\tz\partial_\tz$. In fact, the coherent sheaf on $\Afu$ associated with $G^+$ is Deligne's canonical extension of $\Sym^k\Kl_2$ to a logarithmic connection whose residue at $0$ has all eigenvalues equal to $0$.

\begin{lemma}\label{Lemma:G_to_Gplus}
The inclusion $(G^\bast , \tz\partial_\tz) \to (G, \tz\partial_\tz)$ is a quasi-isomorphism.
\end{lemma}

\begin{proof}
This follows from the equality $G = \bigcup_{r\geq 0} \tz^{-r}G^\bast $ and the fact that $\tz\partial_\tz$ acts invertibly
on $\tz^{-(r+1)}G^\bast /\tz^{-r}G^\bast $ (with eigenvalue $-(r+1)$ and one Jordan block) for all $r\geq 0$.
\end{proof}

Let $\deg\colon G^\bast \to (\ZZ_{\geq 0},+)$ be the multiplicative \emph{degree} map uniquely determined by
\begin{equation}\label{eq:deg_G}
\deg \tz = 2, \quad \deg \symbasis_i = i.
\end{equation}
(This degree is the one induced from the Newton degree associated with the Laurent polynomial $f_k$ that naturally appears in the computation of the tensor power $\Kl_2^{\otimes k}$, see Section~\ref{subsubsec:wj} \textit{infra}.)
Then $\tz\partial_\tz$ is (inhomogeneous) of degree one.
Let $\gr G^\bast $ be the associated graded module. The induced graded $\CC$-linear map $\ov{\tz\partial_\tz}\colon \gr G^\bast \to \gr G^\bast [1]$ is $\CC[z]$-linear and we shall regard it as a two-term complex $(\gr G^\bast , \ov{\tz\partial_\tz})$.

\begin{lemma}\label{Lemma:Coh_grG}
If $k$ is odd, then $\coH^0(\gr G^\bast , \ov{\tz\partial_\tz}) = 0$ and the vector space $\coH^1(\gr G^\bast , \ov{\tz\partial_\tz})$ is generated by the classes of $\tz^j\symbasis_0$, for $0\leq j \leq \sfrac{(k-1)}{2}$.

If $k$ is even, then $\coH^0(\gr G^\bast , \ov{\tz\partial_\tz})$ and $\coH^1(\gr G^\bast , \ov{\tz\partial_\tz})$ are the free rank-one modules over the graded ring $\CC[z]$ generated by~$\sum_{i=0}^{k/2} (-1)^i\binom{k/2}{i} \tz^i \symbasis_{k-2i}$ and the class of $\symbasis_0$ respectively.
\end{lemma}

\begin{proof}
We shall determine the structure of the endomorphism $\ov{\tz\partial_\tz}$
on the finitely generated module $\gr G^\bast$
over the principal ideal domain $\CC[\tz]$. From formula \eqref{eq:tdteta} we see that $\tz\partial_\tz$ induces an isomorphism of $\CC[\tz]$-modules
\begin{equation}\label{eq:zpt_without_k}
\ov{\tz\partial_\tz}\colon \bigoplus_{i=0}^{k-1} \CC[\tz]\symbasis_i \to \gr G^\bast /\CC[\tz]\symbasis_0
\end{equation} and that, with respect to the basis $\{\symbasis_i\}$, the action $\ov{\tz\partial_\tz}$ has determinant $(k!!)^2(-\tz)^{(k+1)/2}$ if~$k$ is odd and zero otherwise.
If $k$ is odd, the space $\coH^1(\gr G^\bast , \ov{\tz\partial_\tz})$
has dimension $\psfrac{k+1}{2}$ and coincides with the image of $\CC[\tz]\symbasis_0$ through the isomorphism \eqref{eq:zpt_without_k}. Therefore, the elements~$\tz^j\symbasis_0$, for $0\leq j \leq \psfrac{k-1}{2}$, form a basis of $\coH^1(\gr G^\bast, \ol{\tz\partial_\tz})$.

If $k$ is even, then \eqref{eq:zpt_without_k} gives an isomorphism $\CC[\tz]\symbasis_0 \isom \coH^1(\gr G^\bast, \ol{\tz\partial_\tz})$.
On the other hand,
notice that the map $\ol{\tz\partial_\tz}$ splits as a direct sum
$\ol{\tz\partial_\tz}{}' \oplus \ol{\tz\partial_\tz}{}''$,
where
\begin{equation*}\label{eq:splitting_olzpt}
\ol{\tz\partial_\tz}{}'\colon \bigoplus_{j=0}^{k/2} \CC[\tz]\symbasis_{2j} \to \bigoplus_{j=1}^{k/2} \CC[\tz]\symbasis_{2j-1}, \quad \ol{\tz\partial_\tz}{}''\colon \bigoplus_{j=1}^{k/2} \CC[\tz]\symbasis_{2j-1} \to \bigoplus_{j=0}^{k/2} \CC[\tz]\symbasis_{2j}.
\end{equation*}
Moreover, $\ol{\tz\partial_\tz}{}'$ is surjective and $\ol{\tz\partial_\tz}{}''$ injective, and hence $\coH^0(\gr G^\bast, \ol{\tz\partial_\tz})$ is contained in the submodule $\bigoplus_{j=0}^{k/2} \CC[\tz]\symbasis_{2j}$. The statement then follows from an inspection of the coefficients~$a_i$ in the equation
$\ol{\tz\partial_\tz}'(\sum_{i=0}^{k/2}a_iu_{k-2i})=0$.
\end{proof}

To finish the proof of Proposition~\ref{prop:Coh_grM},
we use the spectral sequence
\[
E_1^{p,q} = \coH^p\Bigl(\gr_{q-p} G^\bast \To{\ov{\tz\partial_\tz}} \gr_{q-p+1} G^\bast \Bigr) \implique \coH^p(G^\bast , \tz\partial_\tz) \quad (p \in \{0,1\},\; q\geq 0)
\]
associated with the grading \eqref{eq:deg_G}, which degenerates at the $E_2$\nobreakdash-page. If $k$ is odd, all terms~$E_1^{0, q}$ vanish by the first part of Lemma \ref{Lemma:Coh_grG}, and the spectral sequence yields an isomorphism of vector spaces $\coH^1_\dR(\Gm, \Sym^k\Kl_2) \simeq \coH^1(\gr G^\bast, \ov{\tz\partial_\tz})$. The statement follows using Lemma \ref{Lemma:Coh_grG} again. If $k$ is even, then~$\coH^1_\dR(\Gm, \Sym^k\Kl_2)$ is isomorphic to the cokernel of the induced map
\begin{equation}\label{eqn:cokerH1evenk}
\tz\partial_\tz\colon \coH^0(\gr G^\bast, \ov{\tz\partial_\tz}) \to \coH^1(\gr G^\bast, \ov{\tz\partial_\tz}).
\end{equation}
For each $r\geq 0$, the equality
\[
\tz\partial_\tz\biggl(\tz^r \sum_{i=0}^{k/2} a_i \tz^i \symbasis_{k-2i} \biggr)= \sum_{i=0}^{k/2} (r+i)a_iz^{r+i}\symbasis_{k-2i} \equiv c_rz^{r+\sfrac{k}{2}}\symbasis_0
\]
holds in $\coH^1(\gr G^\bast , \ov{\tz\partial_\tz})$ for some $c_r \in \CC$. Therefore, the classes of $z^j u_0$, for $0 \leq j \leq \sfrac{k}{2}-1$, are linearly independent in the cokernel of \eqref{eqn:cokerH1evenk}. Since there are as many as the dimension of $\coH^1_\dR(\Gm, \Sym^k\Kl_2)$ by \eqref{eq:dimSymk}, they form a basis.
\end{proof}

\subsection{The Hodge filtration}
In this section, we prove Theorem~\ref{th:main}. In order to do so, we first establish the analogue of this result for $\Sym^k\wtKl_2$, which is stated as follows:

\begin{prop}\label{prop:main}
The mixed Hodge structure $\coH^1(\Gm, \Sym^k\wtKl_2)$ has weights $\geq k+1$ and the following numerical data:
\begin{enumerate}
\item\label{prop:mainodd}
For odd $k$, it is mixed of weights $k+1$ and $2k+2$, with
\[
\dim \coH^1(\Gm, \Sym^k\wtKl_2)^{p,q}=
\begin{cases}
1, & p+q = k+1,\ p \in \{1,\dots, k\}, \\
1, & p = q = k+1, \\
0, & \text{otherwise}.
\end{cases}
\]
\item\label{prop:maineven}
For even $k$, it is mixed of weights $k+1$, $k+2$ and $2k+\nobreak2$, with
\[
\dim \coH^1(\Gm, \Sym^k\wtKl_2)^{p,q}=
\begin{cases}
1, & p+q = k+1,\; p\in\{1,\dots, k\}\text{ and }p\neq k/2,\,k/2+1, \\
1, & p = q =k/2+1, \\
1, & p = q =k+1, \\
0, & \text{otherwise}.
\end{cases}
\]
\end{enumerate}
Furthermore, the mixed Hodge structure $\coH^1_{\rmid}(\Gm,\Sym^k\wtKl_2)$ is pure of weight $k+1$ and is equal to $W_{k+1}\coH^1_\dR(\Gm, \Sym^k\wtKl_2)$.
\end{prop}

\begin{proof}
The weight properties of $\coH^1(\Gm, \Sym^k\wtKl_2)$ and the purity of $\coH^1_{\rmid}(\Gm,\Sym^k\wtKl_2)$ were already obtained in the more general setting of Theorem \ref{th:weightsMHM}. To compute the Hodge numbers, we take up the argument in its proof for~\hbox{$n=1$}. Recall that the $\Cltaut$-module $\wtM$ is irreducible, generically of rank $k$, and underlies a pure Hodge module of weight~$k$
(Proposition~\ref{prop:HwtM} and Lemma \ref{lem:irreducibility}). Let us describe the Hodge filtration on $\wtM^\rH$. We start with the nearby cycles at infinity. Since the monodromy around infinity is maximally unipotent (Corollary~\ref{cor:wtM}), the non-zero graded pieces of the weight filtration on $\psi_{1/\taut}\wtM^\rH=\psi_{1/\taut, 1}\wtM^\rH$, which is the monodromy filtration associated with $\wtrN$ centered at $k-1$,
are the $\wtrN^\ell\rP_k$. They are hence of the form
$\gr_{2j}^W\psi_{1/\taut}\wtM^\rH$, for $0\leq j\leq k-1$, and one-dimensional. It follows that the mixed Hodge structure $\psi_{1/\taut}\wtM^\rH$ is of Hodge-Tate type and that
\begin{equation}\label{eq:psitauMtilde}
\gr^p_F\psi_{1/\taut}\wtM^\rH=\gr_{2p}^W\psi_{1/\taut}\wtM^\rH,\quad p=0,\dots,k-1,
\end{equation}
has dimension one. The compatibility property of \cite[3.2.1]{MSaito86} between the Hodge filtration and the Kashiwara-Malgrange filtration of the filtered $\cD$-modules underlying Hodge modules implies, in the case of smooth curves, the equality $\rk\gr^p_F\wtM^\rH=\dim\gr^p_F\psi_{1/\taut}\wtM^\rH$. Hence,~$\gr^p_F\wtM^\rH$ is generically a rank-one bundle for $p=0,\dots,k-1$.

Recall the equality $j_{0+}\Sym^k\wtKl_2=\FT\Pi(\wt M)$ from \eqref{eqn:PiofMtilde}. From Proposition \ref{prop:HwtM}, we derive an exact sequence of mixed Hodge structures
\[
0 \to \coH^1(\Afu_t,\FT \wtM^\rH) \to \coH^1(\Gm,\Sym^k\wtKl_2) \to \coH^1(\Afu_t,\FT \wtM'^\rH) \to 0
\]
by applying the functor of Notation~\ref{nota:FTMH}. Since $\wtM^{\prime\rH}$ is pure of weight~$2k+1$, Corollary~\ref{cor:weightssimple}\eqref{cor:weightssimple1} says that $\coH^1(\Afu_t,\FT \wtM'^\rH)$ is pure of weight $2k+2$. Besides, this space is one\nobreakdash-dimensional, since $\coH^1_\dR(\Afu_t,\FT \wtM)=\coH^1_\dR(\Afu_t,j_{0\dag+}\Sym^k\wtKl_2)$ has codimension one in~$\coH^1_\dR(\Gm,\Sym^k\wtKl_2)$ by the argument in the proof of Proposition~\ref{prop:dimH1mid}. This yields the lines $p=q=k+1$ in \eqref{prop:mainodd} and \eqref{prop:maineven}.

If $k$ is odd, then $0$ is not a singular point of $\wtM$, so Corollary \ref{cor:weightssimple}\eqref{cor:weightssimple2} applies. It follows that $\coH^1(\Afu_t,\FT \wtM^\rH)$ is pure of weight $k+1$ and its Hodge numbers are the ranks of $\gr^{p-1}_F\wtM^\rH$. Since $\gr^p_F\wtM^\rH$ has rank one for $p=0,\dots,k-1$ and is zero otherwise, this yields the rest of~\eqref{prop:mainodd}.

If $k$ is even, then $0$ is a singular point of $\wtM$ and, according to Corollary~\ref{cor:weightssimple}\eqref{cor:weightssimple3}, there is an isomorphism of mixed Hodge structures
\[
\coH^1(\Afu_t,\FT \wtM^\rH) \simeq \coker[\wt\rN\colon \psi_{\taut,1}\wtM^\rH\to\psi_{\taut,1}\wtM^\rH(-1)].
\]
Since~$\wtM$ is an intermediate extension at $\taut=0$ and $\dim\phi_{\tau,1}\wtM=1$ by Corollary \ref{cor:wtM}\eqref{cor:wtM2}, the vanishing $\wt\rN^2=0$ holds. Since $\wtM$ has generic rank $k$,
the primitive parts of the Lefschetz decomposition of $\gr^W\psi_{\taut,1}\wtM^\rH$ are thus
\begin{itemize}
\item
$\rP_1=\gr^W_k\psi_{\taut,1}\wtM^\rH$ of dimension one,
\item
$\rP_0=\gr^W_{k-1}\psi_{\taut,1}\wtM^\rH$ of dimension $k-2$,
\end{itemize}
and we get the equality
\[
\gr^W\coker\wt\rN=\rP_0(-1)\oplus\rP_1(-1).
\]
In particular,
$\gr_{k+2}^W\coH^1(\Afu_\tt,\FT\wtM^\rH)=\gr_{k+2}^W \coH^1(\Gm,\Sym^k\wtKl_2)$
corresponds to the summand $\rP_1(-1)$ and has dimension one, so is of Hodge type $(k/2+1,k/2+1)$, yielding the corresponding line in \eqref{prop:maineven}. We conclude the proof by using the equality
\[
\rk\gr^{p-1}_F\wtM^\rH=\dim\gr^{p-1}_F\psi_{\taut,1}\wtM^\rH=\dim\gr^p_F(\rP_1(-1))+\dim\gr^p_F(\rP_0(-1))+\dim\gr^{p+1}_F(\rP_1(-1)),
\]
which follows from the Hodge-Lefschetz decomposition on noting that $\gr^{p+1}_F\rP_1=\gr^p_F(\rN\rP_1)$. The leftmost term is one-dimensional for $p=1,\dots,k$ and zero otherwise. We already know that $\gr^p_F(\rP_1(-1))$ is one-dimensional for $p=k/2+1$ and zero otherwise. Hence $\gr^{p+1}_F(\rP_1(-1))$ is one-dimensional for $p=k/2$ and zero otherwise and we obtain $\dim\gr^p_F(\rP_0(-1))=1$ for the remaining values of $p$, yielding the first line in \eqref{prop:maineven}.
\end{proof}

We can now show that the bases of $\coH^1_\dR(\Gm, \Sym^k\wtKl_2)$ and $\coH^1_\dR(\Gm, \Sym^k \Kl_2)$ given in Proposition~\ref{prop:Coh_grM} are adapted to the Hodge filtration if $k$ is odd and that the first half of them are so if $k$ is even. This information will suffice to prove Theorem \ref{th:main} at the end of this section. (A full basis of $\coH^1_{\dR,\rmid}(\Gm, \Sym^k \Kl_2)$ adapted to the Hodge filtration is constructed in \hbox{\cite[Cor.\,3.28]{F-S-Y20b}} by exploiting the explicit calculation of the intersection pairing on this space.)
The proof will rely on the identification of the Hodge filtration on these spaces with their irregular Hodge filtration as de Rham fibers of exponential mixed Hodge structures (Theorem~\ref{th:EMHSMHStg}) and on toric techniques to compute the latter. In what follows, we still denote by $F^\bbullet$ the irregular Hodge filtration.

\begin{prop}\label{prop:FH1dR}
With respect to the bases from Proposition~\ref{prop:Coh_grM},
\begin{enumerate}
\item\label{prop:FH1dR1}
the Hodge filtration on $\coH^1_\dR(\Gm, \Sym^k\wtKl_2)$ is given by
\[
F^p\coH^1_\dR(\Gm, \Sym^k\wtKl_2) =
\Bigl\langle t^j\wtbasis_0^k\,\dfrac{\de t}{t}\Bigm| 0\leq j\leq k+1-p\Bigr\rangle
\]
if $k$ is odd, or if $k$ is even and $p>k/2$.
\item\label{prop:FH1dR2}
the Hodge filtration on $\coH^1_\dR(\Gm,\Sym^k\Kl_2)$ is given by
\[
F^p\coH^1_\dR(\Gm,\Sym^k\Kl_2) =
\Bigl\langle z^j\basis_0^k\,\frac{\de z}{z} \Bigm| 0\leq j \leq \Bigl\lfloor\frac{k+1-p}{2}\Bigr\rfloor\Bigr\rangle
\]
if $k$ is odd, or if $k$ is even and $p>k/2$.
\end{enumerate}
\end{prop}

\begin{proof}[Proof of the inclusion $\supset$]
The inclusions
\begin{align*}
\coH_\dR^1(\Gm, \Sym^k\Kl_2)&\hto\coH_\dR^1(\Gm, \Kl_2^{\otimes k})\simeq\coH_\dR^{k+1}(\Gm^{k+1}, \fqq{f_k}) \\
\coH_\dR^1(\Gm, \Sym^k\wtKl_2)&\hto\coH_\dR^1(\Gm, \wtKl_2^{\otimes k})\simeq\coH_\dR^{k+1}(\Gm^{k+1}, \fqq{\wtf_k})
\end{align*}
are strict with respect to the irregular Hodge filtration and map the basis elements $z^j\basis_0^k\,\sfrac{\de z}{z}$ of $\coH_\dR^1(\Gm, \Sym^k\Kl_2)$ and $t^j\wtbasis_0^k\sfrac{\de t}{t}$ of $\coH_\dR^1(\Gm, \Sym^k\wtKl_2)$ to
\[
w_j= z^j\,\frac{\de z}{z}\,\frac{\de x_1}{x_1} \cdots \frac{\de x_k}{x_k} \in \coH_\dR^{k+1}(\Gm^{k+1}, \fqq{f_k}) \text{ and }
\wt{w}_j=t^j\,\frac{\de t}{t}\frac{\de y_1}{y_1}\cdots \frac{\de y_k}{y_k}\in\coH_\dR^{k+1}(\Gm^{k+1}, \fqq{\wtf_k})
\]
respectively. It is therefore enough to prove that
\begin{enumeratei}
\item\label{enum:wj}
$w_j\in F^{k+1-2j}\coH_\dR^{k+1}(\Gm^{k+1}, \fqq{f_k})$ for $j\geq0$ if $k$ is odd and $0\leq 2j\leq k/2$ if $k$ is even,\item\label{enum:wtj}
$\wt w_j\in F^{k+1-j}\coH_\dR^{k+1}(\Gm^{k+1}, \fqq{\wtf_k})$ for $j\geq0$ if $k$ is odd and $0\leq j\leq k/2$ if $k$ is even.
\end{enumeratei}

\subsubsection{Proof of \eqref{enum:wj} and \eqref{enum:wtj} in the case when $k$ is odd}\label{subsubsec:wj}
We start with \eqref{enum:wj}. We identify the set of Laurent monomials in $z, x_i,\dots, x_k$
with the $\ZZ$-lattice $\ZZ^{k+1}$ in~$\RR^{k+1}$
by taking the exponents.
Let $\{\alpha_i\}_{i=0}^k$ be the dual basis
of the standard basis of~$\RR^{k+1}$.
Regardless of the parity of $k$,
the monomials appearing in $f_k=\sum_{j=1}^k x_j + \tz\sum_{j=1}^k\sfrac{1}{x_j}$ all lie in the affine hyperplane $h=1$
in~$\RR^{k+1}$ defined by the equation
$h=2\alpha_0+\sum_{i=1}^k \alpha_i$.
Thus the Newton polytope $\Delta\subset\RR^{k+1}$ of $f_k$ has only one facet that does not contain the origin; it lies on the hyperplane $h = 1$. The cone $\RR_{\geq 0}\Delta$ is given by the $2^k$ inequalities
\[
\alpha_0 + \sum_{i=1}^k \epsilon_i \alpha_i \geq 0,
\quad \epsilon_i \in \{0,1\}.
\]
It is straightforward to check that $f_k$ is non-degenerate
with respect to $\Delta$ if and only if $k$ is odd.
In this case, the irregular Hodge filtration on $\coH_\dR^{k+1}(\Gm^{k+1}, \fqq{f_k})$ arises from the Newton filtration on monomials $\RR_{\geq 0}\Delta$ by \cite[Th.\,1.4]{A-S97} and \cite[Th.\,4.6]{Yu12}. In particular, if $m \in \RR_{\geq 0}\Delta$ is a monomial with Newton degree $h(m)$ such that the top form
$\omega = m\frac{\de z}{z}\frac{\de x_1}{x_1}\cdots \frac{\de x_k}{x_k}$
represents a non-trivial class in $\coH_\dR^{k+1}(\Gm^{k+1}, \fqq{f_k})$,
then
\begin{equation}\label{eqn:representatives}
\omega \in F^p\coH_\dR^{k+1}(\Gm^{k+1}, \fqq{f_k}) \quad \text{if}\quad p\leq k+1-h(m).
\end{equation} In the case at hand, $z^j \in \RR_{\geq 0}\Delta$ has degree $h(z^j) = 2j$, hence the assertion.

For \eqref{enum:wtj},
we consider the function
$\wt h = \alpha_0$ on the cone
generated by the Newton polytope~$\wt{\Delta}$ of the Laurent polynomial $\wtf_k$.
If $k$ is odd, then $\wtf_k$ is non-degenerate. Moreover,
given $m \in \RR_{\geq 0}\wt{\Delta}$
such that the class $\wt\omega$ of $m\frac{\de t}{t}\frac{\de y_1}{y_1} \cdots \frac{\de y_k}{y_k}$ is non-trivial,
$\wt\omega$ belongs to $F^p\coH_\dR^{k+1}(\Gm^{k+1}, \fqq{\wtf_k})$ if $p\leq k+1-\wt h(m)$ holds. The result then follows from the equality $\wt h(t^j)=j$. \qed

\subsubsection{A toric compactification}\label{sec:toric-compactification}
Before proving \eqref{enum:wj} and \eqref{enum:wtj} for even $k$, we describe an explicit compactification of $(\Gm^{k+1}, \wtf_k)$ that will allow us to understand the Hodge filtration on the cohomology of $\fqq{\wtf_k}$.
Since the construction is also used in Section \ref{Sect:etale-real} to study the \'etale realizations of the motive $\Motive_k$, we take the base field to be $\QQ$ before dealing with Hodge filtrations in the second half of this subsection. Let $(U,f)$ be a pair consisting of a smooth quasi-projective variety $U$ and a regular function $f\colon U \to \Afu$.
After Mochizuki \cite[Def.\,2.6]{Mochizuki15b},
we call a smooth compactification $X$ of $U$ \textit{non-degenerate} along $D$ if the boundary $D= X \setminus U$ is a strict normal crossing divisor and $f$ extends to a rational morphism
\begin{displaymath}
f\colon X \dashrightarrow \PP^1
\end{displaymath}
such that, locally for the analytic topology around each point of $X$, there is a coordinate system
$\{ \xi_1,\dots, \xi_a, \eta_1, \dots, \eta_b, \zeta_1,\dots, \zeta_c \}$ and a multi-index $e \in \ZZ_{>0}^r$ satisfying
\begin{equation*}\label{eq:local_chart}
D = (\xi\eta), \quad \text{$f = \sfrac{1}{\xi^e}$ or $\sfrac{\zeta_1}{\xi^e}$}.
\end{equation*}

Recall the equality $\wtf_k=tg^{\boxplus k}=t\sum_{i=1}^k(y_i+1/y_i)$ from \eqref{eqn:relationfandg} and the isomorphism of Proposition~\ref{prop:HdRSymKln} with $n=1$. We first compactify $(\Gm^k,g^{\boxplus k})$.
For this, let $M=\bigoplus_{i=1}^k \ZZ y_i$ be the lattice of Laurent monomials on $\Gm^k$ and let~\hbox{$N=\bigoplus_{i=1}^k \ZZ e_i$} be the dual lattice with basis~$e_i$ dual to~$y_i$. We consider the toric compactification~$X$ of $\Gm^k$ attached to the simplicial fan $F$ in $N_\RR =\bigoplus_{i=1}^k \RR e_i$ generated by the $3^k-1$ rays
\begin{equation}\label{eq:fan}
\RR_{\geq 0}\cdot\sum_{i=1}^k \epsilon_ie_i \quad \text{with $\epsilon_i \in \{ 0,\pm 1\}$ and $(\epsilon_1,\dots, \epsilon_k) \neq 0$}.
\end{equation}
There are $2^kk!$ simplicial cones of maximal dimension $k$ in $F$, each of which provides an affine chart of $X$ isomorphic to $\IA^k$ on which the function $g^{\boxplus k}$ has the same structure. Explicitly as an example, consider the maximal cone of $F$ generated by the $k$ vectors
\[
\sum_{i=1}^r e_i, \quad 1\leq r\leq k.
\]
The affine ring associated with the dual cone in $M_\RR$ is the polynomial ring $\QQ[z_i]_{i=1}^k$ where
\[
z_r = \sfrac{y_r}{y_{r+1}}, \quad 1\leq r <k, \quad \text{and} \quad z_k = y_k.
\]
On this chart $X_1=\Spec(\QQ[z_i]_{i=1}^k) \cong \IA^k$, the equality $g^{\boxplus k}=\sfrac{g_1}{z_1\cdots z_k}$ holds with
\begin{equation}\label{eq:g1_on_Z1}
g_1=1+\sum_{r=2}^k z_1\cdots z_{r-1} + z_1\cdots z_k \sum_{r=1}^k z_r\cdots z_k \in \Gamma(X_1,\cO_{X_1}).
\end{equation}
The toric variety $X$ provides an example of a non-degenerate compactification of $(\Gm^k, g^{\boxplus k})$
as in a neighborhood of $X\setminus\Gm^k$,
the closure of the zero locus of $g^{\boxplus k}$ and $X\setminus\Gm^k$
form a strict normal crossing divisor (see also the paragraph before \S\ref{sect:l-adic_k_odd}).

We now construct a non-degenerate compactification of $(\Gm^{k+1}, tg^{\boxplus k}) \cong (\Gm^{k+1}, \wtf_k)$ starting from $\PP^1_\tt\times X$. For this, we order the $3^k-1$ irreducible components $(S_i)_{1\leq i\leq3^k-1}$ of $X\setminus\Gm^k$ corresponding to the rays \eqref{eq:fan} and consider the tower $\ol{X} \to \cdots \to \PP^1_\tt \times X$ of $3^k-1$ blow\nobreakdash-ups along the intersection of the proper transform of $0\times X$
(on which $tg^{\boxplus k}$ has a simple zero)
and the proper transform of $\PP^1_\tt\times S_i$
(on which $tg^{\boxplus k}$ has a simple pole). Together with the function induced from the blow-up maps, the resulting variety is a non\nobreakdash-degenerate compactification of~$(\Gm^{k+1}, \wtf_k)$ if $k$ is odd. If $k$ is even, then $\ol{X}$ is non\nobreakdash-degenerate away from the $\binom{k}{k/2}$ points
\[
(t, y_i)=(\infty, \epsilon_i), \quad \epsilon_i \in \{\pm 1\} \quad \sum_{i=1}^k \epsilon_i=0
\]
(note that they are defined over $\QQ$). Let $x$ be such a point. For a suitable choice of (analytic or étale) local coordinates $z_1, \dots, z_k$ of $X$ around $x$, the function $g^{\boxplus k}$ takes the form $z_1^2 +\cdots + z_k^2$, which means that $x$ is an ordinary quadratic point of $g^{\boxplus k}=0$. We perform two blow\nobreakdash-ups on~$\ol{X}$: first at each $x$ and then along the intersection of the exceptional divisor and the proper transform of $\infty\times X$. Let $\wt{X}$ be the resulting variety and $E_1$ and $E_2$ the exceptional divisors from the first and the second step respectively. A direct computation reveals that $\wt{X}$ is a non-degenerate compactification of the pair $(\Gm^{k+1}, \wtf_k)$
with $\ord_{E_1} \wtf_k=1$ and $\ord_{E_2} \wtf_k=0$.

\subsubsection{Proof of \eqref{enum:wj} and \eqref{enum:wtj} in the case when $k$ is even}

We now start with \eqref{enum:wtj}. Let $\wt{X}$ be the non-degenerate compactification of $(\Gm^{k+1}, \wtf_k)$ constructed above and $D=\wt{X} \setminus \Gm^{k+1}$. Since the indeterminacy locus of the rational map \hbox{$\wtf_k \colon \wt{X} \dasharrow \PP^1$} has codimension at least two in $\wt{X}$, one can define the pole divisor $P$ of $\wtf_k $ as the closure of the pole divisor of a representative of~$\wtf_k$, and similarly for the zero divisor. The exceptional divisors $E_1$ and $E_2$ are not contained in the support of $P$, and a direct computation shows that the form $\wt{w}_j$ lies in
\begin{equation*}\label{eq:tilde_w_odp}
\Gamma\Bigl(\wt{X}, \Omega^{k+1}_{\wt{X}}(\log D) (jP-(k-j)E_1-(k-2j)E_2) \Bigr).
\end{equation*}
Accordingly, if the inequalities $0\leq j\leq k/2$ hold (so that $\wt{w}_j$ is holomorphic generically on the divisors $E_1$ and~$E_2$), the form $\wt{w}_j$ lies in
\[
\Gamma \Bigl(\wt{X}, \Omega^{k+1}_{\wt{X}}(\log D)(jP) \Bigr).
\]
In this case, we claim that there is a natural map
\[
\Gamma \Bigl(\wt{X}, \Omega^{k+1}_{\wt{X}}(\log D)(jP) \Bigr) \to F^{k+1-j} \coH_\dR^{k+1}(\Gm^{k+1}, \fqq{\wtf_k}),
\]
from which the statement will follow. Indeed, as described in
\cite[\S 4($b$)]{Yu12} (especially in the paragraph containing diagram (26)),
one can resolve the indeterminacies of $\wt{f}_k$ by taking a tower of blow-ups $\pi\colon\wt{X}' \to \wt{X}$ of $\wt{X}$
along the intersections of the zero divisor
and the irreducible components of the pole divisor of the transforms of~$\wt{f}_k$ such that $D' = \wt{X}'\setminus \Gm^{k+1}$
remains a simple normal crossing divisor
and $\wt{f}_k$ extends to an everywhere defined morphism $\wt{f}_k'\colon \wt{X}'\to \PP^1$.
Let $P'$ be the pole divisor of $\wt{f}_k'$.
By \cite[Prop.\,4.4]{Yu12}, the following equality holds:
\[
\Rder\Gamma \Bigl(\wt{X}, \big(\Omega^\bbullet_{\wt{X}}(\log D)((\bbullet-p)P)_+,\de +\de\wt{f}_k\big) \Bigr)
= \Rder\Gamma \Bigl(\wt{X}', \big(\Omega^\bbullet_{\wt{X}'}(\log D')((\bbullet-p)P')_+,\de +\de\wt{f}_k'\big) \Bigr),
\]
where we use the notation
\[
\Omega^i_{\wt{X}}(\log D)((i-p)P)_+ = \begin{cases}
\Omega^i_{\wt{X}}(\log D)((i-p)P), & i\geq p, \\
0, & i<p,
\end{cases}
\]
and similarly for the other complex.
(In \cite[Prop.\,4.4]{Yu12}, the complex on the left-hand side is denoted by $F^p_{\mathrm{NP}}(\nabla)$ and that on the right-hand side by $F^p(\nabla)$, and Prop.\,4.4 of \loccit\ implies that $F^p_{\mathrm{NP}}(\nabla)$ and $\Rder\pi_\ast F^p(\nabla)$ are quasi-isomorphic; see also the proof of Th.\,4.6 in \loccit
Note that the assumption that the Laurent polynomial is non-degenerate is not needed for Prop.\,4.4 of \loccit)
On the other hand, by the $E_1$-degeneration of the irregular Hodge filtration proved in \cite[Th.\,1.2.2]{E-S-Y13}, where $F^p(\nabla)$ is denoted by $F^{\mathrm{Yu}, p}_0(\Omega_{\wt{X}'}(\ast D'), \nabla)$ instead,
the equality
\[
\mathbb{H}^{k+1}\Bigl(\wt{X}', \big(\Omega^\bbullet(\log D')((\cbbullet-p)P')_+,\de +\de\wt{f}_k'\big) \Bigr)
= F^p \coH_\dR^{k+1}(\Gm^{k+1}, \fqq{\wtf_k})
\]
holds. This completes the claim, and hence the proof of \eqref{enum:wtj}.

To prove \eqref{enum:wj}, we observe that the equality
\begin{equation}\label{eqn:notildeinvariantpart}
\coH^1_\dR(\Gm, \Sym^k\Kl_2)=\coH^1_\dR(\Gm, \Sym^k\wtKl_2)^{\mu_2}
\end{equation}
is compatible with the Hodge filtration, so that we can check whether a form belongs to some step of the Hodge filtration by pulling it back by the double cover $[2]$ given by \hbox{$t \mapsto z=t^2$}. Since the pullback $\ramiftwo^*z^j v_0^k\, \sfrac{\de z}{z} = 2t^{2j}\wtbasis_0^k\sfrac{\de t}{t}$ maps to $2\wt{w}_{2j}$, it lies in $F^{k+1-2j}\coH_\dR^1(\Gm, \Sym^k\wtKl_2)$ for all $0\leq 2j \leq k/2$ by part \eqref{enum:wtj}. We thus get
\[
z^j v_0^k\, \frac{\de z}{z} \in F^{k+1-2j}\coH_\dR^1(\Gm, \Sym^k\Kl_2) \quad \text{if $0\leq 2j \leq k/2$},
\]
which ends the proof of the inclusion $\supset$ in Proposition \ref{prop:FH1dR}.
\end{proof}

\begin{proof}[Proof of the equality in Proposition \ref{prop:FH1dR}\eqref{prop:FH1dR1}]
Since $t^j\wt v_0^k\sfrac{\rd t}t$ form a basis of $\coH^1_\dR(\Gm, \Sym^k\wtKl_2)$ and the graded pieces of the Hodge filtration on this space are one-dimensional by Proposition~\ref{prop:main}, the inclusion $\supset$ in Proposition \ref{prop:FH1dR}\eqref{prop:FH1dR1} is necessarily an equality.
\end{proof}

\begin{proof}[Proof of the equality in Proposition \ref{prop:FH1dR}\eqref{prop:FH1dR2} and of Theorem \ref{th:main}]
Thanks to Proposition~\ref{prop:main}, the mixed Hodge structure $\coH^1(\Gm, \Sym^k\wtKl_2)$ has weights $k+1$, $k+2$ (for even $k$) and $2k+2$, and the graded piece $\gr^p_F\coH_\dR^1(\Gm, \Sym^k\wtKl_2)$ is one-dimensional for $p=1,\dots, k+1$ (except for $p=k/2$ if $k$ is even) and zero otherwise. Since the identification \eqref{eqn:notildeinvariantpart} is compatible with the weight and the Hodge filtrations, the possible weights of $\coH^1(\Gm, \Sym^k\Kl_2)$ are $k+1$, $k+2$ (for even~$k$) and $2k+2$, with graded pieces of dimension at most one in the last two cases, and all Hodge numbers are zero or one, depending on whether $\mu_2$ acts as multiplication by~$-1$ or by~$+1$ on $\gr^p_F\coH_\dR^1(\Gm, \Sym^k\wtKl_2)$.

By Proposition \ref{prop:FH1dR}\eqref{prop:FH1dR1}, the class of $z^j v_0^k\, \sfrac{\de z}{z}$ in $\gr_F^{k+1-2j}\coH_\dR^1(\Gm, \Sym^k\Kl_2)$ is non-zero for all $j$ satisfying $0\leq 2j\leq k$ if $k$ is odd and $0\leq 2j\leq k/2$ if $k$ is even, since its pullback to~$\gr_F^{k+1-2j}\coH_\dR^1(\Gm, \Sym^k\wtKl_2)$ is non-zero.

If $k$ is odd, the non-vanishing of $\gr_F^{k+1-2j}\coH_\dR^1(\Gm, \Sym^k\Kl_2)$ for $j$ satisfying $0\leq 2j\leq k$ implies that this space is one-dimensional, the class of $z^j v_0^k\, \sfrac{\de z}{z}$ being a basis. Since these classes form a basis of $\coH_\dR^1(\Gm, \Sym^k\Kl_2)$, all other $\gr_F^p\coH_\dR^1(\Gm, \Sym^k\Kl_2)$ vanish. This concludes the proof of both Theorem \ref{th:main} and Proposition \ref{prop:FH1dR}\eqref{prop:FH1dR2} for odd $k$.

If $k$ is even, the same argument shows that $\gr_F^p\coH_\dR^1(\Gm, \Sym^k\Kl_2)$ is one-dimensional for $p=k+1,k-1,\dots, 2\lceil k/4\rceil+1$. Since $\gr_F^{k+1}\coH_\dR^1(\Gm, \Sym^k\wt\Kl_2)$ lies in weight $2k+2$ by Proposition~\ref{prop:main}\eqref{prop:maineven}, so does $\gr_F^{k+1}\coH_\dR^1(\Gm, \Sym^k\Kl_2)$ and the corresponding graded piece is one-dimensional. This yields the line $p=q=k+1$ in Theorem \ref{th:main}\eqref{th:main:itemeven}. On the other hand, since $k+1$ is odd, the space $\gr_{k+1}^W\coH_\dR^1(\Gm, \Sym^k\Kl_2)$ is even-dimensional by Hodge symmetry. Since $\coH_\dR^1(\Gm, \Sym^k\Kl_2)$ has dimension $\sfrac{k}{2}$ by \eqref{eq:dimSymk} and $\gr_{k+2}^W\coH_\dR^1(\Gm, \Sym^k\Kl_2)$ has dimension at most one, we get the equality\enlargethispage{-2\baselineskip}%
\[
\dim\gr_{k+2}^W\coH_\dR^1(\Gm, \Sym^k\Kl_2)=\begin{cases} 0 & k\not\equiv0\bmod4, \\ 1 & k\equiv0\bmod4.\end{cases}
\]

If $k\not\equiv0\bmod4$ (so that $2\lceil k/4\rceil+1=k/2+2$), the spaces $\gr_F^p\coH_\dR^1(\Gm, \Sym^k\Kl_2)$ lie in weight $k+1$ for $p=k-1,\dots, k/2+2$, and hence $\gr_F^p\coH_\dR^1(\Gm, \Sym^k\Kl_2)$ is one-dimensional for $p=2, 4, \dots, k/2-1$ by Hodge symmetry.

If $k\equiv0\bmod4$ (so that $2\lceil k/4\rceil+1=k/2+1$), then the space $\gr_{k+2}^W\coH_\dR^1(\Gm, \Sym^k\Kl_2)$ is one-dimensional, and hence $\gr_F^{k/2+1}\coH_\dR^1(\Gm, \Sym^k\Kl_2)$ lies in weight $k+2$. This gives the line $p=q=k/2+1$ in Theorem \ref{th:main}\eqref{th:main:itemeven}. To get the remaining Hodge numbers, we argue as above: the spaces $\gr_F^p\coH_\dR^1(\Gm, \Sym^k\Kl_2)$ lie in weight $k+1$ for $p=k-1,\dots, k/2+3$, and hence $\gr_F^p\coH_\dR^1(\Gm, \Sym^k\Kl_2)$ is one-dimensional for $p=2, 4, \dots, k/2-2$ by Hodge symmetry. This completes the proof of both Theorem \ref{th:main}\eqref{th:main:itemeven} and Proposition \ref{prop:FH1dR}\eqref{prop:FH1dR2} for even $k$.
\end{proof}

\section{\texorpdfstring{$L$}{L}-functions}\label{sec:L-functions}

In this section, we compute the $L$-function of the pure motive $\Motive_k$ over~$\QQ$ defined in~\eqref{eq:Mk}. We first compare, in Theorems~\ref{Thm:etale_k_odd} and~\ref{Thm:etale_k_even}, the traces of Frobenius at unramified primes of its~$\ell$\nobreakdash-adic realization with symmetric power moments of Kloosterman sums. These results largely overlap with Yun's \cite[Th.\,1.1.6]{Yun15}. Up to semi-simplification, the two approaches yield the same Galois representations as realizations of two different geometric models. In some sense, we have replaced the use of affine Grassmannians and homogeneous Fourier transformation in \cite{Yun15} with that of exponential mixed Hodge structures and the irregular Hodge filtration to obtain the easier geometric model~$\KM$ in terms of which the motive is defined (compare with \cite[\S 4.1.6]{Yun15}). One advantage of this point of view is that it enables us to determine the structure of the Galois representations at ramified primes
by means of the Picard-Lefschetz formula. In addition, we show in Proposition~\ref{Prop:p-adic_realization} that the Galois representations are crystalline at $p>k$ when $k$ is odd (\resp $p >\sfrac{k}{2}$ when $k$ is even) and we obtain lower bounds for the $p$-adic valuation of the eigenvalues of Frobenius in Corollary~\ref{cor:Newton-Hodge}. We then compute the gamma factor in Corollary \ref{Cor:Gamma-factor}.
Using the Patrikis-Taylor theorem, we finally prove that the motives~\eqref{eq:Mk} are potentially automorphic in the last subsection.
Theorems \ref{thm:intro-k-odd} and \ref{thm:intro-k-even} from the introduction follow by pulling everything together.

\subsection{\'Etale realizations}\label{Sect:etale-real}

\subsubsection{Cohomology of \texorpdfstring{$\Sym^k\Kl_2$}{SymkKl2} over finite fields}\label{sec:properties-characteristic-polynomial}

Recall the $\ell$-adic Kloosterman sheaf $\Kl_2$ on~$\Gm$ over $\FF_p$ from Section \ref{sect:etale_char_p}. In this paragraph, we gather the main properties of the \'etale cohomology of its symmetric powers. All results below are due to Fu-Wan \hbox{\cite[Th.\,0.2]{FW05}} and Yun \cite[Lem.\,4.2.1, Cor.\,4.2.3 and 4.3.5]{Yun15}, who prove them by means of a \hbox{thorough} study of the structure of $\Sym^k\Kl_2$ at zero and infinity. Throughout, $F_p$ denotes the geometric Frobenius in~$\Gal(\ol\FF_p/\FF_p)$ and we consider the reciprocal characteristic polynomials
\begin{align*}
Z_k(p; T) &= \det\Bigl( 1-F_pT \mid
	\coH_{\et,\cp}^1(\Gmolf, \Sym^k\Kl_2) \Bigr) \\
	M_k(p; T) &= \det \Bigl( 1-F_pT \mid \coH^1_{\et, \rmid}(\Gmolf, \Sym^k \Kl_2)\Bigr).
\end{align*}
\begin{itemize}
\item
If $k$ is odd,
then
\begin{equation}\label{eq:degZ_kodd}
\deg Z_k(p; T) = \begin{cases}
\dfrac{k+1}{2}, & p =2 \\[10pt]
\dfrac{k+1}{2} - \flr{\dfrac{k}{2p} + \dfrac{1}{2}}, & p\geq 3,
\end{cases}
\end{equation}
and there is a factorization
\begin{equation}\label{eq:decomp_Z_odd}
Z_k(p; T) = (1-T)M_k(p; T),
\end{equation}
where the reciprocal roots of $M_k(p; T)$ are Weil numbers of weight $k+1$.

\smallskip

\item
If $k$ is even and $p$ is odd, then
\[
\deg Z_k(p; T)= \frac{k}{2} - \flr{\frac{k}{2p}}
\]
and there is a factorization
\begin{equation}\label{eq:decomp_Z_even}
Z_k(p; T) = (1-T)R_k(p; T)M_k(p; T)
\end{equation} such that
the reciprocal roots of $M_k(p; T)$ are again Weil numbers of weight $k+1$. Above, the polynomial $R_k(p; T)$ is given by
\[
R_k(p; T) = \bigl( 1- (-1)^{(p-1)/2}p^{k/2}T \bigr)^{n_k(p)}
	\bigl( 1- p^{k/2}T \bigr)^{m_k(p)-n_k(p)},
\]
\[
n_k(p) = \flr{\frac{k}{4p} + \frac{1}{2}},
\quad
m_k(p) = \flr{\frac{k}{2p}} + \delta_{4\ZZ}(k).
\]

\smallskip

\item There is also an explicit description for even $k$ and $p=2$ in \cite[Lem.\,4.3.4, Cor.\,4.3.5]{Yun15},
namely $Z_k(2; T)$ has degree $\flr{\psfrac{k+2}{4}}$ and factors as
\begin{equation}\label{eq:Z_k_2_T}
Z_k(2; T) =
(1-T) \bigl( 1- 2^{k/2}T \bigr)^{a_k} \bigl( 1+2^{k/2}T \bigr)^{b_k}
M_k(2; T),
\end{equation}
where
$\deg M_k = 2\flr{\psfrac{k+2}{12}} - 2\delta_{12\ZZ}(k)$
and the exponents $a_k$ and $b_k$ are given by
\begin{align*}
a_k&=\begin{cases}
\flr{\frac{k}{24}}+1, & k \equiv 0, 8, 12, 16, 18, 20 \mod{24} \\
\flr{\frac{k}{24}}, & k \equiv 2, 4, 6, 10, 14, 22 \mod{24} \\
\end{cases} \\
b_k&=\begin{cases}
\flr{\frac{k}{24}}+1, & k \equiv 6, 12, 14, 18, 20, 22 \mod{24} \\
\flr{\frac{k}{24}}, & k \equiv 0, 2, 4, 8, 10, 16 \mod{24}. \\
\end{cases}
\end{align*}
\end{itemize}
In all three cases,
the reciprocal roots $\alpha$ of the polynomial $M_k(p; T)$
are stable under the transformation $\alpha \mapsto p^{\psfrac{k+1}{2}}\alpha^{-1}$,
which reflects the self-duality of the middle cohomology.

\subsubsection{Galois representations of symmetric power moments}

Recall from \eqref{eq:g_boxplus_k} the Laurent polynomial $g^{\boxplus k} =\sum_{i=1}^k (y_i + \sfrac{1}{y_i})$ on the torus $\Gm^k$ and its zero locus \hbox{$\KM\subset\Gm^k$}. For each prime number $\ell$, the $\ell$-adic realization
of the motive~$\Motive_k$ is the $\QQ_\ell$-vector~space
\begin{equation}\label{eq:defGaloisreps}
V_{k,\ell} = \gr^W_{k-1}\coH_{\et,\cp}^{k-1}(\KM_{\ol\QQ}, \QQ_\ell)^{
	\symgp_k\times\mu_2,\,\chi}(-1)
\end{equation} equipped with the continuous representation
\begin{equation*}
r_{k, \ell} \colon \Gal(\ol{\QQ} / \QQ) \longrightarrow \GL(V_{k, \ell}).
\end{equation*} Writing $k=2m+1$ for odd $k$ and $k=2m+2$ or $k=2m+4$ with $m$ an even integer for even~$k$,
the vector space $V_{k,\ell}$ is $m$-dimensional by \eqref{eqn:dimensionHmid}.

The goal of the next two sections is to compare the traces of Frobenius at unramified primes with symmetric power moments of Kloosterman sums. For this, we shall consider the toric compactification $X$ of $\Gm^k$ introduced in Section \ref{sec:toric-compactification} and let $\ol{\KM}$ be the closure of $\KM$ in~$X$. We also regard these varieties as defined over general rings (\eg over $\FF_p, \ZZ_p$, etc.). We claim that~$\ol{\KM}$ is smooth along the strict normal crossing divisor $D=X\moins\Gm^k$ and that each irreducible component of $D$ intersects $\ol{\KM}$ in a smooth divisor in such a way that~$\ol{\KM}\setminus\KM$ forms a relative strict normal crossing divisor over~$\ZZ$. Indeed, it is enough to check these properties on each of the $2^kk!$ affine charts of $X$ corresponding to the cones of maximal dimension of the simplicial fan~$F$. For example, on the chart $X_1 \cong \IA^k = \Spec(\ZZ[z_i]_{i=1}^k)$, the function~$g^{\boxplus k}$ is given by \hbox{$g^{\boxplus k} = \sfrac{g_1}{z_1\cdots z_k}$} and the equality $\ol{\KM}\cap X_1 = (g_1)$ holds, with $g_1$ as in \eqref{eq:g1_on_Z1}. One then checks the equalities
\begin{align*}
(g_1)\cap(z_1)&=\emptyset, \quad (g_1)\cap(z_r)=\bigl(1+z_1(1+z_2+\cdots+z_2\cdots z_{r-1})\bigr) \\ &(\partial g_1/\partial z_1)\cap(z_r)=\bigl(1+z_2+\cdots+z_2\cdots z_{r-1}\bigr)
\end{align*} for $r=2,\dots,k$. From the first two, it follows that $\partial g_1/\partial z_1$ does not vanish on~$(z_r)$ for~$r\geq2$, hence the smoothness of~$\ol{\KM}$ along $(z_1\cdots z_r)$. The smoothness of $(g_1)\cap(z_r)$ is also clear, and~$\ol{\KM}\cap(z_{r_1})\cap\cdots\cap(z_{r_i})$ is smooth for any sequence $2\leq r_1\leq\cdots\leq r_i\leq k$, which implies the strict normal crossing property. Besides, over $\QQ$, the variety $\KM\subset\Gm^k$ is smooth when $k$ is odd, while if $k$ is even, its singular locus consists of $\binom{k}{k/2}$ ordinary quadratic points with coordinates $y_i \in\{\pm 1\}$ satisfying $\sum_{i=1}^k y_i= 0$.

\subsubsection{The $\ell$-adic case for odd symmetric powers}
\label{sect:l-adic_k_odd}

Let $k \geq 1$ be an odd integer and $p$ an odd prime number.
The singular locus $\Sigma$ of $\KM$ over $\ol\FF_p$
consists of $\flr{\sfrac{k}{2p}+\sfrac{1}{2}}$
orbits of ordinary quadratic points
under the action of $\symgp_k\times\mu_2$.
Indeed, the orbits are indexed by
odd positive integers $a$ such that $ap\leq k$, each of them being represented by the point with coordinates $y_i = 1$ (\resp $-1$)
for $1\leq i \leq \psfrac{ap+k}{2}$ (\resp $i>\psfrac{ap+k}{2}$).
Locally around this point, writing $y_i = z_i+1$ (\resp $y_i=z_i-1$), the defining equation of $\KM$ in $\ZZ_p[\![z_1,\dots,z_k]\!]$ is given by
\begin{displaymath}
g^{\boxplus k}(z_1, \dots, z_k)=2ap + Q_{ap}+\text{higher order terms},
\end{displaymath}
where $Q_{ap}$ is the non-degenerate quadratic form
\begin{equation}\label{eq:KM_char_p}
Q_{ap}= \sum_{i\leq (ap+k)/2} z_i^2 - \sum_{i>(ap+k)/2} z_i^2.
\end{equation}

Write $k=2m+1$. After choosing a place of $\ol\QQ$ above $p$,
with each $x \in \Sigma$ is associated a \textit{vanishing cycle class} $\delta_x$
in $\coH_{\et}^{k-1}(\ol\KM_{\ol\QQ}, \QQ_\ell)(m)$ that is
well defined up to sign. Letting $\gen{\,,\,}$ denote the pairing obtained from the intersection form and the identification
$\coH_\et^{2k-2}(\ol\KM_{\ol\QQ},\QQ_\ell)(2m) \cong \QQ_\ell$
given by the trace, these classes satisfy
\[
\gen{\delta_x,\delta_y} = \begin{cases}
(-1)^m2 & \text{if } x=y, \\
0 & \text{if }x\neq y.
\end{cases}
\]
By the Picard-Lefschetz formula
\cite[Exp.\,XV, Th.\,3.4]{SGA7}, there is an exact sequence
\begin{displaymath}
0 \longrightarrow \coH_\et^{k-1}(\ol\KM_{\ol\FF_p},\QQ_\ell)
\longrightarrow \coH_\et^{k-1}(\ol\KM_{\ol\QQ},\QQ_\ell)
\stackrel{\gamma}{\longrightarrow} \bigoplus_{x\in\Sigma} \QQ_\ell(-m) \longrightarrow 0,
\end{displaymath}
where the map $\gamma$ is given by taking pairings with $\delta_x$.

In what follows, we keep the notation $\zeta$ for a primitive $p$-th root of unity in $\overline{\QQ}_\ell$,
denote by $-[\zeta]$ the scalar extension $-\otimes_{\QQ_\ell}\QQ_\ell(\zeta)$,
and set
\begin{align*}
\Theta_p^+ &= \{ \text{$a \geq 1$ odd integer}
	\mid \text{$ap \leq k$ with $v_p(a)$ odd} \}, \\
\Theta_p^- &= \{ \text{$a \geq 1$ odd integer} \mid
	\text{$ap \leq k$ with $v_p(a)$ even} \},
\end{align*}
so that the following equality holds:
\begin{equation}\label{eq:sum+and-}
|\Theta_p^+| + |\Theta_p^-| = \flr{\frac{k}{2p} + \frac{1}{2}}.
\end{equation}

\begin{thm}\label{Thm:etale_k_odd}
Let $k=2m+1$ be a positive odd integer, and let $p$ and $\ell$ be distinct prime numbers. Let $V_{k,\ell}$ denote the $\ell$-adic realization
of the motive $\Motive_k$, which is an~$m$\nobreakdash-dimensional $\QQ_\ell$-representation of $\Gal(\ol \QQ/\QQ)$. Fix a place of $\ol\QQ$ above $p$
and let~$I_p$ be the corresponding inertia subgroup of $\Gal(\ol\QQ_p/\QQ_p) \subset \Gal(\ol \QQ/\QQ)$.

\begin{enumerate}
\item\label{thm:structure1} The representation $V_{k,\ell}$ is unramified at $2$ and the $\Gal(\ol{\FF}_2/\FF_2)$-module $V_{k,\ell}[\zeta]$
is isomorphic to
$\coH_{\et,\rmid}^1({\mathbb{G}}_{\mathrm{m}, \ol{\FF}_2}, \Sym^k\Kl_2)$.
\smallskip
\item\label{thm:structure2} If $p$ is an odd prime, then $V_{k,\ell}$ is at most tamely ramified at $p$.
More precisely, the restriction of $V_{k,\ell}$ to $\Gal(\ol\QQ_p/\QQ_p)$
decomposes into an orthogonal sum $M\oplus E$, where
\begin{itemize}
\item
$M[\zeta] = \coH_{\et,\rmid}^1(\Gmolf, \Sym^k\Kl_2)$,
\smallskip
\item
$E$ is generated by vanishing cycles, one for each $a \in \Theta_p^+ \cup \Theta_p^{-},$ on which the Galois group acts through the character $\varepsilon_a \otimes \chi_{\mathrm{cyc}}^{-m-1},$ where \hbox{$\varepsilon_a \colon \Gal(\ol \QQ_p / \QQ_p) \to \{\pm 1\}$} stands for the primitive character associated with the extension
\[\QQ_p\biggl(\sqrt{(-1)^{\sfrac{(1+ap)}{2}} 2ap}\biggr) \quad \text{of}\quad \QQ_p.
\]
\end{itemize}

\smallskip

\noindent In particular, decomposing $E=E^+ \oplus E^{-}$ according to whether $a$ belongs to $\Theta_p^+$ or $\Theta_p^{-}$, the invariants under inertia are $V_{k, \ell}^{I_p}=M \oplus E^+$ and $E^+$ is a semi-simple $\Gal(\ol{\FF}_p/\FF_p)$\nobreakdash-module with reciprocal characteristic polynomial
of Frobenius
\begin{equation*}\label{eq:actionFrobeniusE}
\det ( 1-F_pT \mid E^+) = \prod_{a \in \Theta_p^+}
	\biggl( 1-\biggl(\frac{(-1)^{(1+ap)/2}2a'}{p} \biggr)p^{m+1}T \biggr),
\end{equation*}
where $a' = ap^{-v_p(a)}$ denotes the prime-to-$p$ part of $a$
and $\left(\sfrac{\cdot}{p}\right)$ the Legendre symbol.
\end{enumerate}
\end{thm}

\begin{proof}
There is nothing to prove if $k=1$,
so we assume $k\geq 3$.
To shorten notation,
we omit the coefficients $\QQ_\ell$
from the \'etale cohomology groups and write $G=\symgp_k \times \mu_2$, so that
\[
V_{k,\ell}= \gr^W_{k-1} \coH_{\et,\cp}^{k-1}(\KM_{\ol\QQ})^{G,\chi}(-1).
\]
Set $\ol\KM^{(0)} = \ol\KM$
and, for each $i \geq 1$, let $\ol\KM^{(i)}$ denote the disjoint union of all $i$-fold intersections
of distinct irreducible components of $\ol\KM\setminus\KM$. The spectral sequence
\begin{equation}
E_1^{i, j} = \coH_\et^j(\ol\KM^{(i)}_{\ol\FF})
\Longrightarrow \coH_{\et,\cp}^{i+j}(\KM_{\ol\FF})
\qquad (i, j \geq 0)
\end{equation} computes the étale cohomology with compact support of $\KM$ over the fields $\FF=\QQ$ or $\FF_p$.

For $\FF=\QQ$,
the spectral sequence degenerates at $E_2$
since, $\ol\KM^{(i)}_\QQ$ being a smooth proper variety for all $i \geq 0$,
the source and the target of the differentials in the second page are pure of different weights. The graded piece of weight $k-1$ is thus isomorphic to $E_2^{0, k-1}$ and
\begin{equation}\label{eqn:isom1graded}
\begin{aligned}
\gr^W_{k-1} \coH_{\et,\cp}^{k-1}(\KM_{\ol\QQ})
&= \ker \left\{ \coH_\et^{k-1}(\ol\KM_{\ol\QQ})
	\to \coH_\et^{k-1}(\ol\KM^{(1)}_{\ol\QQ}) \right\} \\
&= \image \left\{ \coH_{\et,\cp}^{k-1}(\KM_{\ol\QQ})
	\to \coH_\et^{k-1}(\ol\KM_{\ol\QQ}) \right\},
\end{aligned}
\end{equation}
where the second map is the surjective edge map from the abutment $\coH_{\et,\cp}^{k-1}(\KM_{\ol\QQ})$ to $E_2^{0, k-1}$. On the other hand, by Theorem \ref{thm:etalerealisation}, there is an isomorphism
\[
\coH_{\et,\rmid}^1(\Gmolf, \Sym^k\Kl_2)\cong\gr^W_{k-1} \coH_{\et,\cp}^{k-1}(\KM_{\ol\FF_p})^{G, \chi}(-1)[\zeta].
\]
This cohomology is pure of weight~\hbox{$k+1$}
and has dimension $m$ if $p=2$ and \hbox{$m-\flr{\sfrac{k}{2p} + \sfrac{1}{2}}$}
if $p\geq 3$ by \eqref{eq:degZ_kodd}.
For each $i \geq 1$,
the variety $\ol\KM^{(i)}_{\ol\FF_p}$ is smooth and proper.
Hence the étale cohomology $\coH_\et^{k-1-i}(\ol\KM^{(i)}_{\ol\FF_p})$
is pure of weight $k-1-i$,
and so is $E_\infty^{i,k-1-i}$
in the above spectral sequence for $\FF=\FF_p$.
The only contribution of weight $k-1$ is thus given by
\begin{equation}
\label{eqn:isom2graded}
\gr^W_{k-1} \coH_{\et,\cp}^{k-1}(\KM_{\ol\FF_p})
= \gr^W_{k-1}\image \Bigl\{ \coH_{\et,\cp}^{k-1}(\KM_{\ol\FF_p})
	\to \coH_\et^{k-1}(\ol\KM_{\ol\FF_p}) \Bigr\}.
\end{equation}

Assume $p=2$.
The proper variety $\ol\KM_{\ol\FF_2}$
has a quadratic non-ordinary isolated singularity, namely the point with coordinates
$y_i=1$.
When this is the case, the cohomology sheaf $\Rder^n\Phi$
of the vanishing cycle complex on $\ol\KM_{\ol\FF_2}$
is non-zero only in degree $n=k-1$ by
\hbox{\cite[Cor.\,2.10]{Illusie},} which implies that the cospecialization morphism
$\coH_\et^{k-1}(\ol\KM_{\ol\FF_2})
\to \coH_\et^{k-1}(\ol\KM_{\ol\QQ})$
is injective. From the isomorphism \eqref{eqn:isom1graded}
and the commutativity of the square
\[\xymatrix{
\coH_{\et,\cp}^{k-1}(\KM_{\ol\FF_2}) \ar[r]\ar[d]
& \coH_{\et,\cp}^{k-1}(\KM_{\ol\QQ}) \ar[d] \\
\coH_\et^{k-1}(\ol\KM_{\ol\FF_2}) \ar[r]
& \coH_\et^{k-1}(\ol\KM_{\ol\QQ}),}
\]
we deduce an injection
\begin{equation}\label{eq:p=2_Vkl}
\image \Bigl\{ \coH_{\et,\cp}^{k-1}(\KM_{\ol\FF_2})
	\to \coH_\et^{k-1}(\ol\KM_{\ol\FF_2}) \Bigr\}^{G,\chi}
\to \gr^W_{k-1} \coH_{\et,\cp}^{k-1}(\KM_{\ol\QQ})^{G,\chi}
= V_{k,\ell}(1).
\end{equation}
Since the $\chi$-isotypic part of the left-hand side of \eqref{eqn:isom2graded}
has dimension $m=\dim V_{k,\ell}$, it follows that \eqref{eq:p=2_Vkl} is indeed an isomorphism, hence an isomorphism
\[
V_{k,\ell}= \gr^W_{k-1} \coH_{\et,\cp}^{k-1}(\KM_{\ol\FF_2})^{G,\chi}(-1)
\]
of representations of $\Gal(\ol\FF_2/\FF_2)$. This concludes the proof of~\eqref{thm:structure1}.

Now suppose $p\geq 3$ and consider the $G$-equivariant commutative diagram with exact rows and columns
\[\xymatrix{
& \coH_{\et,\cp}^{k-1}(\KM_{\ol\FF_p}) \ar[r]\ar[d]_\beta
& \coH_{\et,\cp}^{k-1}(\KM_{\ol\QQ}) \ar[d]^\alpha \\
0 \ar[r] & \coH_\et^{k-1}(\ol\KM_{\ol\FF_p}) \ar[r]\ar[d]
& \coH_\et^{k-1}(\ol\KM_{\ol\QQ}) \ar[r]^(.43)\gamma\ar[d]
& \bigoplus_{x\in\Sigma} \QQ_\ell(-m) \ar[r] & 0, \\
& \coH_\et^{k-1}(\ol\KM^{(1)}_{\ol\FF_p}) \ar[r]^\sim
& \coH_\et^{k-1}(\ol\KM^{(1)}_{\ol\QQ}) }
\]
in which the middle row is given by the Picard-Lefschetz formula. Let $\Delta = \bigoplus_{x\in \Sigma} \QQ_\ell(-m)\delta_x$
be the subspace of $\coH_\et^{k-1}(\ol\KM_{\ol\QQ})$
generated by vanishing cycle classes, which is the orthogonal complement of the image of
$\coH_\et^{k-1}(\ol\KM_{\ol\FF_p})$.
Let $\Rder\Psi$ be the complex of nearby cycles on $\ol\KM_{\ol\FF_p}$.
Since $\delta_x$ is a generator of the local cohomology
$\coH_{\{x\}}^{k-1}(\ol\KM_{\ol\FF_p}, \Rder\Psi(m))$
with support~$\{x\}$ contained in $\KM_{\ol\FF_p}$, the subspace
$\Delta$ lies in the image of $\alpha$.
The image of $\beta$ and $\Delta$ being orthogonal as subspaces of
$\coH_\et^{k-1}(\ol\KM_{\ol\QQ})$, the equalities
\[
\image(\beta)^{G,\chi} = \gr^W_{k-1}\image(\beta)^{G,\chi} \quad \text{and}\quad
\image(\alpha)^{G,\chi} = \image(\beta)^{G,\chi} \oplus \Delta^{G,\chi}
\]
hold, with $\dim \Delta^{G,\chi} = \flr{\sfrac{k}{2p} + \sfrac{1}{2}}$
by a dimension count and \eqref{eq:decomp_Z_odd}.
These are the factors~$M$ and $E$ in part~\eqref{thm:structure2} of the theorem.

To compute the Galois action on~$E$, recall the quadratic form $Q_{ap}$ from \eqref{eq:KM_char_p} and consider the projective quadric $D = (2apw^2 + Q_{ap})$ in $\PP_{\QQ_p}^k$, as well as the hyperplane section $C = D \cap (w)$. The space $\QQ_\ell(-m)\delta_x$ is described in \cite[Exp.\,XV, Prop.\,2.2.3]{SGA7} as $\coH_{\et,\cp}^{k-1}((D\setminus C)_{\ol\QQ_p})$, which is equal to the primitive part of $\coH_\et^{k-1}(D_{\ol\QQ_p})$ by the localization sequence for \'etale cohomology with compact support. As a non-degenerate quadratic form over $\QQ_p$, the defining equation of~$D$ has discriminant $d=(-1)^{\sfrac{(k-ap)}{2}}2ap$, and hence $\Gal(\ol\QQ_p \slash \QQ_p)$ acts on
\begin{displaymath}
\det \coH_\et^{k-1}(D_{\ol\QQ_p}, \QQ_\ell(m))=\coH^{k-1}_{\et, \textup{prim}}(D_{\ol\QQ_p})(m)
\end{displaymath}
via the character $\varepsilon_a$
corresponding to the extension
$\QQ_p(\sqrt{(-1)^{\sfrac{(k+1)}{2}}d})$ by~\cite[\S 5.2]{Saito12}.
Noting the extra twist in the expression \eqref{eq:defGaloisreps} of $V_{k, \ell}$ in terms of $\coH^{k-1}_{\et}(\KM_{\ol\QQ})$, this proves the first statement about $E$. This extension is unramified if and only if $v_p(a)$ is odd, in which case it is equal to $\QQ_p(\sqrt{(-1)^{(1+ap)/2}2a'})$ and the last assertion in \eqref{thm:structure2} follows.
\end{proof}

\begin{remark} In the case at hand, instead of invoking \cite{Saito12} one can directly see the action of $\Gal(\ol\QQ_p/\QQ_p)$ on the primitive cohomology of the quadric by regarding it as defined over~$\ZZ$. Indeed, $D$ has good reduction at all primes $r$ with even $v_r(d)$ and, \eg by point counting over $\FF_r$, Frobenius acts as multiplication by $\left(\frac{(-1)^{(k+1)/2}d}{r}\right)r^{m}.$ Chebotarev's density theorem then implies that $\Gal(\ol\QQ / \QQ)$ acts on $\coH^{k-1}_{\et, \textup{prim}}(D_{\ol \QQ})(m)$ through the character corresponding to the extension $\QQ(\sqrt{(-1)^{(1+ap)/2}2ap})$. Note that $D$ has good reduction at~$p$ if $v_p(a)$ is odd.
\end{remark}

From Theorem~\ref{Thm:etale_k_odd} and Serre's recipe \cite{Serre70}, we immediately derive the local $L$-factors and the conductor of the system of Galois representations $\{V_{k,\ell}\}_\ell$
associated with the motive~$\Motive_k$. For each prime $p$, define $L_k(p; T)$ as the reciprocal of the polynomial with integer coefficients
\begin{equation}\label{eq:bad-factors-odd-case}
\det ( 1-F_pT \mid V_{k, \ell}^{I_p})=\begin{cases}
M_k(2; T) & \text{if }p = 2, \\
M_k(p; T) \prod_{a \in \Theta_p^+}
\left( 1-\left(\frac{(-1)^{(1+ap)/2}2a'}{p}\right)p^{m+1}T \right) & \text{if }p\geq 3.
\end{cases}
\end{equation}
The $L$-function of $\Motive_k$ is the Euler product
\[
L_k(s) = \prod_p L_k(p; p^{-s}),
\]
which converges absolutely on the half-plane $\mathrm{Re}(s)>1+\sfrac{(k+1)}{2}$.

Recall from \cite[eq.\,(11) and (29)]{Serre70} that the exponent of $p$ in the global conductor of $\{V_{k, \ell}\}_\ell$ is given by the sum of the Swan conductor of the restriction of $V_{k,\ell}$ to $\Gal(\ol\QQ_p/\QQ_p)$ and the codimension of $V_{k, \ell}^{I_p}$. Since $V_{k, \ell}$ is at most tamely ramified at all primes $p \neq \ell$, the Swan conductor vanishes and we are left with $\sfrac{(k-1)}{2}-\dim V_{k,\ell}^{I_p}$, which in view of \eqref{eq:degZ_kodd}, \eqref{eq:decomp_Z_odd}, and \eqref{eq:sum+and-} is equal to~$|\Theta_p^-|$ if $p$ is odd and to zero if $p=2$. Thus, the value of the conductor is
\begin{equation}\label{eq:conductor_k_odd}
\cond_k = \prod_{\text{$p$ odd}} p^{|\Theta_p^-|}
= 1_\mathrm{s} 3_\mathrm{s} 5_\mathrm{s} \cdots k_{\mathrm{s} },
\end{equation} where $n_\mathrm{s}$ denotes the product of all primes~$p$
such that $v_p(n)$ is odd.

\begin{remark}\label{rmk:L_ss_k_odd}
It is clear from Theorem~\ref{Thm:etale_k_odd} that both the $L$-factor $L_k(p;T)$
and the conductor~$\cond_k$
remain unchanged if one replaces the input $\{V_{k,\ell}\}_\ell$
with its semi-simplification~$\{V_{k,\ell}^\ss\}_\ell$.
\end{remark}

\subsubsection{The $\ell$-adic case for even symmetric powers}

Let $k \geq 2$ be an even integer and $p$ an odd prime number.
In this case, the singular locus of $\KM_{\ol\FF_p}$
consists of $1+\flr{\sfrac{k}{2p}}$
orbits of ordinary quadratic points
under the action of $\symgp_k\times\mu_2$.
They are indexed by
non-negative even integers~$b$ satisfying $bp\leq k$,
with points with coordinates $y_i = 1$ (\resp $-1$)
for $1\leq i \leq \psfrac{bp+k}{2}$ (\resp $i>\psfrac{bp+k}{2}$) as representatives.
Writing $y_i = z_i+1$ (\resp $y_i=z_i-1$),
locally around each singularity the defining equation of $\KM$ in $\ZZ_p[\![z_1,\cdots,z_k]\!]$ has the shape
\begin{equation}\label{eq:KM_bp_char_p}
2bp + Q_{bp}+\text{higher order terms},
\quad
Q_{bp}= \sum_{i\leq (bp+k)/2} z_i^2 - \sum_{i>(bp+k)/2} z_i^2.
\end{equation}

\begin{thm}\label{Thm:etale_k_even}
Let $k$ be a positive even integer, either of the form $2m+2$ or $2m+4$ with even~$m$, and let $p$ and $\ell$ be distinct prime numbers with $p \geq 3$. Let $V_{k,\ell}$ be the $\ell$\nobreakdash-adic realization
of the motive $\Motive_k$, which is an $m$-dimensional $\QQ_\ell$-representation of~$\Gal(\ol \QQ/\QQ)$. Fix a place of~$\ol\QQ$ above $p$ and let~$I_p \subset \Gal(\ol\QQ_p/\QQ_p) \subset \Gal(\ol \QQ/\QQ) $ be the corresponding inertia subgroup.

Then $I_p$ acts unipotently on $V_{k,\ell}$. More precisely, the equality $(\sigma-1)^2=0$ holds for each~\hbox{$\sigma \in I_p$} acting on~$V_{k, \ell}$ and there exists an isotropic subspace $U \subset V_{k,\ell}$
of dimension~$\flr{\sfrac{k}{2p}}$, the image of the logarithm of the monodromy operator,
generated by vanishing cycles and such that the equality~\hbox{$V_{k,\ell}^{I_p} = U^\perp$} holds and that the induced map~$\sigma - 1\colon V_{k,\ell} \to V_{k,\ell}/U$ is zero. Moreover, there is an isomorphism of $\Gal(\ol{\FF}_p/\FF_p)$-modules
\begin{displaymath}
V_{k,\ell}^{I_p}[\zeta]=\coH_{\et,\cp}^1({\mathbb{G}}_{\mathrm{m}, \ol{\FF}_p}, \Sym^k\Kl_2)/E[\zeta],
\end{displaymath}
where $E$ is the trivial representation $\QQ_\ell(0)$ if $4$ does not divide $k$
and an extension of $\QQ_\ell(\sfrac{-k}{2})$ by $\QQ_\ell(0)$ otherwise.
\end{thm}

\begin{proof} There is nothing to prove for $k=2$, so we assume $k\geq 4$. Again, we set $G = \symgp_k\times\mu_2$ and we omit the coefficients $\QQ_\ell$ from the \'etale cohomology groups. Let us first recall from Theorem \ref{th:main} that, in characteristic zero, the mixed Hodge structure $\coH_\cp^{k-1}(\KM)^{G,\chi}/W_0$ is pure of weight~$k-1$
and has dimension $\psfrac{k-2}{2}$ if $k \equiv 2 \bmod{4}$, whereas if $k \equiv 0 \bmod{4}$, it is mixed of weights~$k-2$ and $k-1$
with graded pieces of dimension $1$ and $\psfrac{k-4}{2}$, respectively.

Let $S$ be the singular locus of $\KM$ in characteristic zero, which consists of $\binom{k}{k/2}$ ordinary quadratic points, and let $\KM'$ be the strict transform of $\KM$ inside the blow-up of $\Gm^k$ at $S$.
The preimage of $S$ in $\KM'$ is
a disjoint union of quadrics that we denote by $T$.
Consider the commutative diagram
\[\begin{array}{ccc}
T & \subset & \KM' \\
\big\downarrow && \big\downarrow \\
S & \subset & \KM
\end{array}
\]
and the corresponding commutative diagram with exact rows
\[\xymatrix{
\NM^{k-2}(T) \ar[r] &
\NM_\cp^{k-1}(\KM' \setminus T) \ar[r]\ar[d]^{\cong} &
\NM_\cp^{k-1}(\KM') \ar[r]\ar[d] & \NM^{k-1}(T) \\
\NM^{k-2}(S) \ar[r] &
\NM_\cp^{k-1}(\KM \setminus S) \ar[r] &
\NM_\cp^{k-1}(\KM) \ar[r] & \NM^{k-1}(S).
}
\]
Since $k$ is even and $\geq 4$,
the vanishing $\NM^{k-2}(S) = \NM^{k-1}(S) = \NM^{k-1}(T) = 0$ holds,
and hence we get an isomorphism
\begin{equation}\label{eq:coh_K'_to_K}
\NM_\cp^{k-1}(\KM') \To{\sim}
\NM_\cp^{k-1}(\KM).
\end{equation}
The above isomorphism remains true
if one replaces $\NM_\cp^{k-1}(\KM)$ and $\NM_\cp^{k-1}(\KM')$
with $\coH_{\et,\cp}^{k-1}(\KM_{\ol\FF})$ and $\coH_{\et,\cp}^{k-1}(\KM'_{\ol\FF})$
respectively, for $\FF = \QQ$ or $\FF_p$.

Recall the compactification $X$ of $\Gm^k$. Let $\ol{\KM}'$ be the closure of $\KM'$
in the blow-up of $X$ along $S$, and let $\ol\KM'^{(i)}$ be the disjoint union of all $i$-fold intersections
of distinct irreducible components of the boundary divisor $\ol\KM' \setminus \KM'$, with the usual convention $\ol\KM'^{(0)}=\ol\KM'$.
Consider the associated spectral sequence
\begin{equation}
(E_1^{i, j})_\FF = \coH_\et^j(\ol\KM'^{(i)}_{\ol\FF})
\Longrightarrow \coH_{\et,\cp}^{i+j}(\KM'_{\ol\FF})
\qquad (i, j \geq 0).
\end{equation}
In characteristic zero, since all $\ol\KM'^{(i)}_{\ol\QQ}$ are smooth and proper,
the spectral sequence degenerates at $E_2$
and one gets
\begin{displaymath}
\begin{aligned}
\gr^W_{k-1} \coH_{\et,\cp}^{k-1}(\KM'_{\ol\QQ})
&= \ker \Bigl\{ \coH_\et^{k-1}(\ol\KM'_{\ol\QQ})
	\to \coH_\et^{k-1}(\ol\KM'^{(1)}_{\ol\QQ}) \Bigr\} \\
&= \image \Bigl\{ \coH_{\et,\cp}^{k-1}(\KM'_{\ol\QQ})
	\stackrel{\alpha}{\to} \coH_\et^{k-1}(\ol\KM'_{\ol\QQ}) \Bigr\}
\end{aligned}
\end{displaymath}
exactly as in the case where $k$ is odd. The equality
\begin{equation}\label{eq:gr_KMp_Q}
V_{k,\ell}(1)=\image(\alpha)^{G, \chi}
\end{equation}
then follows by taking $\chi$-isotypic components. Moreover, the $E_2^{1,k-2}$-term reads
\begin{equation}\label{eq:E_2_1_k-2}
\gr^W_{k-2} \coH_{\et,\cp}^{k-1}(\KM'_{\ol\QQ}) =\frac{\ker\Bigl\{ \coH_\et^{k-2}(\ol\KM'^{(1)}_{\ol\QQ})
	\to \coH_\et^{k-2}(\ol\KM'^{(2)}_{\ol\QQ}) \Bigr\}}{\image\Bigl\{ \coH_\et^{k-2}(\ol\KM'_{\ol\QQ})
	\to \coH_\et^{k-2}(\ol\KM'^{(1)}_{\ol\QQ}) \Bigr\}}.
\end{equation}
By \eqref{eq:coh_K'_to_K}, the right-hand side is isomorphic to $\gr^W_{k-2} \coH_{\et,\cp}^{k-1}(\KM_{\ol\QQ})$, and hence its $\chi$-isotypic component has dimension $\delta_{4\ZZ}(k)$ as recalled at the beginning of the proof.

Since the singularities of $\ol\KM'_{\FF_p}$
consist only of ordinary quadratic points
supported on $\KM'_{\FF_p}$,
the Picard-Lefschetz formula
and base-change yield isomorphisms
\begin{align*}
\coH_\et^n(\ol\KM'_{\ol\FF_p}) &\To{\sim}
	\coH_\et^n(\ol\KM'_{\ol\QQ})
	\qquad \text{for }n \leq k-2, \\
\coH_\et^j(\ol\KM'^{(i)}_{\ol\FF_p}) &\To{\sim}
	\coH_\et^j(\ol\KM'^{(i)}_{\ol\QQ}),
	\quad\hspace{.5mm} \text{for }i \geq 1.
\end{align*}
In particular,
the equality $(E_2^{i, j})_{\FF_p} = (E_2^{i, j})_{\QQ}$
holds for all $i+j = k-1$ with $i\geq 1$, hence the degeneration $(E_2^{1,k-2})_{\FF_p}^{G,\chi}=(E_\infty^{1,k-2})_{\FF_p}^{G,\chi}$. By \eqref{eq:E_2_1_k-2},
this space vanishes for $k\equiv 2 \bmod{4}$ and is one-dimensional of weight $k-2$ if $4$ divides~$k$. Consider again the $G$-equivariant commutative diagram of $\Gal(\ol\QQ_p/\QQ_p)$-representations
\[\xymatrix{
& \coH_{\et,\cp}^{k-1}(\KM'_{\ol\FF_p}) \ar[r]\ar[d]_\beta
& \coH_{\et,\cp}^{k-1}(\KM'_{\ol\QQ}) \ar[d]^\alpha \\
0 \ar[r] & \coH_\et^{k-1}(\ol\KM'_{\ol\FF_p}) \ar[r]
& \coH_\et^{k-1}(\ol\KM'_{\ol\QQ}) \ar[r]^(.43)\gamma
& \bigoplus_{x\in\Sigma} \QQ_\ell(-\sfrac{k}{2}), }
\]
in which the second row is exact and the map $\gamma$ is defined by pairing with vanishing cycle classes
$\delta_x \in \coH_\et^{k-1}(\ol\KM'_{\ol\QQ})(\psfrac{k-2}{2}),$
one for each $x \in \Sigma$. Setting $C = \gamma(\image(\alpha)^{G,\chi})$ and regarding $\coH_\et^{k-1}(\ol\KM'_{\ol\FF_p})$ as a subspace of $\coH_\et^{k-1}(\ol\KM'_{\ol\QQ})$, we obtain a diagram
\begin{equation}\label{eq:weights_k_even}
\begin{array}{c}\xymatrix@=6mm{
& \image(\beta)^{G,\chi} \ar@{^{ (}->}[d] \\
0 \ar[r]
& \image(\alpha)^{G,\chi}\cap\coH_\et^{k-1}(\ol\KM'_{\ol\FF_p}) \ar[r]
& \image(\alpha)^{G,\chi} \ar[r]
& C \ar[r] \ar[r]\ar@{^{ (}->}[d] & 0, \\
&&& \left(\bigoplus_{x\in\Sigma} \QQ_\ell(-\sfrac{k}{2})\right)^{G,\chi}
}\end{array}
\end{equation}
in which the row is exact and the vertical arrows are injective. We now show that both of these inclusions are in fact equalities. Taking the identity $\image\beta = (E_\infty^{0, k-1})_{\FF_p}$ into account, the spectral sequence yields an exact sequence of unramified $\Gal(\ol\QQ_p/\QQ_p)$-modules
\[
0 \to (E_\infty^{1,k-2})_{\FF_p}^{G,\chi}
\to \coH_{\et,\cp}^1(\Gmolf,\Sym^k\Kl_2)(1)/W'_0
\to
\image(\beta)^{G,\chi}/W'_0
\to 0
\]
(here $W'$ denotes the weight filtration on étale cohomology over finite fields
given by the eigenvalues of Frobenius, in order to distinguish it from that for $\Gal(\ol\QQ/\QQ)$-representations). The calculation \eqref{eq:decomp_Z_even} of the action of Frobenius on $\coH_{\et,\cp}^1(\Gmolf, \Sym^k\Kl_2)$ implies that the rightmost term $\image(\beta)^{G,\chi}/W'_0$ has weights $k-2$ and $k-1$, with graded pieces of dimensions
\[
\dim \gr^{W'}_{k-2}\image(\beta)^{G,\chi}=\flr{\frac{k}{2p}} \quad \text{and} \quad \dim \gr^{W'}_{k-1}\image(\beta)^{G,\chi}= \frac{k-2}{2}-2\flr{\frac{k}{2p}}-\delta_{4\ZZ}(k).
\]
In particular, $\image(\beta)^{G,\chi}$ has dimension at least
\[
\frac{k-2}{2}-\delta_{4\ZZ}(k)-\flr{\frac{k}{2p}}=\dim\image(\alpha)^{G,\chi}-\flr{\frac{k}{2p}},
\]
and weights $k-2, k-1$, and possibly zero.
On the one hand, the unramified representation~$(\bigoplus_{x\in\Sigma} \QQ_\ell(-\sfrac{k}{2}))^{G,\chi}$
has dimension $|G\backslash\Sigma| = \flr{\sfrac{k}{2p}}$ and Frobenius acts on it as multiplication by~$p^{\sfrac{k}{2}}$. On the other hand, taking \eqref{eq:gr_KMp_Q} into account,
the space $\image(\alpha)^{G,\chi}$
is equipped with a non-degenerate $\Gal(\ol\QQ_p/\QQ_p)$-equivariant Poincar\'e pairing with values in the unramified Tate twist $\QQ_\ell(1-k)$
of (Frobenius) weight $2k-2$. Since $\ol\KM'$ is a projective variety, the equality~$\coH_\et^{k-1}(\ol\KM'_{\ol\FF_p})= W'_{k-1}\coH_\et^{k-1}(\ol\KM'_{\ol\FF_p})$ holds, and hence the orthogonal complement~$C'$ of the subspace~$W'_{k-2}\image(\beta)^{G,\chi}$ of $\image(\alpha)^{G,\chi}$ contains $\image(\alpha)^{G,\chi}\cap\coH_\et^{k-1}(\ol\KM'_{\ol\FF_p})$. From this we derive the inequalities of dimensions
\[
\flr{\frac{k}{2p}} \leq \dim W'_{k-2}\image(\beta)^{G,\chi}
= \dim\image(\alpha)^{G,\chi}/C'
\leq \dim C \leq \flr{\frac{k}{2p}}.
\]
It follows that the right vertical inclusion is an equality, and hence the left one as well since the dimension of $\image(\alpha)^{G,\chi}\cap\coH_\et^{k-1}(\ol\KM'_{\ol\FF_p})$ is then equal to $\dim\image(\alpha)^{G,\chi}-\flr{\sfrac{k}{2p}} \leq \image(\beta)^{G,\chi}$.
In addition, let $\Delta \subset \coH_\et^{k-1}(\ol\KM'_{\ol\QQ})$
be the subspace generated by the vanishing cycle classes
$\QQ_\ell(\psfrac{2-k}{2})\delta_x$ for $x\in \Sigma$.
Since
$\dim\Delta^{G,\chi} \leq \flr{\sfrac{k}{2p}}$
and $\Delta \subset \Delta^\perp = \ker\gamma$,
we get
\[
\Delta^{G,\chi} = W'_{k-2}\image(\beta)^{G,\chi}
\quad\text{and}\quad
(\Delta^\perp)^{G,\chi} = \image(\beta)^{G,\chi}.
\]

By the Picard-Lefschetz formula \cite[Exp.\,XV, Th.\,3.4]{SGA7},
an element $\sigma$ of the inertia group~$I_p$ acts on $v \in \coH_\et^{k-1}(\ol\KM'_{\ol\QQ})$ as
\[
\sigma(v) = v - (-1)^{k/2} t_\ell(\sigma) \sum_{x \in \Sigma} \gen{v,\delta_x}\delta_x,
\]
where $\gen{v,\delta_x} \in
\coH_\et^{2k-2}(\ol\KM'_{\ol\QQ})(\psfrac{k-2}{2}) \cong \QQ_\ell(-\sfrac{k}{2})$
and $t_\ell\colon I_p \to \varprojlim \mu_{\ell^n}(\ol\QQ_\ell)$
is the fundamental tame character.
From this we derive the vanishing $(\sigma-1)^2=0$
for each $\sigma \in I_p$ acting on $V_{k,\ell}$
and the equality $V_{k,\ell}^{I_p}= \image(\beta)^{G,\chi}(-1)$.
Observe that we have proved that $V_{k, \ell}$ satisfies the \emph{weight-monodromy conjecture}, \ie that the associated Weil-Deligne representation is pure of weight $k+1$, in the terminology of Corollary \ref{cor:automorphy_purity} \textit{infra}.
(Conversely, that corollary can be used to show that the vertical arrows in \eqref{eq:weights_k_even} are equalities.) This completes the proof in the case $k\equiv 2 \bmod{4}$.

Finally, we look at the action of Frobenius on $\Delta$.
Recall that each $V_x = \QQ_\ell(\psfrac{2-k}{2})\delta_x$
corresponds to the singularity defined by the equation
\eqref{eq:KM_bp_char_p} for an even positive integer~$b$.
Consider the quadric $C = (Q_{bp}) \subset \PP_{\FF_p}^{k-1}$, whose primitive cohomology $\coH_{\et, \textup{prim}}^{k-2}(C_{\ol\FF_p})$ coincides with~$V_x$ by \cite[Exp.\,XV, Prop.\,2.2.3]{SGA7}.
The quadratic form $Q_{bp}$ has discriminant $d = (-1)^{(k-bp)/2}$,
and therefore $F_p$ acts on the primitive cohomology as multiplication by
\[
\biggl(\frac{(-1)^{k/2}d}{p}\biggr)p^{(k-2)/2}= (-1)^{bp(p-1)/4} p^{(k-2)/2}.
\]
For $p \equiv 1 \bmod{4}$,
the sign is always positive, whereas, for $p \equiv 3 \bmod{4}$,
there are $\flr{\sfrac{k}{4p} + \sfrac{1}{2}}$
values of $b$ such that the sign is negative.
Comparing with the eigenvalues of $F_p$ in \eqref{eq:decomp_Z_even},
one concludes that in the case $k\equiv 0 \bmod{4}$,
the kernel of
$\beta\colon \coH_{\et,\cp}^{k-1}(\KM'_{\ol\FF_p})^{G,\chi}/W_0
\to \coH_\et^{k-1}(\ol\KM'_{\ol\FF_p})^{G,\chi}$
is a factor $\QQ_\ell(\psfrac{2-k}{2})$. This completes the proof.
\end{proof}

Similarly to the case of odd symmetric powers, the above theorem gives the local $L$-factors and the conductor away from $p=2$. Indeed, defining $L_k(p; T)=\det ( 1-F_pT \mid V_{k, \ell}^{I_p})^{-1}$ for a prime number $p$, Theorem~\ref{Thm:etale_k_even} and equation \eqref{eq:decomp_Z_even} imply the equalities
\begin{equation}\label{eq:local-factors-even}
L_k(p; T)^{-1}=\begin{cases}
(1-p^{\sfrac{k}{2}}T)^{ \flr{\sfrac{k}{2p}}} M_k(p; T) & \text{if }p \equiv 1 \bmod{4}, \\
(1+p^{\sfrac{k}{2}}T)^{\flr{\frac{k}{4p}+\frac{1}{2}}} (1-p^{\sfrac{k}{2}}T)^{ \flr{\frac{k}{2p}}-\flr{\frac{k}{4p}+\frac{1}{2}}} M_k(p; T) & \text{if }p \equiv 3 \bmod{4}.
\end{cases}
\end{equation}
The $L$-function of $\Motive_k$ is the Euler product
\begin{displaymath}
L_k(s) = \prod_p L_k(p; p^{-s}),
\end{displaymath}
which again converges absolutely for $\mathrm{Re}(s)>1+\sfrac{(k+1)}{2}$.

As for the conductor, Serre's recipe yields in this case that the exponent of an odd prime $p$ is given by $\flr{\sfrac{k}{2p}}$. The conductor is thus equal to
\begin{equation}
2^{r_k}\prod_{\text{$p$ odd}} p^{\flr{\sfrac{k}{2p}}}=2^{r_k} \cdot 2_\mathrm{u} 4_\mathrm{u} 6_\mathrm{u} \cdots k_\mathrm{u},
\end{equation} where $r_k=\mathrm{Sw}(V_{k, \ell}|_{\Gal(\ol\QQ_2/\QQ_2)})+\mathrm{codim}\,V_{k, \ell}^{I_2}$ and $n_\mathrm{u}$ stands for the odd part of the radical (\ie the product of all odd primes dividing $n$). Broadhurst and Roberts conjecture that $r_k=\flr{\sfrac{k}{6}}$.

\subsubsection{The $p$-adic case}

We keep the setting of Section \ref{sect:rigid_char_p}, and let $\BB_\dR$, $\BB_\crys,$ and~$\BB_\st$ denote Fontaine's $p$-adic de Rham, crystalline, and semistable period rings over $\QQ_p$.
Recall from \eqref{eqn:isomrigidmiddlebis} and \eqref{eqn:isomrigidmiddle} that, for any prime $p$, there is an isomorphism of Frobenius modules
\begin{displaymath}
\coH_{\rrig,\rmid}^1(\Gm/K, \Sym^k\Kl_2)
= \gr^W_{k+1} \coH_{\rrig,\cp}^1(\Gm/K, \Sym^k\Kl_2).
\end{displaymath}
By \cite[Th.\,B]{Robba}, this $K$-vector space has dimension \hbox{$\flr{\psfrac{k-1}{2}}-\delta_{4\ZZ}(k)$} for all $p>k$ if $k$ is odd and for all $p>\sfrac{k}{2}$ if $k$ is even.

\begin{prop}\label{Prop:p-adic_realization}
Fix an integer $k \geq 1$,
a prime number $p$,
and a place of $\ol{\QQ}$ above $p$.
The $p$-adic representation $V_{k,p}$ of $\Gal(\ol\QQ_p/\QQ_p)$
is de Rham.
If $p$ is odd, then
$V_{k,p}$ is semistable over~$\QQ_p(\sqrt{-p})$
and there is an inclusion of Frobenius modules
\begin{equation}\label{eq:rigid-etale}
\coH_{\rig,\rmid}^1(\Gm/K, \Sym^k\Kl_2)
\to
(V_{k,p} \otimes \BB_\st)^{
\Gal(\ol{\QQ}_p/\QQ_p(\sqrt{-p}))} \otimes K.
\end{equation}
Under the extra assumption that $p > k$ if $k$ is odd (\resp $p >\sfrac{k}{2}$ if $k$ is even), the representation~$V_{k,p}$ is crystalline and there is an isomorphism of Frobenius modules
\[
\coH_{\rig,\rmid}^1(\Gm/K, \Sym^k\Kl_2) \cong (V_{k,p} \otimes \BB_{\crys})^{\Gal(\ol{\QQ}_p/\QQ_p)} \otimes K.
\]
\end{prop}

\begin{proof} The $p$-adic representations $\coH_\et^{k-1}(\ol\KM_{\ol\QQ},\QQ_p)(-1)$
and $\coH_\et^{k-1}(\ol\KM'_{\ol\QQ}, \QQ_p)(-1)$
arising from the smooth proper varieties $\ol\KM$ and $\ol\KM'$ are de Rham (see \eg \cite[\S 3.3(i) and \S 3.4]{Beilinson13}), and any subquotient of a de Rham representation is still de Rham. Hence, the first assertion follows from \eqref{eqn:isom1graded} for odd $k$, and from \eqref{eq:gr_KMp_Q} along with \eqref{eq:coh_K'_to_K} for even $k$.

For the remaining statements, we assume that $p$ is odd. We first treat the case of even~$k$. As in the proof of Theorem~\ref{Thm:etale_k_even}, consider the resolution of singularities $\KM'$ of $\KM$
and its compactification $\ol\KM'$
induced from the blow-up of the ambient torus
and the explicit toric compactification $X$ over $\ZZ_p$.
Recall from Section \ref{sect:rigid_char_p} that the localization sequence for rigid cohomology with compact support yields an isomorphism
\[
\Bigl(\gr^W_{k-1}\coH_{\rig,\cp}^{k-1}(\KM_{\FF_p}/ \QQ_p)^{
	\symgp_k\times\mu_2,\,\chi} \Bigr)(-1)[\varpi]
\cong \gr^W_{k+1}\coH_{\rig,\cp}^1(\Gm/K, \Sym^k\Kl_2).
\]
Besides, arguing as in \eqref{eq:coh_K'_to_K}, there are isomorphisms
\[
\coH_{\rig,\cp}^{k-1}(\KM'_{\FF_p}/\QQ_p) \To{\sim}
\coH_{\rig,\cp}^{k-1}(\KM_{\FF_p}/\QQ_p)
\quad \text{and}\quad
\coH_{\et,\cp}^{k-1}(\KM'_{\ol\QQ_p},\QQ_p) \To{\sim}
\coH_{\et,\cp}^{k-1}(\KM_{\ol\QQ_p},\QQ_p)
\]
of Frobenius modules and $\Gal(\ol\QQ_p/\QQ_p)$-modules respectively.

For rigid cohomology, consider the spectral sequence
\[
E_1^{i, j} = \coH_{\rig}^j(\ol\KM'^{(i)}_{\FF_p}/\QQ_p)
\Longrightarrow \coH_{\rig,\cp}^{i+j}(\KM'_{\FF_p}/\QQ_p) \qquad (i, j \geq 0),
\]
as we did in the $\ell$-adic setting (\cf \cite[Prop.\,8.2.17 \& 8.2.18(ii)]{LeStum07}), and let
\[
\alpha \colon \coH_{\rig}^{k-1}(\ol\KM'_{\FF_p}/\QQ_p) \longrightarrow
\coH_{\rig}^{k-1}(\ol\KM'^{(1)}_{\FF_p}/\QQ_p)
\]
denote the differential from $E_1^{0, k-1}$ to $E_1^{1, k-1}$. Since the varieties $\ol\KM'^{(i)}$ are smooth and proper for all $i\geq 1$, the only contribution of weight $k-1$ to the abutment of the spectral sequence comes from the kernel of $\alpha$, hence an isomorphism
\[
\gr^W_{k-1}\coH_{\rig,\cp}^{k-1}(\KM'_{\FF_p}/\QQ_p) \stackrel{\sim}{\longrightarrow} \gr^W_{k-1}\ker\alpha.
\]

Let $L$ be the ramified quadratic extension of $\QQ_p$ contained in $K$, which is $L=\QQ_p(\sqrt{-p})$ thanks to the equality $\sqrt{-p}=\varpi^{(p-1)/2}$. Since the singularities of $\ol\KM'$ consist only of ordinary quadratic points supported on $\KM'_{\FF_p}$, the $p$-adic Picard-Lefschetz formula \cite[Th.\,1.1]{Mieda07} yields a commutative diagram
of $L$-modules
\[
\xymatrix{
\coH_{\rig,\cp}^{k-1}(\KM'_{\FF_p}/\QQ_p)_L \ar[r]
	& \coH_{\rig}^{k-1}(\ol\KM'_{\FF_p}/\QQ_p)_L
		\ar[r]^{\alpha_L}\ar@{^{ (}->}[d]^\beta
	& \coH_{\rig}^{k-1}(\ol\KM'^{(1)}_{\FF_p}/\QQ_p)_L \ar@{=}[d] \\
& \coH_\dR^{k-1}(\ol\KM'_{\QQ_p})_L \ar[r]
	& \coH_\dR^{k-1}(\ol\KM'^{(1)}_{\QQ_p})_L, }
\]
in which the map $\beta$ is injective. Hence, $\beta$ induces an inclusion
\begin{equation}\label{eq:rigid-deRham}
\ker\alpha_L \to
\ker\left\{ \coH_\dR^{k-1}(\ol\KM'_{\QQ_p})_L \to
	\coH_\dR^{k-1}(\ol\KM'^{(1)}_{\QQ_p})_L \right\}.
\end{equation}

Besides, over the ring of integers $\ZZ_p[\sqrt{-p}]$ of $L$, with uniformizer $\sqrt{-p}$, each ordinary quadratic point of $\ol\KM'$ is formally defined by an equation
\[
Q -u\cdot (\sqrt{-p})^2,
\]
where $u$ is some unit and $Q$ equals $Q_{ap}$ from \eqref{eq:KM_char_p} for odd $k$
and $Q_{bp}$ from \eqref{eq:KM_bp_char_p} for even $k$. In both cases, the equation $Q=0$ defines a smooth quadric in $\PP^{k-1}$ over $\ZZ_p[\sqrt{-p}]$,
so that the assumptions of \cite[(2.3)]{Mieda07} are satisfied.
Let $\ol\KM''$ be the blow-up of
$\ol\KM'\otimes_{\ZZ_p}\ZZ_p[\sqrt{-p}]$ along the ordinary quadratic points. Then $\ol\KM''$ is semistable over $\ZZ_p[\sqrt{-p}]$ by \loccit, and hence any subquotient of $\coH_\et^{k-1}(\ol\KM''_{\ol\QQ_p},\QQ_p)$
is a semistable $\Gal(\ol\QQ_p/L)$-representation \eg by \cite[\S 3.3(iii)]{Beilinson13}. In particular, on noting the equalities $\KM'_L = \KM''_L$
and $\ol\KM'^{(i)}_L = \ol\KM''^{(i)}_L$, the representation
$V_{k,p}(1)$ is semistable and the expression \eqref{eq:gr_KMp_Q} yields
\begin{align}
\left(\gr^W_{k-1} \coH_{\et,\cp}^{k-1}(\KM''_{\ol\QQ_p},\QQ_p)
	\otimes \BB_\st \right)^{\Gal(\ol\QQ_p/L)} \otimes L
&= \left(\gr^W_{k-1} \coH_{\et,\cp}^{k-1}(\KM'_{\ol\QQ_p},\QQ_p)
	\otimes \BB_\dR \right)^{\Gal(\ol\QQ_p/L)} \nonumber \\
&= \ker\left\{ \coH_\dR^{k-1}(\ol\KM'_{\QQ_p})_L \to
	\coH_\dR^{k-1}(\ol\KM'^{(1)}_{\QQ_p})_L \right\} \label{eq:etale-deRham}
\end{align}
by the $p$-adic Hodge comparison theorem.
In addition, since both the Frobenius structure on the $L$-vector space
$\coH_\dR^{k-1}(\ol\KM'_{\QQ_p})_L
= \coH_\dR^{k-1}(\ol\KM''_L)$
and the map $\beta$ are constructed
by means of logarithmic de Rham-Witt complexes
(see the proof of \cite[Th.\,2.13]{Mieda07}),
$\beta$ is compatible with the Frobenius action.
The inclusion \eqref{eq:rigid-etale} of Frobenius modules follows by extending scalars to $K$ and taking $\chi$-isotypic components
in \eqref{eq:rigid-deRham} and \eqref{eq:etale-deRham}.

If $k$ is even and $p>\sfrac{k}{2}$, then
$\ol\KM'$ is smooth.
Moreover,
$\beta$ is induced by the isomorphism
\[
\coH_{\rig}^{k-1}(\ol\KM'_{\FF_p}/\QQ_p) \to
\coH_\dR^{k-1}(\ol\KM'_{\QQ_p})
\]
and $\ker\alpha$ is pure of weight $k-1$.
By \cite[\S 3.3(iii)]{Beilinson13}
and a similar argument as in the above case,
we obtain the identity
\[
(V_{k,p}\otimes \BB_\crys)^{\Gal(\ol\QQ_p/\QQ_p)}[\varpi]
= \coH_{\rig,\rmid}^1(\Gmf /K, \Sym^k\Kl_2),
\]
thus finishing the proof for even $k$.

In case $k$ is odd, there is no need to perform the first resolution of singularities, so we simply take $\KM' = \KM$ in what precedes and do the same proof as for even $k$.
\end{proof}

\begin{cor}\label{cor:Newton-Hodge}
Let $k \geq 1$ be an integer and $p$ an odd prime number.
The Newton polygon of the Frobenius module
$\coH_{\rig,\cp}^1(\Gmf/K, \Sym^k\Kl_2)$
lies above the Hodge polygon of~$\coH_{\dR,\cp}^1(\Gm,\Sym^k\Kl_2)$.
In case $p>k$ if $k$ is odd or $2p>k$ if $k$ is even,
the endpoints of both polygons coincide.
\end{cor}

\begin{proof}
Considered as a representation of $\Gal(\ol{\QQ}_p/\QQ_p(\sqrt{-p}))$, the $p$-adic \'etale realization $V_{k,p}$
of~$\Motive_k$ is semistable, and hence the associated filtered $(\varphi, \rN)$-module \hbox{$(V_{k, p} \otimes\!\BB_\st)^{\Gal(\ol{\QQ}_p/\QQ_p(\sqrt{-p}))}$} is (weakly) admissible. This means precisely that its Newton polygon lies above its Hodge polygon, both having the same endpoints.
Notice that these polygons are additive
with respect to the Minkowski sum
(i.e., the sum of the convex sets above the polygons in the plane~$\RR^2$)
for filtered Frobenius modules.
Since the Frobenius module $\coH_{\rig,\rmid}^1(\Gmf/K, \Sym^k\Kl_2)$ injects (after extending scalars to $K$) into the Frobenius module associated with $V_{k,p}$ by Proposition~\ref{Prop:p-adic_realization},
the Newton polygon of the former lies above that of the latter,
which in turns lies above the Hodge polygon of $\coH_{\dR,\rmid}^1(\Gm, \Sym^k\Kl_2)$ by admissibility. Moreover, under the condition $p>k$
(\resp \hbox{$p >\sfrac{k}{2}$})
if $k$ is odd (\resp even),
the two Frobenius modules are equal by Proposition \ref{Prop:p-adic_realization}, and hence both polygons have the same endpoints. We conclude the proof by putting the trivial factor back, which is one-dimensional with Frobenius and Hodge slopes $0$ if $k\not\equiv 0\bmod{4}$ and two-dimensional with Frobenius and Hodge slopes $0$ and $\sfrac{k}{2}$
otherwise.
\end{proof}

\begin{remark} Writing $Z_k(p; T)=\sum c_n T^n$,
the above corollary implies in particular that the inequality
\begin{equation}\label{eqn:NewtonaboveHodge}
v_p(c_n) \geq n(n-1)
\end{equation} holds for all $p \geq 3$. This sharpens a theorem of Haessig \cite[Th.\,1.1]{Haessig17}, who obtained the lower bound $\left( 1 - \sfrac{1}{(p-1)} \right) n(n-1)$ for all $p \geq 5$ using $p$-adic analysis \`a la Dwork. While this paper was being refereed, he managed to prove \eqref{eqn:NewtonaboveHodge} for all $p \geq 2$ by strengthening his previous arguments \cite{Haessig20}.
As we explain in Remark \ref{rmk:l=p_comparison} \emph{infra}, it is also possible to obtain this lower bound in all cases except for $p=2$ and even $k$ from the potential automorphy of the motive~$\Motive_k$. Note that, when $k$ is even, the Hodge polygon of $\coH_{\dR,\cp}^1(\Gm, \Sym^k\Kl_2)$
lies strictly above the polygon with vertices $n(n-1)$, as Figure \ref{figure:Newton_polygon} shows.
\begin{figure}[htb]
\begin{center}
\begin{tikzpicture}[yscale=.5]
\draw (0,0) -- (1,0) -- (2,2) -- (3,6) -- (4,12);
\node[below] at (0,0) {$0$};
\node[below] at (4,0) {$4$};
\node[below] at (2,-1) {$k=7$};
\node[right] at (4,6) {$6$};
\node[right] at (4,12) {$12$};
\draw[help lines] (0,0) grid (4,13);

\draw (-4,0) -- (-3,0) -- (-2,2) -- (-1,7);
\node[below] at (-4,0) {$0$};
\node[below] at (-1,0) {$3$};
\node[below] at (-2.5,-1) {$k=6$};
\node[right] at (-1,7) {$7$};
\draw[help lines] (-4,0) grid (-1,13);

\draw (5,0) -- (6,0) -- (7,2) -- (8,6) -- (9,13);
\node[below] at (5,0) {$0$};
\node[below] at (9,0) {$4$};
\node[right] at (9,6) {$6$};
\node[right] at (9,13) {$13$};
\node[below] at (7,-1) {$k=8$};
\draw[help lines] (5,0) grid (9,13);
\end{tikzpicture}
\caption{The Hodge polygons of
$\coH_{\dR,\cp}^1(\Gm, \Sym^k\Kl_2)$}
\label{figure:Newton_polygon}
\end{center}
\end{figure}
\end{remark}

\subsection{The gamma factor}\label{Sect:periods}

We first recall Serre's recipe \cite[\S 3]{Serre70}
describing the conjectural shape of the gamma factor at infinity in the complete $L$-function of a pure motive over $\QQ$.
Let $V$ be a finite-dimensional vector space over $\CC$
together with an $\RR$-Hodge decomposition
of weight~$w$,
i.e.,
the data of a grading
$V = \bigoplus_{p\in\ZZ} V^p$ and a $\CC$-linear involution $\sigma$ of $V$
such that $\sigma(V^p) = V^{w-p}$.
Given an $\RR$-Hodge decomposition, we set $h(p) = \dim_{\CC} V^p$
and
\begin{displaymath}
h(\sfrac{w}{2})^{\pm} =\dim_{\CC} \left\{ v \in V^{w/2} \,|\, \sigma(v) = \pm(-1)^{w/2}v \right\}
\end{displaymath}
if $w$ is even and $h(\sfrac{w}{2})^{\pm}=0$ otherwise. Setting
\[
\Gamma_\RR(s) = \pi^{-s/2}\Gamma(s/2),
\quad
\Gamma_\CC(s) = 2(2\pi)^{-s}\Gamma(s)
	= \Gamma_\RR(s)\Gamma_\RR(s+1),
\]
the \textit{gamma factor} $\Gamma_V(s)$
of $V$ is defined as
\[
\Gamma_V(s) =
\Gamma_\RR(s-w/2)^{h(w/2)^+}
\Gamma_\RR(s-w/2+1)^{h(w/2)^-}
\prod_{p<w/2} \Gamma_\CC(s-p)^{h(p)}.
\]

\begin{cor}\label{Cor:Gamma-factor}
For each integer $k \geq 1$, the gamma factor of the motive
$\Motive_k$ is equal to
\[
L_k(\infty,s) =
\pi^{-m\sfrac{s}{2}} \prod_{j=1}^m \Gamma\Bigl(\dfrac{s-j}{2} \Bigr),
\quad
m = \flr{\frac{k-1}{2}} - \delta_{4\ZZ}(k).
\]
\end{cor}

\begin{proof}

In our geometric setting, the grading is given by
\begin{displaymath}
V^p=\gr_F^p \coH_{\dR,\rmid}^1(\Gm, \Sym^k\Kl_2)
\end{displaymath}
and the $\RR$-structure comes from the maps $\sigma$ induced by complex conjugation \hbox{$\KM(\CC) \to \KM(\CC)$} on the singular cohomology
$\coH^{k-1}(\KM(\CC))$
and the singular cohomology with compact support $\coH^{k-1}_\cp(\KM(\CC))$, \cf \cite[\S 3.3(b)]{Serre70}.
These form an $\RR$-Hodge decomposition
of weight~$w = k+1$.

Observe that the middle degree factor $V^{w/2}$ is non-trivial
if and only if $k = 4r + 3$ for some integer $r \geq 0$, in which case the weight is $w = 4r+4$ and $V^{w/2}$ has dimension one. Assuming this, let $\varepsilon \in \{\pm 1\}$ denote the sign of the action of $\sigma$ on $V^{w/2}$. Since $V$ has dimension $2r+1$
and $\sigma$ interchanges~$V^p$ and $V^{w-p}$,
the equality $\det\sigma = (-1)^r\epsilon$
holds in $\det \coH_{\dR,\rmid}^1(\Gm, \Sym^k\Kl_2)$.
Therefore, it suffices to compute $\det \sigma$.
Thanks to the orthogonal pairing~\eqref{eq:pairing-motive}, the above determinant is, up to a twist, the de~Rham realization of the rank-one Artin motive associated with a quadratic field extension of~$\QQ$
and one only needs to decide whether this field is real or imaginary.
To do so, we look at the $\ell$-adic representation $r_{k, \ell} \colon \Gal(\ol\QQ/\QQ) \to \GL(V_{k, \ell})$. For each odd prime~$p$, the determinant of Frobenius was computed in \cite[Th.\,0.1]{FW10}:
\begin{displaymath}
\det(F_p | \coH_{\et,\rmid}^1(\Gmolf, \Sym^k\Kl_2))=p^{\sfrac{(k+1)\dim \coH_{\rmid}^1}{2}} \Bigl(\frac{2}{p}\Bigr)^{\flr{\sfrac{k}{2p}+\sfrac{1}{2}}} \!\!\prod_{\substack{0 \leq j \leq \sfrac{(k-1)}{2} \\ p \nmid 2j+1}}\! \Bigl(\frac{(-1)^j(2j+1)}{p} \Bigr).
\end{displaymath}
From this we immediately derive that, for all primes $p > k$, the equality
\begin{displaymath}
\det(r_{k, \ell}(\mathrm{Frob}_p))=\Bigl(\frac{(-3) \cdot 5 \cdots (-1)^{\sfrac{(k-1)}{2}} k }{p}\Bigr)p^{\sfrac{(k^2-1)}{4}}=\Bigl(\frac{p}{k!!} \Bigr)p^{\sfrac{(k^2-1)}{4}}
\end{displaymath}
holds, with $k!!=3 \cdot 5 \cdots k$. Chebotarev's density theorem then yields
$\det r_{k, \ell}=(\sfrac{\cdot}{k!!}) \chi_{\mathrm{cyc}}^{\psfrac{1-k^2}{4}}$.
It follows that the quadratic number field to which this character gives rises through class field theory is equal to $\QQ(\sqrt{\pm k!!}),$ with sign adjusted by the condition that the radicand is congruent to $1$ modulo $4$ (otherwise, $2$ would be a ramified prime).
Noting that $k=4r+3$,
this sign is given by $(-1)^{r+1}$ and the power of the cyclotomic character appearing in $\det r_{k, \ell}$ is even. Putting everything together, one derives $\epsilon = -1 = -(-1)^{w/2}$, and hence $h(w/2)^+ = 0$ and $h(w/2)^- = 1$, which was the missing information to compute the gamma factor.
\end{proof}

\subsection{Potential automorphy, meromorphic continuation, and functional equation}

In this final section, we pull everything together to prove Theorems \ref{thm:intro-k-odd} and \ref{thm:intro-k-even} from the introduction. We first compute the $\epsilon$-factors of the Galois representations $V_{k, \ell}$
and recall the particular case of the theorem of Patrikis and Taylor
that will imply potential automorphy.

\subsubsection{Weil-Deligne representations and $\epsilon$-factors}
\label{subsec:epsilon}
For each integer $k \geq 1$, consider the system~$\{V_{k,\ell}\}_\ell$
of $\ell$-adic realizations of the motive $\Motive_k$. We investigate its global $\epsilon$\nobreakdash-factor $\epsilon_k(s)$ by means of the information obtained in Theorems \ref{Thm:etale_k_odd} and \ref{Thm:etale_k_even}.
We refer the reader to \cite{Taylor} for an accessible introduction to Weil-Deligne representations.

As inputs for defining the local $\epsilon$-factor of $\{V_{k, \ell}\}$
at each place $p$ of $\QQ$,
we fix the additive character~$\psi$ and the Haar measure $\de x$
on $\QQ_p$ as follows.
If $p<\infty$, then
$\psi$ is the composition
\begin{displaymath}
\QQ_p \to \QQ_p/\ZZ_p = \ZZ[\sfrac{1}{p}]/\ZZ \to \CC^\times,
\end{displaymath}
where the first map is the quotient and the second map sends $\alpha$ to $\exp(2\pi \sfi\alpha)$. The Haar measure
$\de x$ is normalized so that $\int_{\ZZ_p} \de x = 1$ holds; note that it is self-dual with respect to~$\psi$.
For $p = \infty$,
we set $\psi(\alpha) = \exp(-2\pi \sfi\alpha)$ for $\alpha\in\RR$,
and we take as $\de x$ the usual Lebesgue measure.
Letting $\mathbf{A}_\QQ$ denote the adele ring of $\QQ$, these local characters and Haar measures
are compatible in the sense that the product of the $\psi$'s induces a character of $\mathbf{A}_\QQ/\QQ$
and the compact quotient
$\mathbf{A}_\QQ/\QQ$ has volume~$1$
with respect to the induced measure \cite[\S 3.10]{Del73}.

For each $p < \infty$, let $W(\ol\QQ_p/\QQ_p)$ be the Weil group of $\QQ_p$, \ie the subgroup of $\Gal(\ol\QQ_p/\QQ_p)$ consisting of those elements whose image in $\Gal(\ol\FF_p / \FF_p)$ is an integral power of Frobenius together with the topology making $I_p$ with its usual topology into an open subgroup, and let~$F_p \in W(\ol\QQ_p/\QQ_p)$ be a lifting of the geometric Frobenius. Local class field theory provides an isomorphism between $\QQ_p^\times$ and the maximal abelian quotient $W(\ol\QQ_p/\QQ_p)^\ab$; following the convention of \cite[\S 2.3]{Del73},
we normalize it so that $p$ is mapped to $F_p$. For $s \in \CC$,
let
\[
\omega_s\colon W(\ol\QQ_p/\QQ_p) \to \CC^\times
\]
be the homomorphism defined by composition of the quotient map to
$W(\ol\QQ_p/\QQ_p)^\ab \cong \QQ_p^\times$ with the map from $\QQ_p^\times$ to $\CC^\times$ sending $\alpha$ to $\|\alpha\|^s$, with \hbox{$\|p\| = 1/p$.}

With a continuous representation~$\rho$ of $W(\ol\QQ_p/\QQ_p)$
on a discrete topological vector space~$V$
over a field of characteristic zero, to which we shall refer as a \textit{Weil representation},
is associated
a local $\epsilon$-factor $\epsilon_0(\rho,s)=\epsilon_0(\rho\cdot\omega_s,0)$,
depending on $\psi$ and $\de x$, in \cite[Th.\,4.1]{Del73}. By (5.5.2) in \loccit, the following equality holds:
\begin{equation}\label{eq:eps_s}
\epsilon_0(\rho,s) = \omega_s(p^{a(\rho)}) \cdot \epsilon_0(\rho,0)
= p^{-a(\rho)s} \cdot \epsilon_0(\rho,0),
\end{equation}
where $a(\rho)$ denotes the conductor of $\rho$
and we regard $\omega_s$ as a map
$\QQ_p^\times \cong W(\ol\QQ_p/\QQ_p)^\ab \to \CC^\times$.

A \textit{Weil-Deligne representation} $(\rho, \rN)$ on $V$ consists of a Weil representation $\rho$ on $V$ as above
and a nilpotent endomorphism $\rN$, called the \textit{logarithm of the unipotent part
of the local monodromy}, such that the equality $\rho(w)\rN\rho(w)^{-1}=p^{-v(w)}\rN$ holds for all $w \in W(\ol \QQ_p / \QQ_p)$, where~$v(w)$ denotes the power of $F_p$ to which $w$ is mapped in $\Gal(\ol\FF_p/\FF_p)$. There is a canonical way to attach a Weil-Deligne representation to an $\ell$-adic representation $r$ of~$W(\ol\QQ_p / \QQ_p)$: by Grothendieck's quasi-unipotency theorem, there exists a unique nilpotent endomorphism~$\rN$ satisfying $r(\sigma)=\exp(t_\ell(\sigma)\rN)$
for all $\sigma$ in a finite index subgroup of $I_p$, and one~sets
\begin{equation}\label{eq:Weilreps}
\rho(\sigma F_p^n)=r(\sigma F_p^n)\exp(-t_\ell(\sigma)\rN)
\end{equation} for all $\sigma \in I_p$ and all $n \in \ZZ$, \cf \cite[\S 8.4]{Del73}. Setting
\[
\epsilon_1((\rho,\rN),s)
= \det\bigl(-p^{-s}F_p \mid V^{\rho(I_p)}/\ker(\rN)^{\rho(I_p)}\bigr),
\]
the local $\epsilon$-factor of the Weil-Deligne representation $(\rho, \rN)$ defined in \cite[above Rem.\,5.2.1]{Del79} is equal to the product
\begin{equation}\label{eq:eps_WD}
\epsilon((\rho,\rN),s) = \epsilon_0(\rho,s) \cdot \epsilon_1((\rho,\rN),s).
\end{equation}
Notice the equality $V^{r(I_p)} = \ker(\rN)^{\rho(I_p)}$ from \cite[\S 8.12]{Del73}.

Let $\ell$ be a prime number distinct from $p$.
For $s \in \ZZ$,
we also regard $\omega_s$ as a homomorphism to $\QQ_\ell^\times$. We consider the Weil-Deligne representation $(\rho, \rN)$ on $V_{k,\ell}$ corresponding to the $\ell$\nobreakdash-adic representation $V_{k,\ell}$
of $\Gal(\ol\QQ_p/\QQ_p)$ and denote by $\epsilon_k(p,s) = \epsilon((\rho, \rN),s)$ its $\epsilon$-factor.

Suppose $k$ is odd.
For $2<p<\infty$,
the representation $V_{k,\ell}$ of the inertia group $I_p$
is tame and factors through characters of subgroups of index at most two
by Theorem~\ref{Thm:etale_k_odd}.
The associated Weil-Deligne representation $(\rho, \rN)$
has thus $\rN=0$ and $\rho$ equals the restriction of~$V_{k,\ell}$ to $W(\ol\QQ_p/\QQ_p)$, so that the equality $\epsilon_k(p,s) = \epsilon_0(\rho,s)$ holds in this case. By definition (\cf \cite[(4.5.4)]{Del73}), the conductor of $\rho$ is given by
\begin{displaymath}
a(\rho) = \dim V_{k,\ell} - \dim V_{k,\ell}^{\rho(I_p)} = |\Theta_p^-|
\end{displaymath}
and from the ``formulaire'' in \loccit~we find
\begin{equation}\label{eq:eps^2}
\begin{aligned}
1 &= \epsilon_0(\rho,0)\cdot\epsilon_0(\rho^\vee\cdot\omega_1,0)\cdot \det(\rho)(-1)
	&&&& \text{by \cite[(5.4), (5.7.1)]{Del73}} \\
&= \epsilon_0(\rho,0)\cdot\epsilon_0(\rho\cdot\omega_{k+2},0)\cdot \det(\rho)(-1)
	&&&& \text{since $V_{k,\ell}^\vee = V_{k,\ell}(k+1)$} \\
&= \epsilon_0(\rho,0)^2\cdot (p^{|\Theta_p^-|})^{-(k+2)} \cdot \det(\rho)(-1)
	&&&& \text{by \eqref{eq:eps_s}}.
\end{aligned}
\end{equation} Recall from the proof of Corollary~\ref{Cor:Gamma-factor} that $\det(\rho)$ is the non-trivial character
associated with the quadratic extension $\QQ_p(\sqrt{\pm k!!})$
with positive sign if $k \equiv 1, 7 \bmod{8}$ and negative sign otherwise. Therefore, $\det(\rho)(-1)$ is given by the Hilbert symbol $(-1, \pm k!!)$ and there exists a fourth root of unity $w_p \in \mu_4(\CC)$ with $w_p^2=(-1, \pm k!!)$ such that
\[
\epsilon_k(p,s) = w_p \cdot (p^{|\Theta_p^-|})^{\psfrac{k+2}{2}-s}.
\]
Moreover, if $V_{k,\ell}$ is unramified, then the equality $\epsilon_k(p,s) = 1$ holds (recall that this includes the case $p=2$).
According to~\cite[\S 5.3]{Del79}, at $p = \infty$ the associated $\epsilon$-factor $\epsilon_k(\infty,s)$ is given, in the notation of Section \ref{Sect:periods} \textit{supra}, by $\sfi=\sqrt{-1}$ to the power
\begin{displaymath}
\sum_{p<q}(q-p+1)h(p)+h(\sfrac{(k+1)}{2})^{-}=\frac{k^2-1}{8}.
\end{displaymath}
Now the product formula for Hilbert symbols implies that $\varepsilon_k(\infty, s)\cdot \prod_{p<\infty} w_p$ belongs to $\{\pm 1\}$, and putting everything together we get
\[
\epsilon_k(s) = \prod_{p\leq \infty} \epsilon_k(p,s)
	= \pm \cond_k^{\psfrac{k+2}{2}-s},
\]
where $\cond_k$ is the integer defined in \eqref{eq:conductor_k_odd}.

\begin{remark}\label{rmk:eps_ss_k_odd}
It is obvious that in this case the $\epsilon$-factors remain unchanged
if one replaces the input $\{V_{k,\ell}\}_\ell$
with its semi-simplification $\{V_{k,\ell}^\ss\}_\ell$.
\end{remark}

We now suppose that $k$ is even and keep notation from Theorem~\ref{Thm:etale_k_even}.
If $2<p<\infty$,
there exist a basis $\{e_i\}$ of $U$
and elements $\{e_i'\}$ inducing a basis of $V_{k,\ell}/U^\perp$
satisfying $\ker(\rN) = U^\perp$
and $\rN(e_i') = -(-1)^{k/2}e_i$.
In this case, the inertia group acts trivially on the Weil-Deligne representation $\rho$ given by \eqref{eq:Weilreps}, hence the equalities $a(\rho) = 0$
and $\epsilon_0(\rho,s) = 1$.
From the identity $V_{k,\ell}^\vee = V_{k,\ell}(k+1)$
as representations of $\Gal(\ol\QQ_p/\QQ_p)$,
we derive $\det (\rho(F_p), V_{k,\ell})^2 = p^{m(k+1)}$, and since the duality pairing is symplectic the determinant has positive sign: $\det (\rho(F_p), V_{k,\ell})= p^{m\psfrac{k+1}{2}}$
(note that~$m = \dim V_{k,\ell}$ is even). Using the definition \eqref{eq:eps_WD} and Theorem~\ref{Thm:etale_k_even},
we obtain
\begin{equation}\label{eq:eps_no_2}
\epsilon_k(p,s) = (-1)^{v_p}\cdot p^{\flr{\sfrac{k}{2p}}\left(\psfrac{k+2}{2}-s\right)},
\quad v_p = \begin{cases}
\flr{\dfrac{k}{2p}} & \text{if }p \equiv 1 \bmod{4}, \\[10pt]
\flr{\dfrac{k}{4p}} & \text{if }p \equiv 3 \bmod{4}.
\end{cases}
\end{equation}
Besides, the computation of Hodge numbers yields
\[
\epsilon_k(\infty,s) = \left.\begin{cases}
1 & \text{if }k \equiv 2 \bmod{4}, \\
(-1)^{\psfrac{k-4}{4}} & \text{if }k \equiv 0 \bmod{4} \end{cases}\right\}
= (-1)^{\delta_{8\ZZ}(k)},
\]
from which we get the value of the global epsilon factor away from $p=2$:
\[
\prod_{2<p\leq \infty} \epsilon_k(p,s) = (-1)^{v'}{\cond'_k}^{\psfrac{k+2}{2}-s},
\quad
v' = \sum_{p\equiv 1 \,(4)} \flr{\frac{k}{2p}}
	+ \sum_{p\equiv 3 \,(4)} \flr{\frac{k}{4p}}
	+ \delta_{8\ZZ}(k).
\]

\begin{remark}
The factor $\epsilon_k(s)$ is of the form $AB^{-s}$
since all its local factors are.
Suppose that there is a functional equation
\[
\wh{L}_k(s) = \epsilon_k(s) \cdot \wh{L}_k(k+2-s),
\quad
\wh{L}_k(s) = L_k(\infty,s) \cdot \prod_{p<\infty} L_k(p,p^{-s}).
\]
By applying it twice, we get $A^2 = B^{k+2}$.
Suppose $k$ is even and $p=2$, and
let $a = a(\rho)$ be the conductor
of the associated Weil-Deligne representation.
The same computation as in \eqref{eq:eps^2} gives
$|\epsilon_0(\rho,0)| = 2^{a\psfrac{k+2}{2}}$.
On the other hand,
suppose the quotient $V^{\rho(I_2)}/\ker(\rN)^{\rho(I_2)}$
has dimension $r$ and that $\det(F_2)$ acts as $\delta$.
Plugging this into \eqref{eq:eps_WD}, we obtain
\[
\epsilon_k(2,s) = w''|\delta|\,2^{a\psfrac{k+2}{2}}\, 2^{-(a+r)s}
\]
for some $|w''| = 1$.
Under the assumption that the functional equation holds
with \eqref{eq:eps_no_2} we get $|\delta| = 2^{r\psfrac{k+2}{2}}$, and hence the equality
\[
\epsilon_k(s) = w\cdot (2^{a+r}\cond_k')^{\psfrac{k+2}{2}-s},
\quad
w = (-1)^{v'}w''.
\]
Based on the numerical data, the equality $a+r = \flr{\sfrac{k}{6}}$, which is also
the exponent of $3$ in~$\cond_k'$, is conjectured in \cite{Broadhurst17}.
It is further conjectured that $w'' = (-1)^{v''}$
with $v'' = \flr{\sfrac{k}{8}}$.
One possible structure fitting these data would be that
$V_{k,\ell}$ is tamely ramified at $2$, that the inertia group acts trivially on the associated Weil-Deligne representation, and that the reciprocal characteristic polynomial of Frobenius is
\[
\det\left( 1-F_2T \mid V_{k,\ell}^{I_2} \right) =
\bigl( 1- 2^{k/2}T \bigr)^{\flr{\sfrac{k}{8}}} \bigl( 1+2^{k/2}T \bigr)^{b_k}
M_k(2; T)
\]
of degree $\psfrac{k-2}{2}-\delta_{4\ZZ}(k)-\flr{\sfrac{k}{6}},$
where $b_k$ and $M_k(2; T)$ are defined in \eqref{eq:Z_k_2_T}.
\end{remark}

\subsubsection{The theorem of Patrikis-Taylor}

Let $m \geq 1$ be an integer and $S$ a finite set of prime numbers. We consider a \textit{weakly compatible} system of continuous semi-simple representations
\begin{displaymath}
r_\ell \colon \Gal(\ol \QQ / \QQ) \longrightarrow \GL_m(\ol \QQ_\ell)
\end{displaymath}
with $\ell$ running over all prime numbers.
The notion of being ``weakly compatible'' is borrowed from~\hbox{\cite[5.1]{BGGT}} and means that the following three conditions hold:
\begin{itemize}
\item if $p \notin S$, then for all $\ell \neq p$ the representation $r_\ell$ is unramified at $p$ and the characteristic polynomial of $r_\ell(\mathrm{Frob}_p)$ lies in $\QQ[T]$ and is independent of $\ell$;

\item each representation $r_\ell$ is de Rham and in fact crystalline if $\ell \notin S$;

\item the Hodge-Tate weights of $r_\ell$ are independent of $\ell$.
\end{itemize}

\begin{thm}[Patrikis-Taylor, {\cite[Th.\,A]{PT15}}]\label{thm:Patrikis-Taylor}
Let $\mathcal{R}=\{r_\ell\}$ be a weakly compatible system that satisfies the following three properties:
\begin{itemize}
\item (Purity) There exists an integer $w$ such that, for each prime $p \notin S$, the roots of the common characteristic polynomial of $r_\ell(\mathrm{Frob}_p)$ are Weil numbers of weight $w$.

\item (Regularity) The representation $r_\ell$ has $m$ distinct Hodge-Tate weights.

\item (Odd essential self-duality) Either each $r_\ell$ factors through a map to $\mathrm{GO}_m(\ol \QQ_\ell)$ with even similitude character or each $r_\ell$ factors through a map to $\mathrm{GSp}_m(\ol \QQ_\ell)$ with odd similitude character.
Moreover, these characters form a weakly compatible system.
\end{itemize} Then there exists a finite, Galois, totally real number field over which all of the $r_\ell$ become automorphic.
In particular, the partial $L$-function
\begin{displaymath}
L^S(\mathcal{R}, s)=\prod_{p \notin S} \det(1-r_\ell(\mathrm{Frob}_p)p^{-s})^{-1}
\end{displaymath}
admits a meromorphic continuation to the complex plane.
\end{thm}

Let $p$ and $\ell$ be distinct prime numbers
and let $(\rho, \rN)$ be
the Weil-Deligne representation on~$V \simeq \ol\QQ_\ell^m$
associated with $r_\ell$.
There is a unique increasing \textit{monodromy filtration}
\hbox{$V^{\leq\bbullet} \subset V$}
attached to the nilpotent endomorphism $\rN$
such that the inclusion $\rN V^{\leq a} \subset V^{\leq a-2}$ holds
for each integer $a$ and that the map $V^{\leq a}/V^{\leq a-1} \to V^{\leq -a}/V^{\leq -a-1}$
induced by $\rN^a$ is an isomorphism for all $a\geq 0$.
Recall that $(\rho, \rN)$ is called \textit{pure} of weight $w$
if the eigenvalues of Frobenius~$F_p$ acting on
$V^{\leq a}/V^{\leq a-1}$ are $p$-Weil numbers of weight $w+a$ for all $a$.
Building on vast work on constructions of Galois representations
attached to automorphic representations,
which is partly summarized in~\hbox{\cite[Th.\,2.1.1]{BGGT},}
Theorem \ref{thm:Patrikis-Taylor} implies the following:

\begin{cor}[{\cite[Cor.\,2.2\,(ii)]{PT15}}]\label{cor:automorphy_purity}
Let $\mathcal{R}=\{r_\ell\}$ be a weakly compatible system
that is pure of weight $w$, regular, and odd essential self-dual. Then, for any distinct primes $p$ and $\ell$,
the Weil-Deligne representation $\mathrm{WD}_p(\mathcal{R})$ of $\Gal(\ol\QQ_p/\QQ_p)$
associated with $r_\ell$ is pure of weight $w$. Moreover, with the notation of \cite[5.1]{BGGT}, the completed $L$-function
\begin{displaymath}
\Lambda(\mathcal{R}, s)=L_\infty(\mathcal{R}, s) \cdot \prod_{p \in S} L(\mathrm{WD}_p(\mathcal{R}), s) \cdot L^S(\mathcal{R}, s)
\end{displaymath}
satisfies the functional equation $\Lambda(\mathcal{R}, s)=\varepsilon(\mathcal{R}, s) \Lambda(\mathcal{R}^\vee, 1-s)$.
\end{cor}

\subsubsection{Proof of Theorems \ref{thm:intro-k-odd} and \ref{thm:intro-k-even}} For each integer $k \geq 1$, the $\ell$-adic representations
\begin{displaymath}
r_{k, \ell} \colon \Gal(\ol\QQ / \QQ) \longrightarrow
\GL(V_{k, \ell}\otimes\ol\QQ_\ell) \simeq \GL_m(\ol\QQ_\ell)
\end{displaymath}
are pure of weight $k+1$.
After choosing an embedding $\QQ_p\hookrightarrow\CC$,
we get a filtered isomorphism
\begin{align*}
(V_{k,p}\otimes\BB_\dR)^{\Gal(\ol\QQ_p/\QQ_p)}\otimes\CC
&= (\gr^W_{k-1}\coH_{\et,\cp}^{k-1}(\KM_{\ol\QQ_p},\QQ_p)^{\symgp_k\times\mu_2,\chi}
	\otimes\BB_\dR)^{\Gal(\ol\QQ_p/\QQ_p)}(-1)\otimes\CC \\
&= \gr^W_{k-1}\coH_{\dR,\cp}^{k-1}(\KM_{\QQ_p})^{\symgp_k\times\mu_2,\chi}(-1)\otimes\CC \\
&\cong \coH_{\dR,\rmid}^1(\Gm,\Sym^k\Kl_2)
\end{align*}
by the $p$-adic comparison theorem. By definition, the Hodge-Tate weights of $V_{k, p}$ are those integers $a$
such that the graded piece
$\gr_F^a\coH_{\dR,\rmid}^1(\Gm,\Sym^k\Kl_2)$ is non-zero, counted with multiplicity its dimension. As the Hodge numbers are either $0$ or $1$ by Theorem~\ref{th:main}, the system $\{r_{k,\ell}\}$ is regular. Besides, the existence of the $(-1)^{k+1}$\nobreakdash-symmetric perfect pairing~\eqref{eq:pairing-motive} implies that the $r_{k, \ell}$ factor through $\mathrm{GO}_m(\ol\QQ_\ell)$ (\resp $\mathrm{GSp}_m(\ol\QQ_\ell)$) if $k$ is odd (\resp even) with similitude character~$\chi_{\mathrm{cyc}}^{-k-1}$.
Choose a basis of $\QQ_\ell(-k-1)$
and regard the perfect pairing
\hbox{$V_{k,\ell}\times V_{k,\ell} \to \QQ_\ell(-k-1)$}
as a compatible non-degenerate bilinear form
on the module $V_{k,\ell}$ over the group ring
$\QQ_\ell[\Gal(\ol\QQ/\QQ)]$
with the involution $g \mapsto \chi_{\mathrm{cyc}}^{-k-1}(g)g^{-1}$.
By \cite[Th.\,4.2.1]{Serre18}, the semi-simplification
$r_{k,\ell}^\ss$ also factors through
$\mathrm{GO}_m(\ol\QQ_\ell)$ (\resp $\mathrm{GSp}_m(\ol\QQ_\ell)$)
with similitude character $\chi_{\mathrm{cyc}}^{-k-1}$.
Moreover, $r_{k,\ell}^\ss$ is de Rham at all primes $\ell$
and crystalline if~$\ell > k$
by Proposition~\ref{Prop:p-adic_realization}. Therefore, the $r_{k,\ell}^\ss$ form a weakly compatible system satisfying the assumptions of
the theorem of Patrikis and Taylor and the partial $L$\nobreakdash-function of $\{r_{k,\ell}^\ss\}$
has meromorphic continuation and satisfies the expected automorphic functional equation.

We now show that the $L$-function and the $\epsilon$-factor
of $\{r_{k,\ell}^\ss\}$ coincide with those of $\{r_{k,\ell}\}$.
For odd $k$, this was the content of Remarks \ref{rmk:L_ss_k_odd} and \ref{rmk:eps_ss_k_odd}. For even $k$, we rely on the following lemma, which is certainly well-known to experts. (We thank one of the referees for suggesting the statement and sketching a proof, which we include for lack of an appropriate reference.)

\begin{lemma}
Let $p$ and $\ell$ be distinct prime numbers,
and $r\colon \Gal(\ol\QQ_p/\QQ_p) \to \GL(V)$
an $\ell$-adic representation.
Suppose that there exists a sequence
\[
0= V_0 \subset V_1 \subset \cdots \subset V_c = V
\]
of $\Gal(\ol\QQ_p/\QQ_p)$-stable subspaces
such that the Weil-Deligne representation associated with the induced representation $\bar{r}$
on $\ol{V} = \bigoplus_{i=1}^c V_i/V_{i-1}$ is pure.
Then the Weil-Deligne representation associated with $r$ is pure as well,
and $r$ and $\bar{r}$
have the same $L$ and $\epsilon$-factors.
\end{lemma}

\begin{proof} Let $(\rho, \rN)$ and $(\bar{\rho}, \ol{\rN})$ be the Weil-Deligne representations associated with $r$ and $\bar r$.
By its defining properties (\cf \cite[Th.4.1(1)]{Del73}), the factor $\varepsilon_0(\rho, s)$ in \eqref{eq:eps_WD} only depends on the semi-simplification of~$r$. Set $\varphi=\rho(F_p)$. It follows from the relation $\rN\varphi=p\varphi \rN$ that, if $\alpha$ is an eigenvalue of $\varphi$ acting on $V^{\leq a}/V^{\leq a-1}$ with multiplicity $\mu$, then $\alpha p^{-b}$ is an eigenvalue of $\varphi$ acting on $V^{\leq a-2b}/V^{\leq a-2b-1}$ with multiplicity at least $\mu$ for all $0\leq b\leq a$. Besides, by the uniqueness of the monodromy of $\ell$-adic representations in Grothendieck's quasi-unipotency theorem, $\rN$ restricts to the logarithm on each $V_i$, and hence induces the monodromy $\ol{\rN}$ on~$\ol{V}$. Since $(\bar\rho,\ol{\rN})$ is pure and $\det(1-\varphi T) = \det(1-\bar\varphi T)$, we conclude that the monodromy filtrations on $(\rho, \rN)$ and $(\bar\rho,\ol{\rN})$ have the same dimension on each graded piece and that $(\rho, \rN)$ is pure as well. In addition, the two Weil representations
$\rho|_{\ker(\rN)}$ and $\bar\rho|_{\ker(\ol{\rN})}$
have the same semi-simplification. Finally, observe that taking $I_p$-invariants is exact on Weil representations
since~$I_p$ acts through finite quotients, hence the identities
\begin{align*}
\det\big(1-\varphi T \mid \ker(\rN)^{\rho(I_p)}\big)
&= \det\big(1-\bar\varphi T \mid \ker(\ol{\rN})^{\bar\rho(I_p)}\big), \\
\det\big(\varphi T \mid V^{\rho(I_p)}/\ker(\rN)^{\rho(I_p)}\big)
&= \det\big(\bar\varphi T \mid \ol{V}^{\bar\rho(I_p)}/\ker(\ol{\rN})^{\bar\rho(I_p)}\big),
\end{align*} from which the equality of the $L$ and the $\epsilon$-factors follows.
\end{proof}

The discussion of paragraph \ref{subsec:epsilon} \textit{supra} then implies that this functional equation is, up to sign, precisely the one from Theorems \ref{thm:intro-k-odd} and \ref{thm:intro-k-even}. To conclude, we need to show that the sign is always positive for odd $k$; for this we use T. Saito's result~\cite{Saito} that the sign of the functional equation of the $L$-function of an orthogonal motive of even weight is always positive. \qed

\begin{remark}\label{rmk:l=p_comparison}
As explained in \cite[Cor.\,2.2\,(ii)]{PT15}, another consequence of the potential automorphy of the weakly compatible system $\{r_{k,\ell}\}_\ell$ is that it is indeed \textit{strictly} compatible (see~\hbox{\cite[p.\,214]{PT15}} for this notion). Given a prime $p$, this allows one to transfer some properties of the $\ell$-adic representations $r_{k,\ell}$ for $\ell \neq p$, to the $p$-adic representation $r_{k,p}$ of~$\Gal(\ol\QQ_p/\QQ_p)$, since their associated Weil-Deligne representations $\rho_{k, \ell}$ and $\rho_{k, p}$ are isomorphic up to semi\nobreakdash-simplification.
(See e.g., \cite[\S 1]{Taylor}
for the construction of the Weil-Deligne representation
associated with a $p$-adic de Rham representation.)
In particular, the representation $r_{k,p}$ is semistable over $L=\QQ_p$ if $k$ is even and $L=\QQ_p(\sqrt{-p})$ if $k$ is odd since $\rho_{k,\ell}(I_L) = \{1\}$, where $I_L \subset\Gal(\ol\QQ_p/L)$ denotes the inertia group (see Section \ref{subsec:epsilon}). Moreover, if $k$ is odd, then $\rN=0$, and hence $r_{k,p}$ is indeed crystalline over $L$. This strengthens the statement of Proposition \ref{Prop:p-adic_realization}, where semi-stability was only proved over $\QQ_p(\sqrt{-p})$. Besides,
we have shown in
\eqref{eq:bad-factors-odd-case} and \eqref{eq:local-factors-even}
that the polynomial $M_k(p;T)$ is a factor of
$\det ( 1-F_pT \mid V_{k, \ell}^{I_p})$
for all primes $p$ if $k$ is odd and for all primes $p\neq 2$ if $k$ is even.
From the (weak) admissibility of~$r_{k,p}$ and the equalities
\[
\det \big( 1-F_pT \mid V_{k, \ell}^{I_p}\big)
= \det\big( 1-F_pT \mid \ker(\rN)^{\rho(I_p)}\big)
= \det\big( 1-\varphi T \mid (V_{k,p}\otimes\BB_\st)^{\Gal(\ol\QQ_p/L)}\big),
\]
it follows that the $p$-adic Newton polygon of $M_k(p;T)$
lies above the Hodge polygon of~$\coH_{\dR,\rmid}^1(\Gm,\Sym^k\Kl_2)$. We then recover Corollary \ref{cor:Newton-Hodge}
by adding back the trivial factor,
including additionally the case of $p=2$ and odd $k$, which was not treated in \loccit
\end{remark}

\appendix
\section*{Appendix. Exponential mixed Hodge structures and\texorpdfstring{\\}{cr} irregular Hodge filtrations}
\refstepcounter{section}

In this appendix, we prove some of the theorems used in the main text concerning mixed Hodge structures obtained from exponential mixed Hodge structures. We start by recalling the necessary material on $\cD$-modules, mixed Hodge modules, and exponential mixed Hodge modules. Proposition \ref{prop:EMHSMHS} then provides us with a condition for getting a mixed Hodge structure from the irregular Hodge filtration, and a criterion for this condition to be satisfied is proved in Theorem \ref{th:EMHSMHStg}. Finally, Theorem \ref{th:MHSgeneral} gives a way to compute the corresponding Hodge and weight filtrations. As it is customary, we adopt the convention that filtrations with lower (\resp upper) indices are increasing (\resp decreasing), and we pass from increasing to decreasing filtrations by setting~\hbox{$\irrF^p=F^\irr_{-p}$ for any $p\in\QQ$,} and similarly for the ordinary Hodge filtration. In the theory of mixed Hodge modules one usually considers increasing filtrations, while the Hodge filtration of mixed Hodge structures is usually decreasing. We will make use of both conventions without further explanation.

\subsection{Notation and results from the theory of \texorpdfstring{$\cD$}{D}-modules}\label{subsec:notationD}

We refer the reader \eg to~\hbox{\cite{H-T-T08, Kashiwara03}} for this section, although the notation therein may be somewhat different.

Given an algebraic morphism $h\colon X\to Y$ between smooth complex algebraic varieties, we denote by $h_\bast$ (\resp $h^\bast$) the derived pushforward (\resp pullback) in the sense of $\cD$\nobreakdash-modules; that is, for a $\cD_X$-module or a bounded complex of $\cD_X$-modules (\resp $\cD_Y$-modules) $M$, we~set
\[
h_\bast M=\Rder h_*\Bigl(
\cD_{Y\leftarrow X}\overset{\mathrm{L}}{\otimes}_{\cD_X}M \Bigr)\quad (\resp h^\bast M=
\cD_{X\to Y}\overset{\mathrm{L}}{\otimes}_{h^{-1}\cD_Y}h^{-1}M).
\]
We denote by $\DR M$ the analytic de~Rham complex of $M$,
with $M^\an$ sitting in degree zero, and by $\pDR M$ the shifted complex $\DR M[\dim X]$, which is a perverse sheaf. We also denote by~$h_\bexc$ the adjoint by duality of~$h_\bast$ (that is, the functor $h_\bexc=\bD_Yh_\bast\bD_X$, where $\bD$ denotes the duality functor in the category of $\cD$-modules), so that there is a natural morphism $h_\bexc M\to h_\bast M$. In particular, given a holonomic~$\cD_X$\nobreakdash-module $M$ with regular singularities at infinity, the morphism \hbox{$h_\bexc M\to h_\bast M$} induces the natural morphism
$\Rder h_!\pDR M\to\Rder h_*\pDR M$ upon application of the shifted de~Rham functor and taking the isomorphisms $\pDR h_\bast M\simeq\Rder h_*\pDR M$ and~$\pDR h_\bexc M\simeq \Rder h_!\pDR M$ into account.

Given a $\cD_X$-module $M$ on a complex manifold (\resp smooth algebraic variety) $X$, we denote by $\coH^k_\dR(X,M)$ the hypercohomology in degree $k$ of the analytic (\resp algebraic) de~Rham complex of $M$. Note that here we do not shift the de~Rham complex as it is usually done in the theory of~$\cD$-modules. When dealing with an affine variety~$X$, we will identify algebraic~$\cD_X$\nobreakdash-modules with their global sections.

Let $j\colon U\hto X$ be the open embedding of the complement of a divisor $D$ on $X$. Given a holonomic $\cD_U$\nobreakdash-module $M$, the extension $j_\bast M$ is the holonomic $\cD_X$-module on which any local equation of $D$ acts in an invertible way. We denote by $j_\bea M$ the intermediate extension, defined as the maximal $\cD_X$-submodule of $j_\bast M$ that has no quotient supported on $D$. For a holonomic $\cD_U$-module $M$ with regular singularities at infinity, the inclusion $j_\bea M\to j_\bast M$ corresponds
to the natural morphism $j_{!{*}}\pDR M\to\Rder j_*\pDR M$ via the shifted de Rham functor.

If~$h\colon X \to Y$ is smooth, then $h^\bast$ sends holonomic $\cD_Y$-modules to holonomic $\cD_X$-modules. This functor corresponds to the usual pullback of vector bundles with connection. If $j\colon U\hto X$ is an open embedding, then $j^\bast$ is the usual restriction functor from holonomic $\cD_X$-modules to holonomic $\cD_U$-modules.

Instrumental for the theory of mixed Hodge modules \cite{MSaito87} is the notion of \emph{nearby cycle} and \emph{vanishing cycle functors} $\psi_f$ and $\phi_f$ along a function $f\colon X\to\Afu$ on the category of perverse sheaves and holonomic $\cD$-modules on~$X$. If $M$ is a holonomic $\cD_X$-module, then~$\psi_f M$ and~$\phi_fM$ are holonomic $\cD_X$\nobreakdash-modules supported on $f^{-1}(0)$, which are defined in terms of the \emph{Kashiwara\nobreakdash-Malgrange filtration} of~$M$. Both $\psi_f M$ and $\phi_fM$ are equipped with an automorphism $\mathrm T$ and decompose with respect to its eigenvalues. We denote by $\psi_{f,\lambda}M$ and $\phi_{f,\lambda}M$ the generalized eigencomponent corresponding to an eigenvalue $\lambda\in\CC^*$. If $\lambda\neq1$, there is an isomorphism $\psi_{f,\lambda}M\simeq\phi_{f,\lambda}M$ compatible with $\mathrm T$, whereas for $\lambda=1$ there is a quiver\vspace*{-5pt}
\[
\xymatrix{
\psi_{f,1}M\ar@/^1pc/[r]^-{\can}&\phi_{f,1}M\ar@/^1pc/[l]^-{\var}
}
\]
with the property that the maps $\exp(2\pi\sfi\var\circ\can)$ and $\exp(2\pi\sfi\can\circ\var)$ coincide with the unipotent automorphism $\mathrm T$ on $\psi_{f,1}M$ and $\phi_{f,1}M$ respectively. The nilpotent endomorphisms $\var\circ\can$ and $\can\circ\var$ on $\psi_{f,1}M$ and $\phi_{f,1}M$ are denoted by $\rN$. Up to a suitable shift, the nearby and vanishing cycles of a regular holonomic $\cD_X$-module $M$ correspond to the nearby and vanishing cycles of the perverse sheaf $\pDR M$.

Given a meromorphic function $\varphi\in\Gamma(X,\cO_X(*P))$ on $X$ with pole divisor $P$, we set
\begin{equation*}\label{eq:Evarphi}
E^\varphi=(\cO_X(*P),\rd+\rd\varphi),
\end{equation*}
and we denote by $\rme^\varphi$ the generator $1\in \cO_X(*P)$ of $E^\varphi$. Consider the product $\Afu_\tt\times\Afu_\taut$ of two affine lines with coordinates $\tt$ and $\taut$, and denote by~$p_\tt$ and $p_\taut$ the projections to the first and the second factors respectively. The \textit{Fourier transform} of a $\cD$\nobreakdash-module (or~a~bounded complex of $\cD$-modules)~$M$ on the affine line $\Afu_\taut$ is the complex defined as
\[
\FT_\taut M= p_{\tt\bast}(p_\taut^\bast M\otimes E^{\tt\taut}).
\]
If $M$ is a holonomic $\Cltaut$-module, then $\FT_\taut M$ has cohomology concentrated in degree zero (\ie is a holonomic $\Cltt$-module). This yields a functor $\FT_\taut \colon \catD^\rb_\hol(\cD_{\Afu_\taut})\to\catD^\rb_\hol(\cD_{\Afu_\tt})$. If~$M$ has regular singularities everywhere, including at infinity, then the only singularity of~$\FT_\taut M$ on $\Afu_\tt$ is the origin, which is also regular.

Let $s\colon\Afu\times\Afu\to\Afu$ denote the sum map. The \emph{additive $*$-convolution} $M_1\star_*M_2$ of $M_1$ and~$M_2$ is the object $s_+(M_1\boxtimes M_2)$ of $\catD^\rb_\hol(\cD_{\Afu})$. This operation is associative and corresponds to derived tensor product through Fourier transformation:
\[
\FT(M_1\star_*M_2)\simeq\FT M_1\overset{\mathrm{L}}\otimes\FT M_2=\delta^+(\FT M_1\boxtimes\FT M_2),
\]
where $\delta\colon \Afu \hto \Afu\times\Afu$ is the diagonal embedding. For holonomic $\cD$\nobreakdash-modules with regular singularities $M_1$ and $M_2$, the cohomologies in non-zero (\ie negative) degrees of the complex~\hbox{$\FT M_1\otimes^\mathrm{L}\FT M_2$} are supported at the origin of~$\Afu$, so the corresponding cohomologies of~$M_1\star_*M_2$ are constant $\cD_{\Afu}$-modules. We refer the reader to \cite[\S1.1]{D-S12} for details.

\subsection{Notation and results from the theory of mixed Hodge modules}\label{subsec:notaMHM}

Let $X$ be a smooth complex projective variety,
and let $M^\rH=(M,F_\bbullet M,\cF_\QQ, \alpha)$ be the data of a regular holonomic left $\cD_X$\nobreakdash-module $M$, an increasing good filtration $F_\bbullet M$ on $M$, a $\QQ$\nobreakdash-perverse sheaf~$\cF_\QQ$ on $X$, and an isomorphism $\alpha\colon\pDR M\isom\cF_\CC$. As it is customary, we shall omit~$\alpha$ from the notation. Let~$\Sp$ denote the Spencer complex of a right $\cD_X$-module. We say that $M^\rH$ is a \emph{pure polarizable Hodge module} of weight $w$ if the associated right filtered $\cD_X$\nobreakdash-mod\-ule
\[
(\omega_X\otimes_{\cO_X}M,\omega_X\otimes_{\cO_X}F_{\bbullet-\dim X}M),
\]
together with~$\cF_\QQ$ and the isomorphism \hbox{$\Sp(\omega_X\otimes M)=\pDR M\simeq\cF_\CC$} forms a pure polarizable Hodge module of weight $w$ in the sense of Saito \hbox{\cite[5.1.6\,\&\,5.2.10]{MSaito86}} (\cf also the introduction of \loccit). There is a similar definition for left \emph{mixed Hodge modules},
with the left-to-right correspondence
$W_\bbullet M\leftrightarrow \omega_X\otimes W_\bbullet M$
between weight filtrations. We refer the reader to \cite{MSaito87} for the properties of the category $\MHM(X)$ of \emph{algebraic mixed Hodge modules}, which are always assumed to be graded-polarizable. The derived category $\catD^\rb(\MHM(X))$ is endowed with a six-functors formalism, and we denote the functors with a lower left index~$\rH$ in order to keep the category in mind. For example, the pushforward functor $\Hm f_*$ by a morphism~$f$ for left mixed Hodge modules is obtained by composing the similar functor for mixed Hodge modules with the side-changing functor. There is a similar definition for other functors. In particular, additive $*$-convolution of holonomic $\cD_{\Afu}$-modules lifts to $\MHM(\Afu)$, see formula \eqref{eqn:additiveconvolutionMHM} below. Be aware that some of these functors may not correspond to the corresponding functors on the underlying $\cD$-modules: for example, $\Hm\otimes$ does not correspond to $\otimes^\mathrm{L}$ as used above.

Assume that $(M,F_\bbullet M)$ underlies a pure Hodge module $M^\rH$ of weight $w$ that is smooth (\ie~$M$ is a $\cO_X$-locally free module of finite rank with integrable connection~$\nabla$) and consider the associated decreasing filtration $F^pM=F_{-p}M$. Then $(M,\nabla,F^\bbullet M)$ is a polarizable variation of pure Hodge structures of weight $w-\dim X$.

Let $U$ be a smooth complex quasi-projective variety of dimension $d$. We denote by~$\pQQ_U^\rH$ the pure Hodge module of weight $d$ with underlying perverse sheaf $\pQQ_U=\QQ_U[d]$ and underlying filtered $\cD_U$-module $(\cO_U,F_\bbullet\cO_U)$ with $\gr_p^F\cO_U=0$ except for $p=0$ (it is denoted by $\QQ_U^\rH$ in~\cite{MSaito87}).

The duality functor $\Hm\bD$ is a contravariant anti t-exact involution on $\catD^\rb(\MHM(X))$, and hence preserves $\MHM(X)$. There is a natural isomorphism $\Hm\bD\,\pQQ_U^\rH\simeq\pQQ_U^\rH(-d)$.

Given a morphism $f\colon U\to\Afu_\ts$, there are functors
\[
\Hm f_*,\Hm f_!\colon \catD^\rb(\MHM(U))\to\catD^\rb(\MHM(\Afu)).
\]
For the open immersion $j\colon U\setminus D\hto U$ of the complement of a divisor $D$, the localization functor $\Hm j_*\,\Hm j^*$ is simply denoted by $[*D]$ (and the corresponding functor for $\cD$\nobreakdash-modules by~$(*D)$), while the dual localization functor $\Hm j_!\,\Hm j^*$ is denoted by $[!D]$
(and by~$(!D)$, respectively). If~\hbox{$i:D\hto U$} denotes the closed immersion, there are distinguished dual triangles in~$\catD^\rb(\MHM(U))$ (\cf\cite[(4.4.1)]{MSaito87}):
\begin{gather}\label{eq:exactseqloc1}
\Hm i_*\,\Hm i^! M^\rH\to M^\rH\to M^\rH[*D]\To{+1},\\
M^\rH[!D]\to M^\rH\to \Hm i_*\,\Hm i^* M^\rH\To{+1},\notag\label{eq:exactseqloc2}
\end{gather}
which, for $M^\rH\in\MHM(U)$, reduce to the dual exact sequences
\begin{gather*}
0\to\cH^0\Hm i_*\,\Hm i^! M^\rH\to M^\rH\to M^\rH[*D]\to \cH^1\Hm i_*\,\Hm i^! M^\rH\to0,\\
0\to\cH^{-1}\Hm i_*\,\Hm i^* M^\rH\to M^\rH[!D]\to M^\rH\to \cH^0\Hm i_*\,\Hm i^* M^\rH\to0.
\end{gather*}

We also make use of the nearby and vanishing cycle functors $\psi_f=\psi_{f,1}\oplus\psi_{f,\neq1}$ and $\phi_{f,1}$ on~$\MHM(U)$, together with the associated nilpotent operator $\rN$. If $f$ is smooth in the neighborhood of $f^{-1}(0)$, we regard them as taking values in $\MHM(f^{-1}(0))$.

\begin{example}\label{exem:MHSgeneral}
Set $U=\Afu$ and $f=\id$, and let $M^\rH$ be a mixed Hodge module on~$\Afu$. Since we are considering left modules, the convention for filtrations is that $\dim F^p\psi_fM=\rk F^pM$ and, if $0$ is not a singular point of $M$, then $W_\ell\psi_fM=\psi_fW_{\ell+1}M$.

If $M^\rH$ is pure, then $M^\rH=M_1^\rH \oplus M_2^\rH$ is the direct sum of a Hodge module $M_1^\rH$ whose underlying~$\cD_{\Afu}$\nobreakdash-module $M_1$ has no section supported at $0$ and a Hodge module $M_2^\rH$ supported at $0$. Then the following holds:
\[
\cH^0\Hm i_*\,\Hm i^! M_1^\rH=0,\quad \cH^1\Hm i_*\,\Hm i^! M_2^\rH=0,\quad \cH^0\Hm i_*\,\Hm i^! M_2^\rH\simeq \Hm i_*\,\Hm i^! M_2^\rH.
\]
Moreover, $\cH^1\Hm i_*\,\Hm i^! M_1^\rH$ is isomorphic to the mixed Hodge structure
\[
\coker\rN_\taut\colon\psi_{\taut,1}M_1^\rH\to\psi_{\taut,1}M_1^\rH(-1),
\]
where $\rN_\taut$ denotes the nilpotent part of the monodromy operator for its eigenvalue one on the nearby cycles of $M^\rH$ at the origin. Indeed, considering the diagram
\[
\xymatrix{
\phi_{\taut,1}M^\rH_1\ar[d]_\var\ar[r]&\phi_{\taut,1}(M^\rH_1[*0])\ar[d]^\var_\wr\\
\psi_{\taut,1}M^\rH_1(-1)\ar[r]^-\sim&\psi_{\taut,1}(M^\rH_1[*0])(-1),
}
\]
where the horizontal arrows are functorially obtained from $M^\rH_1\to M^\rH_1[*0]$, the mixed Hodge structure $\cH^1\Hm i_*\,\Hm i^! M_1^\rH$ is identified with the cokernel of the upper horizontal arrow, and hence of the left vertical one. Since $\can\colon\psi_{\taut,1}M^\rH_1\to\phi_{\taut,1}M^\rH_1$ is an epimorphism and $\rN_\taut=\var\circ\can$, the conclusion follows.
\end{example}

\begin{example}\label{exem:locsequences}
Given a reduced divisor $D\subset U$, define $\pQQ_D^\rH=\Hm a_D^*\QQ^\rH_{\Spec\CC}[\dim D]$, where $a_D$ denotes the structure morphism. Then there is an isomorphism
\[
\Hm i^*\,\pQQ_U^\rH=\cH^0\Hm i^*\,\pQQ_U^\rH\simeq \pQQ_D^\rH
\]
and an exact sequence
\[
0\to\Hm i_*\pQQ_D^\rH\to\pQQ_U^\rH[!D]\to\pQQ_U^\rH\to0.
\]
If, moreover, $D$ is smooth, there is also an isomorphism
\[
\Hm i^!\,\pQQ_U^\rH[1]=\cH^1\Hm i^!\,\pQQ_U^\rH\simeq \pQQ_D^\rH(-1),
\]
and the above exact sequence can be completed to a diagram
\begin{equation}\label{eqn:exactsequenceExampleA3}
\begin{aligned}
\xymatrix{
0\ar[r]&\pQQ_U^\rH\ar[r]&\pQQ_U^\rH[*D]\ar[r]&\Hm i_*\pQQ_D^\rH(-1)\ar[r]&0,\\
0&\ar[l]\pQQ_U^\rH\ar@{=}[u]\ar@{}[ur]|\circlearrowright&\ar[l]\pQQ_U^\rH[!D]\ar[u]&\ar[l]\Hm i_*\pQQ_D^\rH&\ar[l]0.
}
\end{aligned}
\end{equation} in which the lower row is the dual exact sequence and the square commutes.
\end{example}

\subsection{A review of exponential mixed Hodge structures}
\label{subsec:FT_EMHS}

Let $\Afu_\ts$ be the affine line with coordinate $\ts$. In \cite[\S4]{K-S10}, Kontsevich and Soibelman define the category $\EMHS$ of \emph{exponential mixed Hodge structures} as the full subcategory of $\MHM(\Afu_\ts)$ consisting of those objects~$N^\rH$ whose underlying perverse sheaf has vanishing global cohomology. The assignment
\begin{displaymath}
N^\rH \mto \Pi_\ts(N^\rH)=N^\rH \star \Hm j_! \cO_{\Gm}^\rH,
\end{displaymath}
where $\star$ stands for additive $\ast$-convolution defined as
\begin{equation}\label{eqn:additiveconvolutionMHM}
N_1^\rH\star N_2^\rH=\Hm s_*(N_1^\rH\boxtimes N_2^\rH)
\end{equation} and $j \colon \Gm \hookrightarrow \Afu$ for the inclusion, yields an exact functor~$\Pi_\ts\colon\MHM(\Afu_\ts)\to\MHM(\Afu_\ts)$, which is a projector onto $\EMHS$ that is left adjoint to its inclusion as a subcategory. In particular, there is a natural morphism $N^\rH\to\Pi_\ts(N^\rH)$ in~$\MHM(\Afu_\ts)$. Its kernel and cokernel are constant mixed Hodge modules on $\Afu_\ts$. More precisely, the~$\cD_{\Afu_\ts}$\nobreakdash-module underlying the kernel is the maximal constant submodule of~$N$. We will also write $\Pi$ for~$\Pi_\theta$ when the coordinate is clear.

Each object of $\EMHS$ is endowed with a weight filtration, defined from that in $\MHM(\Afu_\ts)$ by the formula $W^{\scriptscriptstyle\mathrm{EMHS}}_n\Pi(N^\rH)=\Pi(W_n N^\rH)$ for any mixed Hodge module $N^\rH$ on~$\Afu_\ts$.

\begin{remark}\label{rem:WEMHS} This weight filtration has the following properties:
\begin{itemize}
\item
if $N=\Pi(N)$ (and hence $N^\rH=\Pi(N^\rH)$) there is a functorial morphism
\[
W_nN^\rH\to W^{\scriptscriptstyle\mathrm{EMHS}}_nN^\rH;
\]

\item
if $N^\rH$ is mixed of weights $\leq n$ (\resp $\geq n$) in $\MHM(\Afu_\ts)$, then $\Pi(N^\rH)$ is mixed of weights $\leq n$ (\resp $\geq n$) in $\EMHS$; therefore, if $N^\rH$ is pure of weight $n$, then $\Pi(N^\rH)$ (which in general is mixed of weights $\geq n$ as an object of $\MHM(\Afu_\ts)$) is a pure object of $\EMHS$ of weight $n$.

\item
For the sake of simplicity, when there is no risk of confusion, we write $W_\bbullet \Pi(N^\rH)$ instead of $W_\bbullet^{\scriptscriptstyle\mathrm{EMHS}}\Pi(N^\rH)$ for the weight filtration of an exponential mixed Hodge structure~$\Pi(N^\rH)$.
\end{itemize}
\end{remark}

The category $\EMHS$ is endowed with a tensor structure induced by the additive convolution~$\star$ on $\MHM(\Afu)$ such that the equality
\[
\Pi(N_1^\rH)\star\Pi(N_2^\rH)=\Pi(N_1^\rH\star N_2^\rH)
\]
holds. Moreover, the weight filtration is strictly compatible with the tensor product (\cf \hbox{\cite[Prop.\,4.3]{K-S10}}). Note that $N_1^\rH\star N_2^\rH$ could be a complex, but its cohomologies in non-zero degrees are constant (this is checked on the underlying $\cD$-modules), and hence are annihilated by $\Pi$; on the other hand, the cohomology of $\Pi(N_1^\rH)\star\Pi(N_2^\rH)$ is concentrated in degree zero.

Recall that the functor
\begin{equation}\label{eqn:exactfunctor}
\Mod_{\hol\reg}(\cD_{\Afu_\ts})\to\mathsf{Vect}_\CC, \quad N\mto \coH^1_{\dR}(\Afu_\ts,N\otimes \fqq{\ts})
\end{equation}
is exact and that
$\coH^j_{\dR}(\Afu_\ts,N\otimes \fqq{\ts})$ vanishes for $j\neq1$. Moreover, the morphism $N\to\Pi(N)$ induces an isomorphism
\[
\coH^1_{\dR}(\Afu_\ts,N\otimes \fqq{\ts})\to\coH^1_{\dR}(\Afu_\ts,\Pi(N)\otimes \fqq{\ts}).
\]

Let $N^\rH$ be a mixed Hodge module on~$\Afu_\ts$, and let $\Pi(N^\rH)$ be the associated object in $\EMHS$. Let $a_{\Afu_\ts}$ denote the structure morphism of $\Afu_\ts$. The \emph{de~Rham fiber functor} on $\EMHS$ is given~by
\begin{equation}\label{eq:deRhamfibre}
\Pi(N^\rH)\mto\coH^1_{\dR}(\Afu_\ts,\Pi(N)\otimes \fqq{\ts})= \coH^1_{\dR}(\Afu_\ts,N\otimes \fqq{\ts})
=\coH^0a_{\Afu_\ts,\bast}(N\otimes \fqq{\ts}).
\end{equation}
An object of $\EMHS$ is zero if and only if its de~Rham fiber is zero. It follows from the exactness of \eqref{eqn:exactfunctor} that the natural morphism
\begin{equation}\label{eqn:imageweightfiltration}
\coH^1_{\dR}(\Afu_\ts, W_\bbullet N \otimes \fqq{\ts})\to \coH^1_{\dR}(\Afu_\ts,N\otimes \fqq{\ts})
\end{equation} is injective. The filtration $W_\bbullet\coH^1_{\dR}(\Afu_\ts, N\otimes \fqq{\ts})$ is defined as its image. Then any morphism $N_1^\rH\to N_2^\rH$ induces a strictly filtered morphism
\[
(\coH^1_{\dR}(\Afu_\ts,N_1\otimes \fqq{\ts}),W_\bbullet)\to
(\coH^1_{\dR}(\Afu_\ts,N_2\otimes \fqq{\ts}),W_\bbullet).
\]

On the other hand, letting $j_\infty\colon\Afu_\ts\hto\PP^1_\ts$ denote the inclusion, the (decreasing) \emph{irregular Hodge filtration} on $j_{\infty\bast}(N\otimes \fqq{\ts})$ is defined in~\cite[\S6.b]{Bibi08} out of the Hodge filtration $F^\bbullet N$ of~$N^\rH$. It is indexed by $\QQ$ and denoted by $F_\mathrm{Del}^\bbullet$ in~\loccit, whereas here we denote it by $\irrF^\bbullet$, to emphasize the compatibility with other possible definitions \cite{E-S-Y13, S-Y14}.

\begin{prop}\label{prop:bifilt}
The natural morphism
\[
\coH^1_{\dR}(\PP^1_\ts,\irrF^\bbullet j_{\infty\bast}(N\otimes \fqq{\ts}))\to
\coH^1_{\dR}(\PP^1_\ts,j_{\infty\bast}(N\otimes \fqq{\ts}))=\coH^1_{\dR}(\Afu_\ts,N\otimes \fqq{\ts})
\]
is injective and, defining $\irrF^\bbullet\coH^1_{\dR}(\Afu_\ts,N\otimes \fqq{\ts})$ as its image,
the functor from $\MHM(\Afu_\ts)$ to bifiltered vector spaces\enlargethispage{\baselineskip}%
\begin{equation}\label{eqn:factorsthrough}
N^\rH\mto \bigl(\coH^1_{\dR}(\Afu_\ts,N\otimes \fqq{\ts}),\irrF^\bbullet,W_\bbullet)
\end{equation}
factors through $\Pi$.
Any morphism in $\MHM(\Afu_\ts)$ (or in $\EMHS$) gives rise by means of \eqref{eqn:factorsthrough} to a strictly bifiltered morphism.
\end{prop}

\begin{proof}
Injectivity is proved in \cite[Th.\,6.1]{Bibi08} for polarized Hodge modules. The case of mixed Hodge modules is deduced from it in \hbox{\cite[Th.\,3.3.1]{E-S-Y13}} (in~a more general setting). Since the kernel and the cokernel of~$N^\rH\to\Pi(N^\rH)$ are constant mixed Hodge modules and $\coH^j_{\dR}(\Afu_\ts,N\otimes \fqq{\ts})$ vanishes in all degrees $j$ for constant $N$, it is clear that \eqref{eqn:factorsthrough} factors through~$\Pi$. The last statement is obtained by successively applying Theorems 0.2(2),~0.3(4), and~0.3(2) of~\cite{Bibi15}.
\end{proof}

\subsection{Mixed Hodge structures as exponential mixed Hodge structures}
\label{subsec:MHS_EMHS}

Let $\MHS$ be the category of mixed Hodge structures, endowed with its natural tensor product. We consider the open immersion $j_0\colon \Gm\hto\Afu_\ts$ and the closed immersion $i_0\colon \{0\}\hto\Afu_\ts$. Letting $\Hm i_{0!}$ denote the pushforward functor $\MHS\to\MHM(\Afu_\ts)$, there is an isomorphism $\phi_{\ts,1}\circ\Hm i_{0!}\simeq\id_{\MHS}$. Let us now consider the composed functor $\Pi\circ\Hm i_{0!}:\MHS\to\EMHS$. Then there is an isomorphism $\phi_{\ts,1}\circ(\Pi\circ\Hm i_{0!})\simeq\id_{\MHS}$,
since $\phi_{\ts,1}\circ\Pi\simeq \phi_{\ts,1}$
(because $\phi_{\ts,1}$ of a constant object in $\MHM(\Afu_\ts)$ is zero). From standard properties of mixed Hodge modules one checks the following lemma.

\begin{lemma}\label{lem:weightsMHS}
The functor $\Pi\circ\Hm i_{0!}\colon\MHS\to\EMHS$ is compatible with tensor products and makes $\MHS$ into a full tensor subcategory of $\EMHS$. Moreover, if $V^\rH$ is a mixed Hodge structure
with weight filtration $W_\bbullet V^\rH$ and associated exponential mixed Hodge structure $\Pi_\theta(\Hm i_{0!}V^\rH)$, then the equality $W_\bbullet\Pi_\theta(\Hm i_{0!}V^\rH)=\Pi_\theta(\Hm i_{0!}W_\bbullet V^\rH)$ holds (\cf Remark~\ref{rem:WEMHS} for the notation).
\end{lemma}

\begin{prop}\label{prop:EMHSMHS}
Let $N^\rH$ be an object of $\MHM(\Afu_\ts)$ such that $\Pi(N^\rH)$ belongs to $\MHS$. Then the bifiltered vector space $\bigl(\coH^1_{\dR}(\Afu_\ts,N\otimes \fqq{\ts}),\irrF^\bbullet,W_\bbullet)$ is naturally isomorphic to that associated with the mixed Hodge structure $\phi_{\ts,1}N^\rH$.
\end{prop}

\begin{proof}
Let $V^\rH$ be a mixed Hodge structure. The vector space $\bigl(\coH^1_{\dR}(\Afu_\ts, i_{0\bast}V\otimes \fqq{\ts}),\irrF^\bbullet,W_\bbullet)$ endowed with its two filtrations is easily identified with $(V,F^\bbullet,W_\bbullet)$. Taking the isomorphism
\[
\bigl(\coH^1_{\dR}(\Afu_\ts,\Pi(i_{0\bast}V)\otimes \fqq{\ts}),\irrF^\bbullet,W_\bbullet)
\simeq
\bigl(\coH^1_{\dR}(\Afu_\ts, i_{0\bast}V\otimes \fqq{\ts}),\irrF^\bbullet,W_\bbullet)
\]
into account, we deduce an isomorphism
\begin{equation}\label{eq:EMHSMHS}
\bigl(\coH^1_{\dR}(\Afu_\ts,\Pi(i_{0\bast}V)\otimes \fqq{\ts}),\irrF^\bbullet,W_\bbullet) \simeq(V,F^\bbullet,W_\bbullet).
\end{equation}
Now, if $\Pi(N^\rH)$ belongs to $\MHS$, then $\Pi(N^\rH)=\Pi(\Hm i_{0!}V^\rH)$ for some mixed Hodge structure~$V^\rH$, which satisfies $V^\rH\simeq\phi_{\ts,1}\Pi(N^\rH)\simeq\phi_{\ts,1}N^\rH$. Then \eqref{eq:EMHSMHS} gives
\begin{align*}
(\phi_{\ts,1}N,F^\bbullet,W_\bbullet)&\simeq\bigl(\coH^1_{\dR}(\Afu_\ts,\Pi(i_{0\bast}\phi_{\ts,1}N)\otimes \fqq{\ts}),\irrF^\bbullet,W_\bbullet)\\
&\simeq\bigl(\coH^1_{\dR}(\Afu_\ts,\Pi(N)\otimes \fqq{\ts}),\irrF^\bbullet,W_\bbullet)\\
&\simeq\bigl(\coH^1_{\dR}(\Afu_\ts,N\otimes \fqq{\ts}),\irrF^\bbullet,W_\bbullet),
\end{align*} as we wanted to show.
\end{proof}

\subsection{The exponential mixed Hodge structures associated with a function}
\label{subsec:EMHS_Uf}
Let~$U$ be a smooth complex quasi-projective variety of dimension~$d$ and $f\colon U\to\Afu_\ts$ a regular function. For each object $N_U^\rH$ of $\MHM(U)$ and each integer $r$, we define the mixed Hodge modules
\[
(N^\rH)^r_*=\cH^{r-d}\Hm f_*N_U^\rH,\quad
(N^\rH)^r_!=\cH^{r-d}\Hm f_!N_U^\rH,\quad (N^\rH)^r_\rmid=\image[(N^\rH)^r_!\to(N^\rH)^r_*]
\]
on $\Afu_\ts$, and we denote by $N^r_*,N^r_!,N^r_\rmid$ the respective underlying $\cD_{\Afu_\ts}$-modules. Upon application of the projector $\Pi_\ts$, they define objects of $\EMHS$:
\begin{gather}
\coH^r(U,N^\rH,f)=\Pi_\ts((N^\rH)^r_*),\quad\coH_\rc^r(U,N^\rH,f)=\Pi_\ts((N^\rH)^r_!),\\\coH^r_\rmid(U,N^\rH,f)=\Pi_\ts((N^\rH)^r_\rmid)=\image[\coH_\rc^r(U,N^\rH,f)\to\coH^r(U,N^\rH,f)].
\end{gather}

\begin{example}\label{exem:f=0}
Let $a_U$ denote the structure morphism of $U$. For $f$ the zero function, the above are the mixed Hodge structures
\[
\coH^r(U,N^\rH,0)=\coH^{r-d}(\Hm a_{U*}N^\rH), \quad \coH_\rc^r(U,N^\rH,0)=\coH^{r-d}_\rc(\Hm a_{U*}N^\rH).
\]
\end{example}

There are filtered isomorphisms between the filtered de~Rham fibers
\begin{align}\label{eq:EMHSMHS1}
(\coH^r_\dR(U,N_U\otimes \fqq{f}),\irrF^\bbullet) &\simeq
(\coH^1_\dR(\Afu_\ts,N^r_*\otimes \fqq{\ts}),\irrF^\bbullet),\\
\label{eq:EMHSMHS2}
(\coH^r_{\dR,\rc}(U,N_U\otimes \fqq{f}),\irrF^\bbullet) &\simeq
(\coH^1_\dR(\Afu_\ts,N^r_!\otimes \fqq{\ts}),\irrF^\bbullet).
\end{align}
This follows from \cite[Th.\,1.3(4)]{S-Y14} applied to a compactification of the morphism $f$ and by taking the pushforward by the structure morphism. For the second isomorphism, we use the equality $\coH^1_{\dR,\rc}(\Afu_\ts,N^r_!\otimes \fqq{\ts})= \coH^1_\dR(\Afu_\ts,N^r_!\otimes \fqq{\ts})$, which follows from the fact that $N^r_!\otimes \fqq{\ts}$ has a purely irregular singularity at $\infty$.

The natural morphism $(\coH^r_{\dR,\rc}(U,N_U\otimes \fqq{f}),\irrF^\bbullet)\to(\coH^r_\dR(U,N_U\otimes \fqq{f}),\irrF^\bbullet)$ is strict and its image $\coH^r_{\dR,\rmid}(U,N_U\otimes \fqq{f})$ is endowed with the quotient (equivalently sub) filtration $\irrF^\bbullet$.

On the other hand, the right-hand sides of~\eqref{eq:EMHSMHS1} and \eqref{eq:EMHSMHS2} underlie bifiltered de~Rham fibers $(\coH^1_\dR(\Afu_\ts,N^r_*\otimes \fqq{\ts}),\irrF^\bbullet,W_\bbullet)$ and \hbox{$(\coH^1_{\dR,\rc}(\Afu_\ts,N^r_!\otimes \fqq{\ts}),\irrF^\bbullet,W_\bbullet)$}, the weight filtration~$W_\bbullet$ being defined as the image of \eqref{eqn:imageweightfiltration}. Thanks to the last statement of Proposition \ref{prop:bifilt}, the bifiltered de~Rham fiber $(\coH^1_{\dR,\rmid}(\Afu_\ts,N^r_!\otimes \fqq{\ts}),\irrF^\bbullet,W_\bbullet)$ is unambiguously defined.

\begin{defi}[{\cite[\S5]{K-S10}}] Associated with a function $f\colon U\to\Afu_\ts$ and an integer $r$ as above are the exponential mixed Hodge structures
\[
\coH^r(U,f)=\Pi_\ts((\pQQ^\rH_U)^r_*), \quad \coH_\rc^r(U,f)=\Pi_\ts((\pQQ^\rH_U)^r_!), \quad \coH_\rmid^r(U,f)=\Pi_\ts((\pQQ^\rH_U)^r_\rmid),
\]
with corresponding bifiltered de~Rham fibers
\[
(\coH^r_\dR(U,\fqq{f}),\irrF^\bbullet,W_\bbullet), \quad (\coH^r_{\dR,\rc}(U,\fqq{f}),\irrF^\bbullet,W_\bbullet), \quad (\coH^r_{\dR,\rmid}(U,\fqq{f}),\irrF^\bbullet,W_\bbullet).
\]
\end{defi}

\begin{prop}\label{prop:weightsHdUf}
The exponential mixed Hodge structure $\coH_\rc^d(U,f)$ is mixed of weights $\leq d$, $\coH^d(U,f)$ is mixed of weights $\geq d$, and $\coH_\rmid^d(U,f)$ is pure of weight $d$. Moreover, the following properties are equivalent:
\begin{enumerate}
\item\label{prop:weightsHdUf1}
the natural morphism $\gr_d^W\coH_\rc^d(U,f)\to\gr_d^W\coH^d(U,f)$ is an isomorphism,
\item\label{prop:weightsHdUf2} the equality
$\coH_\rmid^d(U,f)=W_d\coH^d(U,f)$ holds.
\end{enumerate}
\end{prop}

\begin{proof}
It follows from the standard behaviour of weights with respect to pushforward functors (here applied to mixed Hodge modules, \cf \cite[(4.5.2)]{MSaito87}) that the object $\cH^0\Hm f_!\pQQ_U^\rH$ of~$\MHM(\Afu_\ts)$ has weights~$\leq d$, and hence that the object $\coH_\rc^d(U,f)$ of $\EMHS$ has weights~$\leq d$ by Remark \ref{rem:WEMHS}. The argument for $\coH^d(U,f)$ and $\coH_\rmid^d(U,f)$ is the same.

For the last assertion, let us denote by $\cH^0\Hm f_{!*}\pQQ_U^\rH$ the image in $\MHM(\Afu_\ts)$ of the natural morphism $\cH^0\Hm f_{!}\pQQ_U^\rH\to\cH^0\Hm f_{*}\pQQ_U^\rH$. Then the exact sequence
\[
0\to\cH^0\Hm f_{!*}\pQQ_U^\rH\to\cH^0\Hm f_{*}\pQQ_U^\rH\to N^\rH\to0
\]
in $\MHM(\Afu_\ts)$, which defines the object $N^\rH$, dualizes (with Tate twist $-d$) to
\[
0\to N_1^\rH\to\cH^0\Hm f_{!}\pQQ_U^\rH\to\cH^0\Hm f_{!*}\pQQ_U^\rH\to0,
\]
with $N_1^\rH=\Hm\bD N^\rH(-d)$ and the composition of the above morphisms
\[
\cH^0\Hm f_{!}\pQQ_U^\rH\to\cH^0\Hm f_{!*}\pQQ_U^\rH\to\cH^0\Hm f_{*}\pQQ_U^\rH
\]
is the natural morphism. By Remark~\ref{rem:WEMHS}, applying $\Pi$ and $\gr_d^W$ gives exact sequences
\begin{align*}
0\to\coH_\rmid^d(U,f)\to\gr_d^W\coH^d(U,f)\to \gr_d^W \Pi(N^\rH)\to0, \\
0\to\gr_d^W \Pi(N_1^\rH)\to\gr_d^W\coH_\rc^d(U,f)\to\coH_\rmid^d(U,f)\to0.
\end{align*}
Property \eqref{prop:weightsHdUf2} is thus equivalent to the vanishing $\gr_d^W\Pi(N^\rH)=0$, which means that $\gr_d^W\!N^\rH$ is constant. By duality, this is equivalent to $\gr_d^W\!N_1^\rH$ being constant, and hence to the vanishing $\gr_d^W\Pi(N_1^\rH)=0$, which is precisely Property~\eqref{prop:weightsHdUf1}.
\end{proof}

\begin{example}\label{exem:locexactsequence}
Let $D$ be a divisor on which the function $f$ vanishes. By applying $\Hm f_!$ and the exact functor $\Pi_\ts$ to the lower exact sequence in \eqref{eqn:exactsequenceExampleA3}, we obtain a long exact sequence
\begin{equation}\label{eqn:localizationexactsequence}
\cdots\to\coH^{r-1}_\rc(D)\to\coH^r_\rc(U\moins D,f)\to\coH^r_\rc(U,f)\to\coH^r_\rc(D)\to\cdots
\end{equation} on noting the equality $\coH^j_\rc(D, 0)=\coH^{j}_\rc(D)$ from Example \ref{exem:f=0}.
If, moreover, $D$ is smooth, by also applying $\Hm f_*$ to the upper exact sequence in \eqref{eqn:exactsequenceExampleA3}, we obtain a diagram
\[
\xymatrix{
\cdots\ar[r]&\coH^{r-2}(D)(-1)\ar[r]&\coH^r(U,f)\ar[r]&\coH^r(U\moins D,f)\ar[r]&\coH^{r-1}(D)(-1)\ar[r]&\cdots\\
\cdots&\coH^r_\rc(D)\ar[l]&\coH^r_\rc(U,f)\ar[u]\ar@{}[ur]|\circlearrowright\ar[l]&\coH^r_\rc(U\moins D,f)\ar[u]\ar[l]&\ar[l]\coH^{r-1}_\rc(D)&\ar[l]\cdots
}
\]
in the category $\EMHS$, where the rows are long exact sequences and the square commutes.

\end{example}

\begin{example}\label{exem:Gmx}\mbox{}
In this example, we construct isomorphisms
\[
\coH^k(\Gm^k,x_1+\cdots+x_k)\simeq\QQ(-k) \quad \text{and} \quad \coH^k_\rc(\Gm^k,x_1+\cdots+x_k)\simeq\QQ(0)
\]
for each integer $k \geq 1$. For this, we first observe the vanishing $\coH^r(\Afu_x,x)=\coH^r_{\cp}(\Afu_x,x)=0$ for all $r$ since the pushforward of $\pQQ^\rH_{\Afu_x}$ by the identity map is the constant Hodge module concentrated in degree zero, and hence is killed by the projector $\Pi$. The long exact sequences from Example \ref{exem:locexactsequence} then yield
\[
\coH^1(\Gm,x)\simeq\coH^0(\{0\})(-1)=\QQ(-1) \quad \text{and} \quad \coH^1_\rc(\Gm,x)\simeq\coH^0_\rc(\{0\})=\QQ(0).
\]

For $k \geq 2$, let $f\colon U=\Gm^k\to\Afu$ be the sum map $x_1+\cdots+ x_k$. Then $\Hm f_*\pQQ_U^\rH$ is the $k$-fold convolution product of $\pQQ^\rH_{\Afu}[*0]$ with itself. After applying $\Pi$ to the cohomology modules, we obtain the vanishing $\coH^r(U,f)=0$ (in other words, $\cH^{r-k}\Hm f_*\pQQ_U^\rH$ is constant) for $r\neq k$ and an isomorphism $\coH^k(U,f)\simeq\coH^1(\Gm,x)^{\otimes k}$ in $\EMHS$. Since $\MHS$ is a full tensor subcategory of~$\EMHS$ by Lemma \ref{lem:weightsMHS} and $\coH^1(\Gm,x)=\QQ(-1)$ lies in $\MHS$ by the above computation, so does $\coH^k(U,f)$ and there is an isomorphism of mixed Hodge structures
\begin{equation}\label{eqn:Gmsum}
\coH^k(\Gm^k,x_1+\cdots+x_k)\simeq\QQ(-k).
\end{equation}

Finally, with the above assumptions, the isomorphism $\cH^j\Hm f_!\pQQ_U^\rH\simeq(\Hm\bD\cH^{-j}\Hm f_*\pQQ_U^\rH)(-k)$ implies the vanishing $\coH^r_\rc(U,f)=0$ for all $r\neq k$. From \eqref{eqn:Gmsum} we know that the successive quotients of the weight filtration on $\cH^0\Hm f_*\pQQ_U^\rH$ as an object of $\MHM(\Afu)$ are all constant except for one that is isomorphic to $\Hm i_{0!}\QQ(-k)$. Dually, the successive quotients of the weight filtration on $\cH^j\Hm f_!\pQQ_U^\rH$ as an object of $\MHM(\Afu)$ are all constant except for one that is isomorphic to $\Hm i_{0!}\QQ(0)$. There is thus an isomorphism $\coH^k_\rc(\Gm^k,x_1+\cdots+x_k)\simeq\QQ(0)$.
\end{example}

\subsection{A criterion for an object of \texorpdfstring{$\EMHS$}{EMHS} to belong to \texorpdfstring{$\MHS$}{MHS}}\label{subsec:critMHS}

We now give a criterion ensuring that, for certain $f\colon U\to\Afu$ and $N^\rH$ as above, the objects $\coH^r(U,N^\rH,f)$ and $\coH^r_\rc(U,N^\rH,f)$ of~$\EMHS$ belong to the subcategory $\MHS$, that is, are usual mixed Hodge structures. This will allow us to apply Proposition \ref{prop:EMHSMHS} to identify their Hodge and irregular Hodge filtrations.

\begin{thm}[{\cf also \cite[Th.\,3.3]{Yu12}, \cite[Lem.\,6.5.3]{F-J18}}]\label{th:EMHSMHStg}
Assume that $U$ and $f$ are of the form~$U=\Afu_\tt\times V$ and $f=tg$ for some smooth quasi-projective variety $V$ and some regular function $g\colon V\to\Afu$, and let $M_V^\rH$ be an object of $\MHM(V)$.
\begin{enumerate}
\item\label{th:EMHSMHStg1}
If $N_U^\rH=\pQQ_{\Afu_\tt}^\rH\boxtimes M_V^\rH$, then the exponential mixed Hodge structures $\coH^r(U,N^\rH,f)$ and $\coH^r_\rc(U,N^\rH,f)$ belong to $\MHS$ for all $r$, and their bifiltered de~Rham fibers $(\coH^r_\dR(U,N_U\otimes \fqq{f}),\irrF^\bbullet,W_\bbullet)$ and $(\coH^r_{\dR,\rc}(U,N_U\otimes \fqq{f}),\irrF^\bbullet,W_\bbullet)$ underlie the corresponding mixed Hodge structures.

\item\label{th:EMHSlocMHS2}
If $N_U^\rH=\pQQ_{\Afu_\tt}^\rH[*0]\boxtimes M_V^\rH$, then the exponential mixed Hodge structure $\coH^r(U,N^\rH,f)$ belongs to $\MHS$ for all $r$ and its bifiltered de~Rham fiber $(\coH^r_\dR(U,N_U\otimes \fqq{f}),\irrF^\bbullet,W_\bbullet)$ underlies the corresponding mixed Hodge structure.

\item\label{th:EMHSlocMHS3}
If $N_U^\rH=\pQQ_{\Afu_\tt}^\rH[!0]\boxtimes M_V^\rH$, then the exponential mixed Hodge structure $\coH^r_\rc(U,N^\rH,f)$ belongs to $\MHS$ for all $r$ and its bifiltered de~Rham fiber $(\coH^r_{\dR,\rc}(U,N_U\otimes \fqq{f}),\irrF^\bbullet,W_\bbullet)$ underlies the corresponding mixed Hodge structure.
\end{enumerate}
\end{thm}

The last statements in \eqref{th:EMHSMHStg1}--\eqref{th:EMHSlocMHS3} follow from Proposition \ref{prop:EMHSMHS}. We are thus reduced to proving the first statements.

\begin{proof}[Proof of \eqref{th:EMHSMHStg1}]
We start with $\coH^r(U,N^\rH,f)$. Consider the divisor $D=\Afu_\tt\times g^{-1}(0)$ on $U$.
The object \hbox{$\cN_{U*}^\rH=[N_U^\rH\to N_U^\rH[*D]]$}
of $\catD^\rb(\MHM(U))$ is supported on the zero locus of $f$. For each~$r$, there exists a mixed Hodge structure $V^\rH$ such that $\cH^{r-d}\Hm f_*\cN_{U*}^\rH)=\Hm i_{0!}V^\rH$, and hence the object $\Pi_\ts(\cH^{r-d}\Hm f_*\cN_{U*}^\rH)$ of $\EMHS$ belongs to $\MHS$ for all $r$. Thanks to Proposition~\ref{prop:EMHSMHS}, it suffices to prove that the de Rham fibers $\coH^r_{\dR}(U,N_U(*D)\otimes \fqq{f})$ of $\coH^r(U, N_U^\rH[*D], f)=\Pi_\theta(N_U^\rH[*D])_*^r$ vanish in all degrees $r$. By considering the pushforward by the map $(\tt,g)\colon\Afu_\tt\times V\to\Afu_\tt\times\Afu_\taut$ we can reduce the proof to the case~$V=\Afu_\taut$ and $g=\taut$.
We then simply write $M=M_{\Afu_\taut}$,
and we are reduced to proving
\begin{equation}\label{eq:vanishlocalization}
a_{\Afu_\tt\times\Afu_\taut,\bast}\bigl((\cO_{\Afu_\tt}\boxtimes M(*0))\otimes \fqq{\tt\taut}\bigr)=0.
\end{equation}
Let $p_\tt\colon\Afu_\tt\times\Afu_\taut\to\Afu_\tt$ denote the projection. The equality $a_{\Afu_\tt\times\Afu_\taut}=a_{\Afu_\tt}\circ p_\tt$ holds. We note that the complex $p_{\tt\bast}\bigl((\cO_{\Afu_\tt}\boxtimes M(*0))\otimes \fqq{\tt\taut}\bigr)$ is nothing but the Fourier transform $\FT_\taut(M(*0))$, and in~particular~is concentrated in degree zero. Then, identifying a $\cO_{\Afu_\tt}$-module with connection with a $\CC[\tt]$\nobreakdash-module with connection, $a_{\Afu_\tt,\bast}\FT_\taut(M(*0))$ is represented by the complex
\[
\left[ \FT_\taut(M(*0))\To{\partial_\tt}\underset{\cbbullet}{\FT_\taut(M(*0))} \right]
\simeq \left[ M(*0)\To{\taut} \underset{\cbbullet}{M(*0)} \right],
\]
where $\cbbullet$ indicates the term in degree zero. Since $\tau$ acts invertibly on $M(*0)$, the left-hand side is thus quasi-isomorphic to zero.

For $\coH^r_\rc(U,N^\rH,f)$, we argue similarly, considering instead the object $\cN_{U!}^\rH= [N_U^\rH[!D]\to N_U^\rH ]$
and noting that $\Pi_\ts(\cH^{r-d}\Hm f_!\cN_{U!}^\rH)$ belongs to $\MHS$ for all $r$. It is then enough to prove the vanishing $\coH^r_{\dR,\rc}(U,N_U(!D)\otimes \fqq{f})=0$ in all degrees $r$, which reduces to
\begin{equation}\label{eq:vanishduallocalization}
a_{\Afu_\tt\times\Afu_\taut,\bexc}\bigl((\cO_{\Afu_\tt}\boxtimes M(!0))\otimes \fqq{\tt\taut}\bigr)=0
\end{equation}
by taking proper pushforward by $(\tt,g)$. It is known that the complex \hbox{$p_{\tt\bexc}\bigl((\cO_{\Afu_\tt}\boxtimes M(!0))\otimes \fqq{\tt\taut}\bigr)$} is also isomorphic to $\FT_\taut(M(!0))$, and in particular is concentrated in degree zero (see \eg \hbox{\cite[App.\,2, Prop.\,1.7]{Malgrange91}}). From the isomorphism $\FT_\taut(M(!0))\simeq\iota^\bast\bD\FT_\taut((\bD M)(*0))$, where $\iota$ is the involution $\tt\mto-\tt$, we thus get the vanishing
\[
a_{\Afu_\taut,\bexc}\iota^\bast\bD\FT_\taut((\bD M)(*0))\simeq\bD a_{\Afu_\taut,\bast}\FT_\taut((\bD M)(*0))\simeq0
\]
by the first part of the proof applied to $\bD M$.
\end{proof}

\begin{proof}[Proof of \eqref{th:EMHSlocMHS2}]
As in \eqref{th:EMHSMHStg1}, we reduce to the case where $V=\Afu_\taut$ and $g=\taut$, so that $f=\tt\taut$, and we simply denote the pushforward $\Hm g_*M_V^\rH$ by $M^\rH\in\catD^\rb(\MHM(\Afu_\taut))$. We extend the functor~$\Pi_\taut$ to an endofunctor of $\catD^\rb(\MHM(\Afu_\taut))$ that commutes with taking cohomology, and we consider the morphisms
\[
\cO_{\Afu_\tt}^\rH[*0]\boxtimes M^\rH\to\cO_{\Afu_\tt}^\rH[*0]\boxtimes\Pi_\taut(M^\rH)\from\cO_{\Afu_\tt}^\rH\boxtimes\Pi_\taut(M^\rH)
\]
in $\catD^\rb(\MHM(\Afu_\tt\times\Afu_\taut))$ and their pushforwards
\[
\Hm f_*\bigl(\cO_{\Afu_\tt}^\rH[*0]\boxtimes M^\rH\bigr)\to \Hm f_*\bigl(\cO_{\Afu_\tt}^\rH[*0]\boxtimes\Pi_\taut(M^\rH)\bigr) \from \Hm f_*\bigl(\cO_{\Afu_\tt}^\rH\boxtimes\Pi_\taut(M^\rH)\bigr)
\]
in $\catD^\rb(\MHM(\Afu_\ts))$. We will prove that, after applying the projector $\Pi_\ts$, they induce cohomology isomorphisms in $\EMHS$. Since the projections of the cohomologies of the rightmost term belong to $\MHS$ according to \eqref{th:EMHSMHStg1} (note that \eqref{eq:vanishlocalization} holds for a complex $M$ if it holds for its cohomology modules), then so will the projections of the cohomologies of the leftmost term.

For the left morphism, we note that the cohomologies of the simple complex associated with the double complex $M^\rH\to\Pi_\tau M^\rH$ are constant mixed Hodge modules. Indeed, it is enough to check that the underlying $\cD$-modules are constant (\cf\cite[Th.\,4.20]{S-Z85}), and this follows from the long exact sequence in cohomology, upon noting that an extension of constant $\cD_{\Afu_\taut}$\nobreakdash-modules is constant. We are thus reduced to proving that, for any constant mixed Hodge module~$M^\rH$ on~$\Afu_\taut$ and any $j$, the mixed Hodge module $\cH^j\Hm f_*\bigl(\cO_{\Afu_\tt}^\rH[*0]\boxtimes M^\rH\bigr)$ on~${\Afu_\ts}$ is constant. Again, it is enough to prove that the underlying $\cD$-module is constant, which amounts to proving that the de~Rham fiber \eqref{eq:deRhamfibre} of its projection to $\EMHS$ is zero. This fiber is isomorphic~to
\[
\cH^{j} a_{\Afu_\tt\times\Afu_\taut,\bast}\bigl((\cO_{\Afu_\tt}(*0)\boxtimes M)\otimes\fqq{\tt\taut}\bigr).
\]
By first projecting to $\Afu_\tt$, we find
\[
p_{\tt\bast}\bigl((\cO_{\Afu_\tt}(*0)\boxtimes M)\otimes\fqq{\tt\taut}\bigr)\simeq\cO_{\Afu_\tt}(*0)\otimes\FT_\taut(M)=0,
\]
because the Fourier transform of a constant $\cD$-module is supported at $0$.

For the right morphism, a similar argument reduces the proof to showing that the cohomology of the double complex that it defines has constant cohomology or, equivalently, that the cohomology of the double complex
\begin{align*}
\bigl\{a_{\Afu_\ts}\bigl[f_\bast\bigl(\cO_{\Afu_\tt}\boxtimes\Pi_\taut(M)\bigr)\otimes E^\ts\bigr]&\to a_{\Afu_\ts}\bigl[f_\bast\bigl(\cO_{\Afu_\tt}(*0)\boxtimes\Pi_\taut(M)\bigr)\otimes E^\ts\bigr]\bigr\}\\
\simeq \bigl\{a_{\Afu_\tt\times\Afu_\taut,\bast}\bigl((\cO_{\Afu_\tt}\boxtimes\Pi_\taut(M))\otimes\fqq{\tt\taut}\bigr)&\to a_{\Afu_\tt\times\Afu_\taut,\bast}\bigl((\cO_{\Afu_\tt}(*0)\boxtimes\Pi_\taut(M))\otimes\fqq{\tt\taut}\bigr)\bigr\}\\
\simeq \bigl\{a_{\Afu_\tt,\bast}\FT_\taut(\Pi_\taut(M))&\to a_{\Afu_\tt,\bast}\bigl(\cO_{\Afu_\tt}(*0)\otimes\FT_\taut(\Pi_\taut(M))\bigr)\bigr\}
\end{align*}
is zero. This follows from the equality $\FT_\taut(\Pi_\taut(M))=\cO_{\Afu_\tt}(*0)\otimes\FT_\taut(\Pi_\taut(M))$, which is readily checked from the definition of $\Pi_\taut$ (see \cite[Prop.\,12.3.5]{Katz90}).
\end{proof}

\begin{proof}[Proof of \eqref{th:EMHSlocMHS3}]
The dual mixed Hodge module $\Hm\bD(N^\rH_U)$ is of the form considered in \eqref{th:EMHSlocMHS2}, and there is an isomorphism $\Hm f_!N^\rH_U\simeq \Hm\bD\Hm f_*\,\Hm\bD(N^\rH_U)$ in $\catD^\rb(\MHM(\Afu_\ts))$. Therefore, each cohomology (in $\MHM(\Afu_\ts)$) of $\Hm f_!N^\rH_U$ is dual to some cohomology of $\Hm f_*\,\Hm\bD(N^\rH_U)$. We can then conclude by using the fact that the projection $\Pi_\theta(M^\rH)$ of a mixed Hodge module $M^\rH\in\MHM(\Afu_\ts)$ belongs to $\MHS$ if and only if $\Pi_\theta(\Hm\bD M^\rH)$ does. Indeed, the former property is equivalent to~$\FT_\ts(M)(*0)$ being a constant flat bundle with connection and, letting $\iota$ denote the involution $\theta\mto-\theta$, there is an isomorphism $(\FT_\ts(\bD M))(*0)\simeq\FT_\ts(\iota^\bast M)(*0)^\vee$.
\end{proof}

\begin{example}\label{exem:cK} Let us apply Theorem \ref{th:EMHSMHStg} to $M_V^\rH=\pQQ_V^\rH$. Setting $\cK=g^{-1}(0)$, the divisor $D$ in its proof is given by $D=\Afu_t\times\cK$. The vanishing
\[
\coH^r_\rc(\Afu_t\times(V\moins\cK),tg)=\coH^r(\Afu_t\times(V\moins\cK),tg)=0
\]
holds for all $r$ by \eqref{eq:vanishlocalization} and \eqref{eq:vanishduallocalization}. Then, according to the exact sequence \eqref{eqn:localizationexactsequence}, the mixed Hodge structure \hbox{$\coH^r_\rc(\Afu_t\times V,tg)$} provided by Theorem \ref{th:EMHSMHStg}\eqref{th:EMHSMHStg1} is isomorphic to $\coH^r_\rc(\Afu_t\times\cK)$. Furthermore, by the K\"unneth formula, there is an isomorphism
\[
\coH_\cp^r(\Afu\times \KM)
= \coH_\cp^2(\Afu) \otimes \coH_\cp^{r-2}(\KM)
= \coH_\cp^{r-2}(\KM)(-1),
\]
so that, finally, we obtain an isomorphism of mixed Hodge structures
\[
\coH^r_\rc(\Afu_t\times V,tg)\simeq \coH_\cp^{r-2}(\KM)(-1).
\]
On the other hand, let $i_{\KM}$ and $j_{\KM}$ be the complementary closed and open immersion attached to the divisor $\Afu\times\KM$ on $\Afu\times V$. Applying $\Hm(tg)_*$ to the triangle
\[
\Hm i_{\KM,*}\,\Hm i_{\KM}^!\pQQ^\rH_{\Afu\times V}\to\pQQ^\rH_{\Afu\times V}\to\Hm j_{\KM,*}\,\Hm j_{\KM}^*\pQQ^\rH_{\Afu\times V}\To{+1}
\]
in $\catD^\rb(\MHM(\Afu\times V))$ (\cf\cite[(4.4.1)]{MSaito87})
and noting that $(tg)_*\circ i_{\KM,*}$ is the zero map, the vanishing of $\coH^r(\Afu_t\times(V\moins\cK),tg)$ for all $r$ yields an isomorphism of mixed Hodge structures
\[
\coH^r(\Afu_t\times V,tg)\simeq\coH^r_{\Afu_t\times\KM}(\Afu_t\times V)\simeq\coH^r_{\KM}(V).
\]
If $\KM$ is smooth, the rightmost term is also isomorphic to $\coH^{r-2}(\cK)(-1)$.
\end{example}

\subsection{Computation of the Hodge and the weight filtrations}\label{subsec:computeweights}

Let $M^\rH$ be a mixed Hodge module on the affine line $\Afu_\taut$ and $N^\rH=\pQQ_{\Afu_\tt}^\rH\boxtimes M^\rH$. According to Theorem \ref{th:EMHSMHStg}\eqref{th:EMHSMHStg1}, the exponential mixed Hodge structures
\[
\coH^r(\Afu_\tt\times\Afu_\taut,N^\rH,t\taut)\quad\text{and}\quad \coH^r_\rc(\Afu_\tt\times\Afu_\taut,N^\rH,t\taut)
\]
are usual mixed Hodge structures. As already noticed in the proof of Theorem \ref{th:EMHSMHStg}, writing $a_{\Afu_\tt\times\Afu_\taut}=a_{\Afu_\tt}\circ p_\tt$ we get the following result:

\begin{lemma}
The de~Rham fibers of $\coH^{j+1}(\Afu_\tt\times\Afu_\taut,N^\rH,t\taut)$ and $\coH^{j+1}_\rc(\Afu_\tt\times\Afu_\taut,N^\rH,t\taut)$ are, respectively, $\coH_\dR^j(\Afu_\tt,\FT M)$ for $j=0,1$ and $\coH_{\dR,\rc}^j(\Afu_\tt,\FT M)$ for $j=1,2$, and zero otherwise.
\end{lemma}

\begin{nota}\label{nota:FTMH}
We denote by $\coH^j(\Afu_t,\FT M^\rH)$ and $\coH^j_\rc(\Afu_t,\FT M^\rH)$ the mixed Hodge structure $\coH^{j+1}(\Afu_\tt\times\Afu_\taut,N^\rH,t\taut)$ and $\coH^{j+1}_\rc(\Afu_\tt\times\Afu_\taut,N^\rH,t\taut)$ respectively. Their associated bifiltered de~Rham fibers are $(\coH^j_\dR(\Afu_\tt,\FT M),\irrF^\cbbullet,W_\bbullet)$ and $(\coH^j_{\dR,\rc}(\Afu_\tt,\FT M),\irrF^\cbbullet,W_\bbullet)$.
\end{nota}

\emph{A priori}, there might be a source of ambiguity in the notation for the irregular Hodge filtration, since $\FT M$ also underlies an irregular mixed Hodge module on~$\PP^1_t$ in the sense of~\cite{Bibi15}, by means of which $\coH_\dR^j(\Afu_\tt,\FT M)$ acquires an irregular Hodge filtration. However, due to the known $E_1$\nobreakdash-degeneration results for the irregular Hodge filtration, both filtrations on~$\coH_\dR^j(\Afu_\tt,\FT M)$ agree. We shall not use this property.

\begin{thm}\label{th:MHSgeneral}
Let $i_0\colon \{0\} \to \Afu_\tt$ be the inclusion. For each mixed Hodge module~$M^\rH$ on the affine line, the mixed Hodge structures
$\coH^j(\Afu_t,\FT M^\rH)$ and $\coH^j(\Hm i_0^!M^\rH)$ are isomorphic.
\end{thm}

\begin{proof}
Let us first consider the case where $M^\rH$ is supported at $0$, \ie $M^\rH$ is isomorphic to $\Hm i_{0,*}V^\rH$ for some mixed Hodge structure $V^\rH$. Set $N^\rH=\pQQ_{\Afu_\tt}^\rH\boxtimes M^\rH$ and let $f\colon \Afu_\tt \times \Afu_\taut \to \Afu_\ts$ be the function $(\tt,\taut)\mto \tt\taut=\ts$. Then $N^\rH$ is supported on $f=0$ and $\Hm f_*N^\rH[-1]\simeq\cH^{-1}\Hm f_*N^\rH$ is supported at $\ts=0$ and is isomorphic to $\Hm i_{0,*}V^\rH$ (here $i_0$ denotes the inclusion of $0$ in $\Afu_\ts$). Therefore, $\coH^1(\Afu_\tt\times\Afu_\taut,N^\rH,f)$ is isomorphic to $V^\rH$ and all other cohomologies vanish.

Now, for a general object $M^\rH$ of $\MHM(\Afu_\tau)$, we consider the exact sequence \eqref{eq:exactseqloc1} for the divisor $D=\{0\}$. According to \eqref{eq:vanishlocalization}, the vanishing
\[
\coH^{j}_{\dR}(\Afu_\tt\times\Afu_\taut,(\cO_{\Afu_\tt}\boxtimes M(*0))\otimes E^{\tt\taut})=0
\]
holds for all $j$. Setting $\coH^j(\Hm i_0^!M^\rH)=V_j^\rH$, for $j=0,1$, and $N^\rH_j=\pQQ_{\Afu_\tt}^\rH\boxtimes\Hm i_{0,*}\,V_j^\rH$, we thus get an isomorphism of mixed Hodge structures
\[
\coH^{j+1}(\Afu_\tt\times\Afu_\taut,\pQQ^\rH_{\Afu_\tt}\boxtimes M^\rH,\tt\taut)
\simeq \coH^1(\Afu_\tt\times\Afu_\taut,N_j^\rH,\tt\taut)
\]
and the right-hand side underlies $V_j^\rH$ by the first part of the proof.
\end{proof}

\begin{cor}\label{cor:weightssimple}
Assume that $M^\rH$ is a pure object of $\MHM(\Afu_\taut)$ of weight $w$ whose underlying $\cD_{\Afu_\taut}$\nobreakdash-module $M$ has no non-zero section supported at the origin (in particular, $M$ is an intermediate extension at the origin). Then $\coH^{j+1}_\dR(\Afu_t,\FT_\tau M)=0$ for $j\neq0$ and
\begin{enumerate}
\item\label{cor:weightssimple3}
$\coH^1(\Afu_t,\FT_\tau M^\rH)$ is isomorphic to the mixed Hodge structure
\[
\coker[\rN\colon \psi_{\taut,1}M^\rH\to\psi_{\taut,1}M^\rH(-1)];
\]

\item\label{cor:weightssimple2}
if $0$ is not a singular point of $M$, then
$\coH^1(\Afu_\tt,\FT_\taut M^\rH)$ is a pure Hodge structure of weight $w+1$ and the equality
$\dim\gr^p_{\irrF}\coH^1_\dR(\Afu_\tt,\FT_\taut M)=\rk\gr^{p-1}_FM$ holds;
\end{enumerate}
If, moreover, $M$ has no non-zero constant submodule, there is an exact sequence
\[
0\to M^\rH\to\Pi_\taut(M^\rH)\to M^{\prime\rH}\to0
\]
in $\MHM(\Afu)$, where $M'{}^\rH$ is a constant mixed Hodge module on $\Afu_\taut$ of weights $\geq w+1$ with
\begin{enumerate}\setcounter{enumi}{2}
\item\label{cor:weightssimple1}
$\dim\gr^p_{\irrF}\gr_\ell^W \coH^1_\dR(\Afu_\tt,\FT_\taut M')=\rk\gr^{p-1}_F\gr_{\ell-1}^WM'$, for all $\ell,p\in\ZZ$.
\end{enumerate}
\end{cor}

\begin{proof} The assumptions on $M$ imply that the semisimple $\cD_{\Afu_\tt}$\nobreakdash-module $\FT_\tau M$ has no constant submodules, and hence its de~Rham cohomology in degree zero vanishes.
Statement \eqref{cor:weightssimple3} follows from Theorem \ref{th:MHSgeneral} and Example \ref{exem:MHSgeneral}. If $0$ is not a singular point of $M$, then $\psi_{\taut,1}M^\rH=\psi_{\taut}M^\rH$ and the operator $\rN$ is identically zero, so \eqref{cor:weightssimple3} identifies $\coH^1(\Afu_\tt,\FT_\taut M^\rH)$ with the mixed Hodge structure $\psi_{\taut}M^\rH(-1)$, which is pure of weight $w+1$ and satisfies
\[
\dim\gr^p_{\irrF}\coH^1_\dR(\Afu_\tt,\FT_\taut M)=\dim\gr^p_F\psi_{\taut}M_1^\rH(-1)=\rk\gr^{p-1}_FM
\]
by the formulas recalled at the beginning of Example \ref{exem:MHSgeneral}. This proves
\eqref{cor:weightssimple2}. The statement about weights in the last point follows from the same argument we used in the proof of Proposition \ref{prop:HwtM}: were the inclusion $M^\rH\subset W_w\Pi_\taut(M^\rH)$ not an equality, $\Pi_\taut(M)$ would have a non-zero constant submodule, which contradicts the vanishing of its cohomology. Finally, \eqref{cor:weightssimple1} is obtained by applying Example \ref{exem:MHSgeneral} once again.
\end{proof}

\backmatter
\bibliographystyle{amsplain}
\bibliography{kloosterman-revised10}
\end{document}